\documentclass[12pt]{amsart}
\usepackage{amsmath, amsthm, amscd, amssymb, amsfonts, latexsym, hyperref}
\usepackage{fullpage}
\usepackage{bm}
\usepackage[all]{xy}
\usepackage{tikz}
\usepackage{fancybox}
\usepackage[justification=centering]{caption}
\usepackage{bbm}
\usepackage{listings}
\hypersetup{colorlinks,
	linkcolor=green!70!black,
	citecolor=blue!70!black,
	urlcolor=red!50!black
}

    \def\cD{\mathcal D}  \def\cF{\mathcal F}        \def\cN{\mathcal N}          

\def\fm{\mathfrak{m}}
\def\fp{\mathfrak{p}}

\def\ff{\mathfrak{f}}

\def\fa{\mathfrak{a}}
\def\fp{\mathfrak{p}}

\def\Z{{\mathbb Z}} \def\R{{\mathbb R}} \def\F{{\mathbb F}} \def\N{{\mathrm N}} \def\C{{\mathbb C}} \def\Q{{\mathbb Q}}

\newcommand{\dd}{\mathrm{d}}
\newcommand{\p}{\mathfrak{p}}

\newcommand{\n}{M_{d,k}}

\newcommand{\kommentar}[1]{}

\newcommand{\legendre}[2]{\left(\frac{#1}{#2}\right)}

\newcommand{\bquartic}[2]{\overline{\chi}_{(#2)}(#1)}

\DeclareMathOperator{\tr}{Tr}

\DeclareMathOperator{\re}{Re}
\DeclareMathOperator{\im}{Im}

\renewcommand{\mod}[1]{\,{\rm mod}\,#1}
\renewcommand{\bmod}[1]{\,(\mathrm{mod}\,#1)}

\newtheorem{lem}{Lemma}[section]
\newtheorem{prop}[lem]{Proposition}
\newtheorem{thm}[lem]{Theorem}
\newtheorem{defn}[lem]{Definition}
\newtheorem{cor}[lem]{Corollary}

\theoremstyle{definition}

\newtheorem{rem}[lem]{Remark}

\newcommand{\phat}{\widehat{\phi}}

\definecolor{pink}{rgb}{1,.2,.6}
\definecolor{orange}{rgb}{0.7,0.3,0}
\definecolor{blue}{rgb}{.2,.6,.75}
\definecolor{green}{rgb}{.4,.7,.4}
\definecolor{purple}{RGB}{127,0,255}

\begin{document}

\title{One-level densities in families of Gr\"ossencharakters associated to CM elliptic curves}
\author[C. David]{Chantal David}
\address{Chantal David: Department of Mathematics and Statistics, Concordia University, 1455 de Maisonneuve West, Montréal, QC H3G 1M8, Canada}
\email{chantal.david@concordia.ca}

\author[L. Devin]{Lucile Devin}
\address{Lucile Devin: Univ. Littoral C\^ote d'Opale, UR 2597
	LMPA, Laboratoire de Math\'ematiques Pures et Appliqu\'ees Joseph Liouville,
	F-62100 Calais, France
}
\address{CNRS -- Université de Montréal CRM - CNRS}
\email{lucile.devin@univ-littoral.fr}

\author[E. Waxman]{Ezra Waxman}
\address{Ezra Waxman: Department of Mathematics, University of Haifa, 199 Aba Khoushy Ave., Mt. Carmel, Haifa 3498838, Israel}
\address{Unit of Mathematics, Afeka — The Academic College of Engineering in Tel Aviv, Mivtsa Kadesh St 38, Tel Aviv-Yafo 6998812, Israel}
\email{ezrawaxman@gmail.com}

\date{\today}
\numberwithin{equation}{section}
\maketitle
\begin{abstract} We study the low-lying zeros of a family of $L$-functions attached to the CM elliptic curve $E_d \;:\; y^2 = x^3 - dx$, for each odd and square-free integer $d$. Specifically, upon writing the $L$-function of $E_d$ as $L(s-\frac12, \xi_d)$ for the appropriate Gr\"ossencharakter $\xi_d$ of conductor $\mathfrak{f}_d$, we consider the collection $\cF_d$ of $L$-functions attached to $\xi_{d,k}$, $k \geq 1$, where for each integer $k$, $\xi_{d, k}$ denotes the primitive character  inducing  $\xi_d^k$. We observe that $25\%$ of the $L$-functions in $\cF_d$ have negative root number. $\cF_d$ is thus not one of the \emph{essentially homogeneous} families of the Universality Conjecture of Sarnak, Shin and Templier \cite{SST},  with  unitary, symplectic or orthogonal (odd or even) symmetry type.
By computing the one-level density in the family of $L$-functions in $\cF_{d}$ with conductor at most $K^2 \N (\mathfrak{f}_d)$, we find  that $\cF_d$ naturally decomposes into subfamilies: more specifically, a collection of symplectic ($L(s, \xi_{d,k})$ for $k \equiv \alpha \bmod 8$, $\alpha$ even) and orthogonal ($L(s, \xi_{d,k})$ for $k \equiv \alpha \bmod 8$, $\alpha$ odd) subfamilies.  For each such subfamily, we moreover compute explicit lower order terms in decreasing powers of $\log (K^2 \N (\mathfrak{f}_d) )$.
\end{abstract}

\section{Introduction}

Let $d\in \mathbb Z$ be a fixed odd square-free integer, and let $E_d$ denote the complex multiplication (CM) elliptic curve with affine equation $E_d \;:\; y^2 = x^3 - d x$.
The  $L$-function of~$E_d$ is
\begin{align*} 
L(s, E_d) &:= \prod_{p \nmid 2d} \left( 1 - \frac{a_p(E_d)}{p^s} + \frac{1}{p^{2s-1}} \right)^{-1}, \quad \mathrm{Re}(s) > \tfrac{3}{2},
\end{align*}
where for each prime $p \nmid 2d$, one writes $a_p(E_d):= p + 1 - \# E_d(\F_p) $.
Because of complex multiplication by $\mathbb Z[i]$, this $L$-function may be written in terms of the $L$-function of a Gr\"ossencharakter on $\Z[i]$.  More specifically, we have $$L(s, E_d) = L(s-\tfrac12, \xi_d),$$ where  $\xi_d$ is the Gr\"ossencharakter defined in  \eqref{def-xi}. For each fixed $d$, we consider the collection of $L$-functions 
$$
\cF_d := \{ L(s, \xi_{d,k}) \;:\; k \geq 1 \},
$$
where $\xi_{d, k}$ denotes the primitive character inducing the power $\xi_d^k$ of the Gr\"ossencharakter~$\xi_d$. 
For the particular case $d=1$, $\cF_d$ and its subfamilies have been used to study fine scale statistics of Gaussian primes in sectors  \cite{WaxRatio,De2023+,HarmanLewis, RudWax}.

In this paper we study, for fixed odd square-free $d\in\mathbb Z$, the low-lying zeros (i.e. the zeros close to the central point $s=\tfrac12$) of the~$L$-functions in~$\cF_d$. 
The Katz--Sarnak Density Conjecture \cite{KSbook,KS,SST} states that
the distribution of low-lying zeros may be predicted by the symmetry type of the family in question.
Precisely, the conjecture states that the zero distribution corresponds to the distribution of eigenvalues close to $1$ of random matrices in an appropriate classical compact group (unitary, symplectic or orthogonal). 
There is a lengthy history of computing one-level densities, and more generally $n$-level densities, for various $L$-function families including  those of Dirichlet characters \cite{D-G,DPR,HR,Miller-et-al,O-S}, Hecke characters \cite{G-Z,Holm2023,Wa2021}, elliptic curves \cite{HB,M1, Young}, and modular forms \cite{DFS1,ILS,RR}.  We refer the reader to \cite[\S 2]{SST} for an extensive survey of existing results, including results for Artin $L$-functions, $L$-functions of Maass forms, and $L$-functions of automorphic representations.

The symmetry type  is also related to the distribution of the sign of the functional equation.
Let $W(\xi_{d, k})= \pm 1$ denote the root number of $L(s, \xi_{d,k})$ (i.e. the sign of its functional equation).
As demonstrated in Lemma~\ref{lemma-sign}, one has, for any fixed odd square-free $d \in \Z$, that 
\begin{align} \label{average-RN1}
\lim_{K \rightarrow \infty} \frac{\#\{1 \leq k \leq K: W(\xi_{d,k}) = -1\}}{K} = \frac{1}{4}.
\end{align}
Conjecturally, the proportion~\eqref{average-RN1} should equal the average analytic rank of $\cF_{d}$.  As the sign of the functional equation is not constant thorough~$\cF_d$, it is not an ``essentially homogeneous'' family in the language of Sarnak, Shin and Templier \cite{SST}. We further observe that the  orthogonal symmetry type --- which is the average of the  symmetry types $SO(\text{even})$ and $SO(\text{odd})$, and has average rank $\tfrac12$ --- is also excluded.
Upon studying congruence classes of $k$ modulo $8$ (see Theorem~\ref{thm-main}) we show that $\cF_d$ ultimately breaks down into natural subfamilies with different symmetry types. 

In order to state our main result, let us first introduce some notation.
Let $\phi$ be an even Schwartz function such that $
\widehat{\phi}(s) := \int_{-\infty}^\infty \phi(t) e^{-2 \pi i t s} \, \dd t$, the Fourier transform of $\phi$, is compactly supported.
For each fixed odd square-free $d\in \mathbb Z $, we wish to understand the behaviour of the properly normalized low-lying zeros across the subfamilies of $\cF_d$. To this end, we define as in \cite[(54)]{KS},
\begin{equation}\label{def-local_OLD}
\cD(\phi, \xi_{d,k}) := \sum_{L(\frac12+i\gamma, \xi_{d,k}) = 0} \phi \left(  \frac{\gamma \log{\left(k^2 \N(\mathfrak{f}_{d,k})\right)}}{2\pi} \right),
\end{equation}
where $\N(\mathfrak{f}_{d,k})$ is the norm of the conductor of $\xi_{d, k}$, and where the normalisation factor is chosen upon noting that the analytic conductor of $L(s, \xi_{d, k})$ is 
asymptotic to $k^2 \N(\mathfrak{f}_{d,k})$ (for further details concerning this chosen normalisation see the beginning of Section~\ref{section compute D_k} below).  

As the root number $W(\xi_{d,k})$ depends on the congruence class of $k$ modulo $8$ (see Lemma~\ref{lemma-sign}), it is natural to study the subfamilies of $\cF_d$  according to the congruence class of $k$ modulo $8$. For any $\alpha \in \Z/8\Z$,   
we thus define $$\cF_d^{\alpha} := \{ L(s, \xi_{d,k})\;:\; k \geq 1, \; k \equiv \alpha \bmod 8\}.$$  The one-level density of the family $\cF_d^{\alpha}$ is then given for any $K \in \mathbb{N}$, by
\begin{align}\label{One_level_K}
\cD(K; \phi, \cF_d^{\alpha}) &:= \frac{8}{K} \sum_{\substack{1 \leq k \leq K\\k \equiv \alpha \bmod 8}} \cD(\phi, \xi_{d,k}).
\end{align} 
 Our main theorem is as follows.
\begin{thm} \label{thm-main} Let $d$ be an odd square-free integer, $\phi$ be an even Schwartz function with $\mathrm{supp}(\widehat{\phi}) \subset (-1, 1)$ and fix $J \in \mathbb{N}$ and $\alpha \in \{ 1, 2, \dots, 8 \}$.  
Let $M_{d, \alpha} := \N( \ff_{d,\alpha})^\frac12$, where $\ff_{d,\alpha}$ is the conductor of $\xi_{d, \alpha}$ given by~\eqref{conductor}.
When $K \rightarrow \infty$ and $\alpha$ is even,
\[\cD(K; \phi, \cF_d^{\alpha})= \widehat{\phi}(0)-
\frac{1}{2} \int_{\R}\widehat{\phi}(u) \dd u
+\sum_{m=1}^{J}\frac{C_{m}(d,\alpha, \phi)}{(\log(KM_{d, \alpha}))^{m}}+ O_{J}\left(\frac{1 + \log\lvert d\rvert}{(\log(K M_{d, \alpha}))^{J+1}}\right),\]
and when $K \rightarrow \infty$  and $\alpha$ is odd,
\begin{align*}
 \cD(K; \phi, \cF_d^{\alpha}) 
&= \widehat{\phi}(0)+
\frac{1}{2} \int_{\R}\widehat{\phi}(u) \dd u
+\sum_{m=1}^{J}\frac{C_{m}(d,\alpha, \phi)}{(\log(K M_{d,\alpha}))^{m}}
+ O_{J}\left(\frac{1+\log|d|}{(\log(K M_{d,\alpha}))^{J+1}}\right),
\end{align*}
where the $C_{m}(d,\alpha, \phi)$ are given by \eqref{define-Cm}.
\end{thm}

\begin{rem}
Since $\mathrm{supp}(\widehat{\phi}) \subset (-1, 1)$, we have
\begin{equation} \label{maybe} \pm \frac12  \int_{\R}\widehat{\phi}(u) \dd u =  \pm \frac12 \int_{-1}^1 \widehat{\phi}(u) \dd u,
\end{equation}
which is the expected term from Katz--Sarnak heuristics (see \eqref{W_hat_options} and \eqref{def eta}). Note that the equality in~\eqref{maybe} no longer holds when $\widehat{\phi}$ is supported beyond $[-1,1]$. This explains why $\mathrm{supp}(\widehat{\phi}) \subset (-1, 1)$ is a difficult barrier to break: 
one needs to identify an additional term of the exact same size in order to precisely cancel this excess contribution.
\end{rem}

In \cite{Wa2021}, the third author computed the one-level density of the subfamily $k\equiv 0 \bmod 4$ for the particular case $d=1$, for test functions $\phi$ such that $\widehat{\phi}\subset (-1,1)$, up to an error term of size $O(1/(\log K)^{2})$. In such a case, the root number is identically 1 for each element in the family, and the one-level density has symplectic symmetry.  Our work generalizes this in three distinct directions.  Specifically, we allow $d \in \mathbb{Z}$ to be any odd and square-free integer; we consider $k \equiv \alpha\bmod 8$ for any $\alpha \in \Z/8\Z$; and we explicitly compute all lower order terms up to an arbitrary negative power of $\log{(K M_{d,\alpha})}$. 
 In contrast to the main term, which follows the Katz--Sarnak prediction, the lower-order terms have no universal behaviour, as they contain features which depend on the particular family in question. We refer to 
\cite{DFS1, FPS1, FPS2,RicottaRoyer LowerOrder} for comparisons.

An additional interesting feature of this work is a very explicit description of the lower-order contributions in terms of generalized Euler constants. 
Such a description enables a concrete understanding of the speed and direction of convergence to the conjectured distribution (see Section~\ref{alternative_lower_constants} and Appendix~\ref{appendix computation} and in particular Remark~\ref{Rem Appendix}).

Note that Theorem~\ref{thm-main} determines the symmetry type of the families $\cF_d^{\alpha}$. Indeed, by the Katz--Sarnak Density Conjecture and the Sarnak--Shin--Templier Universality Conjecture \cite[Conj. 2]{SST}, we expect that for an \emph{essentially homogeneous} family $\cF$, there exists some compact group $G \in \{U,Sp,SO(\text{even}),SO(\text{odd})\}$ (depending on $\cF$) such that
\[\lim_{K \rightarrow \infty}\cD(K; \phi, \cF) = \int_{\R}W_G(x)\phi(x)\dd x = \int_{\R}\widehat{W}_G(t)\widehat{\phi}(t)\dd t,\]
where
\begin{align} \label{W_options}
W_G(x) := \begin{cases} 1 & \mbox{if $G=U$} \\ 1 - \frac{\sin(2 \pi x)}{2 \pi x} & \mbox{if $G=Sp$} \\
 1 + \frac{\sin(2 \pi x)}{2 \pi x} & \mbox{if $G=SO(\text{even})$} \\
  1 + \delta_0(x) - \frac{\sin(2 \pi x)}{2 \pi x} & \mbox{if $G=SO(\text{odd})$},
\end{cases}
\end{align}
and $\delta_0$ is the Dirac function.
The Fourier transform of each such density is then given by
\begin{align}\label{W_hat_options}
\widehat{W}_G(t) = \begin{cases} \delta_0(t) & \mbox{if $G=U$} \\ \delta_0(t)  - \frac12  \eta(t) & \mbox{if $G=Sp$} \\
\delta_0(t)  + \frac12  \eta(t)  & \mbox{if $G=SO(\text{even})$} \\
\delta_0(t)  +1- \frac12  \eta(t) & \mbox{if $G=SO(\text{odd})$} 
 \end{cases}
\end{align}
where
\begin{equation}\label{def eta}
\eta(t) := \begin{cases} 1 & \text{ if } |t|< 1, \\ \frac12 & \text{ if } |t|=1, \\ 0 & \text{ if }|t|> 1.\end{cases}	
\end{equation}

By Theorem \ref{thm-main} and \eqref{W_hat_options}, $\cF_d^\alpha$ has symplectic symmetry when $\alpha$ is even, and an orthogonal symmetry type when $\alpha$ is odd. 
In the latter case, the two orthogonal symmetry types cannot be distinguished when $\mathrm{supp}(\widehat{\phi}) \subset (-1, 1)$.  
By differentiating the functional equation (see~\eqref{FE}), we moreover observe that  $W(\xi_{d,k}) = -1$ if and only if $\mbox{ord}_{s=\frac12} L(s, \xi_{d,k})$ is odd.
Thus, upon defining
\begin{equation}\label{def S+-}
S_{\pm}(d):= \{ \alpha \in (\Z/8\Z)^\times \;:\; 
W(\xi_{d,k}) = \pm 1 \text{ when } k \equiv \alpha \bmod 8 \},	
\end{equation}
we expect (by heuristically taking $\phi$ to be the indicator function of the set $\lbrace 0 \rbrace$) that
$G=SO(\text{even})$ when $\alpha \in S_{+}(d)$ and $G=SO(\text{odd})$ when $\alpha \in S_{-}(d)$.  The sets $S_{\pm}(d)$ are explicitly computed in~\eqref{set s- computed}.

Note moreover that when $k$ is odd, it follows from \cite[Thm. 12.5]{Iwaniec} that $L(s,\xi_{d,k}) = L(s,f)$, where $f$ is a cusp form of level $4\N(\ff_{d})$, weight $k+1$, and nebentypus equal to the Dirichlet character modulo $4\N(\ff_{d})$ given explicitly by $n \mapsto \legendre{-4}{n} \xi_{d, k}((n))$, where $\legendre{-4}{\cdot}$ is the Kronecker symbol. The nebentypus is then the trivial character modulo $4\N(\ff_{d})$ (see \eqref{conductor_2} and \eqref{chi_at_integers}).  The family $\cF^{\alpha}_{d}$ may thus be viewed as a thin subfamily of $L$-functions associated to cusp forms of fixed level, trivial nebentypus, and varying even weight.  This larger family has an orthogonal symmetry.  This was proven in \cite[Thm.~1.3]{ILS} by computing the corresponding one-level density for $\widehat{\phi}$  supported  in~$(-2,2)$.

As an application of Theorem \ref{thm-main}, we obtain a proportion for non-vanishing at the central point in each family $\cF_d^\alpha$.  This partially answers a question in \cite{De2023+}, posed in the cases $d =\pm1$. The proportion depends upon the symmetry type of the family.

\begin{cor} \label{non-vanishing} Let $d$ be an odd square-free integer and assume the Riemann Hypothesis for the functions in $\mathcal{F}_d$. \\
When $\alpha \in \Z/8\Z$ is even, the proportion of non-vanishing in $\cF_d^{\alpha}$ is at least
\begin{align*}
	\liminf_{K \rightarrow \infty}  \frac{8}{K} \; \# \left\{ 1 \leq k \leq K, \, k \equiv \alpha \bmod 8 \;:\; L( \tfrac12, \xi_{d,k}) \neq 0 \right\} \geq
 75 \%.
\end{align*}
In the case $\alpha$ is odd and $\alpha \in S_{+}(d)$, we have
\begin{align*}
	\liminf_{K \rightarrow \infty}  \frac{8}{K} \; \# \left\{ 1 \leq k \leq K, \, k \equiv \alpha \bmod 8 \;:\; L( \tfrac12, \xi_{d,k}) \neq 0 \right\} \geq
 25 \%. 
\end{align*}
Finally, if $\alpha$ is odd and $\alpha \in S_{-}(d)$, each $L(s, \xi_{d, k})$ vanishes at $s = \tfrac12$, and we have 
\begin{align*}
	\liminf_{K \rightarrow \infty}  \frac{8}{K} \; \# \big\{ 1 \leq k \leq K, \, k \equiv \alpha \bmod 8 \;:\; \mathrm{ord}_{s=\tfrac12} L(s, \xi_{d,k}) = 1\big\} \geq 75 \%.
\end{align*}

\end{cor}

The structure of this paper is as follows. In Section \ref{background} we define the {Gr\"ossencharakters}~$\xi_{d, k}$, and explicitly compute the root numbers $W(\xi_{d, k})$ for all odd square-free $d \in 
\Z$ and $k \geq 1$.  This leads to the observation (equation \eqref{average-RN}) that $W(\xi_{d,k})=-1$ for $25\%$ of the $k \geq 1$, and depends only on  $k \bmod 8$ (for any fixed value of $d$). 
In Section~\ref{section compute D_k}, we then employ the explicit formula to compute the contribution of the $\Gamma$-factors, of the ramified primes, and of the inert primes towards the one-level density computation, for each $\xi_{d,k}$.  This yields the symmetry type of each family $\cF_d^\alpha$ (see Proposition \ref{local_OLD}), provided one can show that the averaged contribution of the split primes (i.e. of $U_{\rm split}(\phi, d, k)$ over $1 \leq k \leq K$ with $k\equiv \alpha \bmod 8$) is small.   In Section~\ref{alternative_lower_constants} we then offer an alternative expression for $c_{j,\textnormal{inert}}(k)$, the contributions of the inert primes to $C_{m}(d,\alpha, \phi)$, helpful in order to compute such constants explicitly in Appendix~\ref{appendix computation}.  Next, in Section~\ref{section_small_support} we provide a simple argument showing that 
Theorem~\ref{thm-main} holds for $\text{supp}(\widehat\phi) \subset (-\frac12, \frac12)$, and in Section \ref{extended-range} we  extend this result to the range $\text{supp}(\widehat\phi) \subset(-1,1)$. Finally, we prove Corollary \ref{non-vanishing} in Section~\ref{section-non-vanishing}.

\section*{Acknowledgements} 
We thank the referee for a very careful read of the manuscript as well as for meticulous comments which have helped to greatly improve the exposition.  We also thank Anders S\"odergren and Alessandro Languasco for helpful discussions. This paper was written during visits of the authors at Concordia University, Universit\'e du Littoral C\^ote d'Opale and Institut Henri-Poincar\'e (Research in Paris program). We thank those institutions for financial support and hospitality.
CD was supported by the Natural Sciences and Engineering Research Council of Canada [DG-155635-2019] and by the Fonds de recherche du Qu\'ebec - Nature et technologies [Projet de recherche en \'equipe 300951]. LD was supported by the PEPS JCJC 2023 program of the INSMI (CNRS). EW was supported by a Zuckerman Post Doctoral Fellowship and by the Israel Science Foundation (Grant No. 1881/20).

\section{Background}\label{background}
\subsection{Quartic Residue Symbol}
We refer to \cite[Ch. 9]{IreRos} for the material of this section.  To begin, we define the \textit{norm} of a Gaussian integer $\alpha \in \Z[i]$ to be $\N(\alpha):=\alpha \overline{\alpha}$.  To any Gaussian prime $\pi \in
\Z[i]$ such that $\pi\nmid 2$, we define the \textit{quartic residue symbol} modulo $\pi$ to be the quartic character
$\chi_{(\pi)}: \left( \Z[i]/(\pi) \right)^\times \longrightarrow \{ \pm 1, \pm i \}
$
such that
\begin{align}
 \label{def-quartic}
\chi_{(\pi)}(\alpha) := \left( \frac{\alpha}{\pi} \right)_4 \equiv \alpha^{(\N(\pi)-1)/4} \bmod \pi,
\end{align}
and extended to $\Z[i]$ by $\chi_{(\pi)}(\alpha)  =0$ when $(\alpha,\pi) \neq 1$.
Consider a non-unit $\beta \in \Z[i]$ with prime decomposition given by $\beta= \pi_1^{e_1} \dots \pi_s^{e_s}$.  When $(\beta, 2)=1$, we may extend by multiplicativity to then define, for any $\alpha \in \Z[i]$,
$$
\chi_{(\beta)}(\alpha) := \left( \frac{\alpha}{\beta} \right)_4 = \prod_{i=1}^s \left( \frac{\alpha}{\pi_i} \right)_4^{e_i}.
$$
In particular, we note \cite[Prop. 9.8.5]{IreRos} that for odd $d \in \mathbb{N}_{\geq 3}$ and $n \in \Z$ such that $(n,d) = 1$, we have
 \begin{equation}\label{chi_at_integers}
\chi_{(d)}(n) =1.
\end{equation}

Recall that the \textit{conductor} of a character $\chi: \left(\Z[i]/\mathfrak{m}\right)^{\times}\longrightarrow S^{1}$ refers to the smallest divisor~$\mathfrak{f}|\mathfrak{m}$ such that $\chi$ factors through 
 $(\Z[i]/\mathfrak{f})^{\times}$ via the projection $(\Z[i]/\fm)^{\times} \twoheadrightarrow (\Z[i]/\mathfrak{f})^{\times}.$  If the conductor of $\chi$ is $\mathfrak{m}$, then $\chi$ is referred to as a \textit{primitive} character mod $\mathfrak{m}$.  In particular, we note that if $\beta \in \Z[i]$ as above is square-free, then 
$\chi_{(\beta)}$ is a primitive 
quartic character with conductor $(\beta)$.

A non-unit Gaussian integer $\alpha \in \Z[i]$ is said to be \textit{primary} if $\alpha \equiv 1 \bmod {(2+2i)}$.  As in \cite[Lem. 7, Sect. 9.7]{IreRos}, we find that any proper ideal $(\alpha) \in \Z[i]$ coprime to $(2)$ has precisely one primary generator $\boldsymbol{\alpha} \in \{\pm \alpha, \pm i \alpha\}$. In what follows, we denote the primary generator $\boldsymbol{\alpha}$ of an ideal $(\alpha)$ by a bold letter.  We cite from \cite[Ch. 9, Thm. 2]{IreRos} the following reciprocity law:

\begin{prop}[Law of Quartic Reciprocity] \label{reciprocity} Let $\boldsymbol{\alpha}, \boldsymbol{\beta} \in \Z[i]$ be primary such that $(\boldsymbol{\alpha},\boldsymbol{\beta})=1$. Then
\begin{align}
\left( \frac{\boldsymbol{\alpha}}{\boldsymbol{\beta}} \right)_4= \left( \frac{\boldsymbol{\beta}}{\boldsymbol{\alpha}} \right)_4 (-1)^{\frac{\N(\boldsymbol{\alpha})-1}{4} \; \frac{\N(\boldsymbol{\beta})-1}{4}}.\end{align}
\end{prop}

We will also use quadratic and quartic characters modulo even ideals. More precisely, we introduce the character $\chi_{(2)}:\left(\Z[i]/(2)\right)^{\times} \longrightarrow \{\pm 1\}$ given by
\begin{align} \label{def-2}
\quad  \chi_{(2)}(\alpha) &:= \begin{cases} 1 & \alpha \equiv 1 \bmod 2 \\ -1 & \alpha \equiv i \bmod 2, \end{cases}
\end{align}
the character $\chi_{(2+2i)}:\left(\Z[i]/(2+2i)\right)^{\times} \longrightarrow \{\pm 1, \pm i\}$ given by
\begin{align}
 \label{def-2+2i}
\chi_{(2+2i)}(\alpha)&:= \begin{cases} 1 & \alpha \equiv 1 \bmod {2+2i} \\
-1 & \alpha \equiv  -1 \bmod {2+2i} \\
-i & \alpha \equiv i \bmod {2+2i} \\
i & \alpha \equiv -i  \bmod {2+2i},
\end{cases}
\end{align}
and the character $\chi_{(4)}:\left(\Z[i]/(4)\right)^{\times} \longrightarrow \{\pm 1, \pm i\}$ given by
\begin{align}\label{def-4}
\chi_{(4)}(\alpha) &:= \begin{cases}
1 & \alpha \equiv 1, -(3+2i) \bmod 4\\
-1 & \alpha \equiv -1, (3+2i) \bmod 4\\
i & \alpha \equiv i, -i(3+2i) \bmod 4\\
-i & \alpha \equiv -i, i(3+2i) \bmod 4.
\end{cases}
\end{align}

Note that  $\chi_{(2)}, \chi_{(2+2i)}$ and $\chi_{(4)}$ are primitive characters on $\Z[i]$, with conductors given by $(2)$, $(2+2i)$, and $(4)$, respectively.  We further remark that for $\alpha \in \Z[i]$, with $(\alpha,2) = 1$, the primary generator of $(\alpha)$ is given by
\begin{equation}\label{boldalpha}
\boldsymbol{\alpha} =  {\chi}_{(2+2i)}(\alpha)  \; \alpha.
\end{equation}
  Finally, by \cite[Thm. 6.9]{Lemmermeyer} (see also \cite[Ch. 9, Ex. $32-33$]{IreRos}), we note that for odd $d \in \mathbb{N}_{\geq 3}$,
 \begin{equation}\label{eisenstein_reciprocity}
\chi_{(d)}(1+i) =
i^{(\boldsymbol{d}-1)/{4}},
\end{equation}
where $\boldsymbol{d}$ denotes the primary generator of $(d)$.

\subsection{$L$-functions of Elliptic Curves}

Much of the following material on elliptic curves may be found in \cite[Ch. 18]{IreRos}.
Let $E$ be an elliptic curve over $\Q$ with conductor~$N_E$. 
For~$p \nmid N_E$, let
$a_{p}(E) := p+1-\# E(\F_{p})$.  By the Hasse bound, $a_{p}(E) \leq 2 \sqrt{p}$.  The \textit{$L$-function} attached to~$E$ is then defined by
\begin{equation}
L(s,E) := \prod_{p}L_p(s,E)^{-1}, \quad \quad \mathrm{Re}(s) > \frac{3}{2},
\end{equation}
where
\begin{align*}
L_p(s,E) &:= \left\{
\begin{array}{l l l}
(1-a_{p}(E)p^{-s}+p^{1-2s}) & \text{ if } p \textnormal{ has good reduction at } p \\
(1-p^{-s}) & \text{ if } p \textnormal{ has split multiplicative reduction at } p \\
(1+p^{-s}) & \text{ if } p \textnormal{ has non-split multiplicative reduction at } p \\
1 & \text{ if } p \textnormal{ has additive reduction at } p.
\end{array} \right.
\end{align*}
In the case $E_d: y^2 = x^3 - dx$, the conductor is $N_{E_d} = 2^{\ell} d^2$, where $\ell = 5$ (resp. $\ell=6$) if $d \equiv 1 \bmod 4$ (resp. $d \equiv 3 \bmod 4$), and 
\begin{align} \label{L-fct}
L(s, E_d) &= \prod_{p \nmid 2d} \left( 1 - \frac{a_p(E_d)}{p^s} + \frac{1}{p^{2s-1}} \right)^{-1}, \quad \mathrm{Re}(s) > \frac{3}{2}.
\end{align} 
The following result will enable us to express $L(s, E_d)$ as the $L$-function of a Gr\"ossencharakter over $\Q(i)$.
\begin{prop}\label{cmcurvesheckethm}
	Let $d$ be an odd square-free integer.
	For $p \equiv 1 \bmod{4}$, $p\nmid d$, write $p \mathbb{Z}[i] = \fp \overline{\fp}$ and let $\boldsymbol{\pi}_{\fp}$ be the primary generator of $\fp$.
	 Then on $\textrm{Re}(s)> 3/2$,  we have  
	\begin{align}  \label{L-fct-E}
		L(s,E_{d})
		& = \prod_{\substack{p \equiv 3 \bmod4\\ p \nmid d}}(1+p^{-2s+1})^{-1}\prod_{\substack{p \equiv 1 \bmod 4 \\(p) = \fp \overline{\fp}  \\ p \nmid d  }}
		\bigg(1-\frac{\overline{\chi_{\fp}(d)}\boldsymbol{\pi}_{\fp}}{p^{s}}\bigg)^{-1}
		\bigg(1-\frac{\chi_{\fp}(d)\overline{\boldsymbol{\pi}}_{\fp}}{p^{s}}\bigg)^{-1}.
	\end{align}
\end{prop} 

\begin{proof}
	Let us first remark that the contribution of the primes $p\equiv 1 \bmod 4$ is well defined because  $\overline{\boldsymbol{\pi}}_\fp$ is the primary generator of $\overline{\fp}$.  Indeed, since $\boldsymbol{\pi}_\fp$ is primary, there exists some $a+bi \in \Z[i]$ such that
	\begin{align*}
		\boldsymbol{\pi}_{\fp} = 1+(a+bi)(2+2i),
	\end{align*}
	from which it follows that
	\begin{align*}
		\overline{\boldsymbol{\pi}}_{\fp} &= \overline{1+(a+bi)(2+2i)}= 1+(-b-ai)(2+2i)\equiv 1 \bmod{(2+2i)} .
	\end{align*}
Moreover, since $d\in \mathbb Z$, we deduce from~\eqref{chi_at_integers} that  $\chi_{\overline\p}(d) \chi_{\p}(d) =  \chi_{(p)}(d) =1$, that is $\chi_{\overline\p}(d) = \overline{\chi_{\p}(d)}$.
	
	The proof then follows directly from \cite[Thm. 5 in Ch. 18]{IreRos} which ensures that for $p \nmid 2d$ and $p \equiv 3 \bmod 4$, we have $\# E_{d}(\F_{p}) = p+1$,
	while for $p \nmid 2d$ and $p \equiv 1 \bmod 4$  we have
	\[\# E_{d}(\F_{p}) = p+1 - a_p(E_d) = p + 1 - \overline{\chi_{\fp}(d)}\boldsymbol{{\pi}}_{\fp} - \chi_{\fp}(d)   \overline{\boldsymbol{\pi}}_{\fp}.\]
	The definition~\eqref{L-fct} concludes the proof.
\end{proof}
 We will now describe a Gr\"ossencharakter (or Hecke character) whose $L$-function is~$L(s,E_{d})$.

\subsection{Gr\"ossencharakters}
Fix a non-zero integral ideal $\mathfrak{m}$ of $\Q(i)$, and let $J^{\mathfrak{m}}$ denote the group of fractional ideals in $\Q(i)$ coprime to $\mathfrak{m}$.  As in {\cite[VII.6.1]{Neukirch}}, we say that $\xi : J^\fm \rightarrow S^1$ is a \textit{Gr\"ossencharakter} modulo $\fm$ on $\Q(i)$, if there exists a pair of characters 
\[\xi_{\mathrm{fin}} :(\Z[i]/\fm)^{\times} \longrightarrow S^1, \quad \xi_\infty : \C^\times  \longrightarrow S^1\]
such that 
\begin{equation}\label{Hecke character def}
\xi((\alpha)) = \xi_{\mathrm{fin}} (\alpha) \xi_\infty (\alpha),
\end{equation}
 for every $\alpha \in \Z[i]$ coprime to $\fm$.  $\xi_{\mathrm{fin}}$ is referred to as the \textit{finite} component of $\xi$, while $\xi_{\infty}$ is referred to as the \textit{infinite} component of $\xi$.  The infinite component~$\xi_{\infty}$ may moreover be written in the form
\[\xi_{\infty}(\alpha) = \left(\frac{\alpha}{|\alpha|}\right)^{\ell}|\alpha|^{it},  \quad \text{ with } \ell \in \Z, t \in \R,\]
in which case $\xi$ is said to be of \textit{type} $(\ell,t)$, and we refer to $\ell \in \Z$ as the \textit{frequency} of $\xi$.  The $L$-function attached to the 
Gr\"ossencharakter $\xi$ is 
\[L(s, \xi):=\prod_{\p \textnormal{ prime}}\left(1-\frac{\xi(\p)}{\N(\p)^{s}}\right)^{-1}, \quad \quad \textrm{Re}(s)>1.\]
The \textit{conductor} of $\xi$ refers to the conductor of $\xi_{\mathrm{fin}}$, and similarly $\xi$ is said to be a \textit{primitive} Gr\"ossencharakter when $\xi_{\mathrm{fin}}$ is primitive.

Fix $d \in \Z$ to be odd and square-free.
For any prime ideal $\fp \subset \Z[i]$ with primary generator~$\boldsymbol{\pi}_\fp$, define
\begin{align} \label{def-xi}
\xi_{d}(\fp) := \begin{cases} 
\displaystyle{\frac{\boldsymbol{\pi}_\fp}{|\boldsymbol{\pi}_\fp|}} \cdot \overline{\chi}_{\fp}(d)  & \text{when } (\boldsymbol{\pi}_\fp,2+2i)=1, \\ 
     \\
0 & \text{when } (\boldsymbol{\pi}_\fp, 2+2i) \neq 1,
\end{cases}
\end{align}
and extend by multiplicativity to all ideals $\fa \subset \Z[i]$.  $\xi_{d}$ then defines a Gr\"ossencharakter modulo $(4d)$.  Indeed, for any ideal $(\alpha) \subset \Z[i]$ coprime to $(4d)$, we may write
$
\xi_{d}((\alpha)) = \xi_{d,\text{fin}}(\alpha) \xi_{\infty}(\alpha),
$
where $\xi_{\infty}: \C^\times \rightarrow S^{1}$
and $\xi_{d,\text{fin}}: \left(\Z[i]/(4d)\right)^{\times} \rightarrow S^{1}$
 are given by
 \begin{align}\label{xi infinite and finite}
 	\xi_{\infty}: \alpha \mapsto  \frac{\alpha}{|\alpha|},  \quad \xi_{d,\text{fin}}: \alpha \mapsto \overline{\chi}_{(\alpha)}(d) \; {\chi}_{(2+2i)}(\alpha).
 \end{align}
By \eqref{boldalpha}, this matches $\eqref{def-xi}$ when $(\alpha) \neq (1+i)$ is prime.

As the primary generator of a prime $p \equiv 3 \bmod 4$ is equal to $-p$, we find by \eqref{chi_at_integers} that when $p\nmid d$, we have $\xi_d((p)) = -1$ . Upon comparing Euler factors with \eqref{L-fct-E}, we see that
\begin{equation}\label{y2x3Lfunction}
L(s-\tfrac12, \xi_d) = \prod_{\p \textnormal{ prime}}\left(1-\frac{\xi_{d}(\p)}{\N(\p)^{s-\frac{1}{2}}}\right)^{-1}=L(s,E_{d}), \quad \quad \re(s) > \frac32.
\end{equation}

\subsection{Computing Conductors}

Let us proceed by providing a more user-friendly expression for~$\xi_{d,\text{fin}}$.

\begin{lem}\label{Lem expression xidfin}
	Let $d \in \Z$ be an odd and square-free integer, and $\xi_{d,\textnormal{fin}}$ as in~\eqref{xi infinite and finite}. 
	One has
	\begin{equation}\label{expression xidfin}
		\xi_{d,\textnormal{fin}} = \begin{cases}
				\overline{\chi}_{(d)}  {\chi}_{(2+2i)}  & \textnormal{ if } d\equiv  1 \bmod 8\\
				\overline{\chi}_{(d)}  {\chi}_{(2)}  \overline{\chi}_{(4)} &\textnormal{ if } d\equiv  3 \bmod 8 \\
				\overline{\chi}_{(d)}  {\chi}_{(2+2i)} \chi_{(2)} & \textnormal{ if } d\equiv  5 \bmod 8 \\
				
				\overline{\chi}_{(d)}   \overline{\chi}_{(4)}& \textnormal{ if } d\equiv  7 \bmod 8.
		\end{cases}
	\end{equation}
In particular the conductor of  $\xi_{d}$ is given by
\[
\mathfrak{f}_d=
\begin{cases}
	((2+2i)d) & \textnormal{ if } d \equiv 1 \bmod 4 \\
		(4d) & \textnormal{ if } d \equiv 3 \bmod 4.
\end{cases}
\]
\end{lem}

\begin{proof}

 If $d \equiv 1 \bmod 4$, then $d$ is primary, and noting that $\N(d)-1 \equiv 0 \bmod 8$, it follows from quartic reciprocity (Proposition \ref{reciprocity}) and \eqref{boldalpha} that
\[{\bquartic{d}{{\alpha}}} = {\bquartic{d}{\chi_{(2+2i)}(\alpha) {\alpha}}} = {\bquartic{\chi_{(2+2i)}(\alpha) {\alpha}}{d}}.\]
Thus, by multiplicativity,
\begin{align*}
\xi_{d,\text{fin}}(\alpha)= {\bquartic{\alpha}{d}} \; {\bquartic{\chi_{(2+2i)}(\alpha)}{d}}  \;{\chi}_{(2+2i)}(\alpha)= {\bquartic{\alpha}{d}}  \;{\chi}_{(2+2i)} (\alpha)^{1-\sum_{\p|d}\frac{(\N(\p)-1)}{4}}. \label{case-D-1}
\end{align*}
Moreover, observe that for $p \equiv 3 \bmod 4$, one has 
$
\frac{\N((p)) -1}{4} = \begin{cases}
		0\bmod 4 & \textnormal{ if } p\equiv  7, 15\bmod{16}\\
	2\bmod 4 & \textnormal{ if } p\equiv  3,11 \bmod{16}
\end{cases}.
$
Similarly for $p  \equiv 1 \bmod 4$, one has 
$
\sum_{\p\mid p}\frac{\N(\p) -1}{4} = \begin{cases}
	0\bmod 4 & \textnormal{ if } p\equiv  1 \bmod 8 \\
		2\bmod 4 & \textnormal{ if } p\equiv  -3 \bmod 8
\end{cases}$.\\
These together imply that 
\begin{equation*}\label{sum of norm even}
	\sum_{\p \mid d}\frac{\N(\p) -1}{4} \equiv 2\,\# \lbrace p\mid d : p \equiv \pm 3 \bmod8\rbrace \bmod 4 \equiv \begin{cases}
		0\bmod4  & \textnormal{ if } d\equiv \pm 1 \bmod 8\\
		2\bmod4 & \textnormal{ if } d\equiv \pm 3 \bmod 8.
	\end{cases} 
\end{equation*}
Upon noting that  $\chi^{2}_{(2+2i)} = \chi_{(2)}$, one has
\begin{equation}\label{power of chi2+2i}
	{\chi}_{(2+2i)} ^{-\sum_{\p|d}\frac{\N(\p)-1}{4}} = \begin{cases}
		\chi_0 & \textnormal{ if } d\equiv \pm 1 \bmod 8\\
		\chi_{(2)} & \textnormal{ if } d\equiv \pm 3 \bmod 8,
	\end{cases}
\end{equation}
where $\chi_0$ is the trivial character modulo $(2+2i)$. 
Regrouping, \eqref{expression xidfin} now follows for the case $d\equiv 1 \bmod 4$. \\

Next, suppose $d \equiv 3 \bmod 4$.  Then $-d$ is primary, and similarly to above we find that 
\begin{align} \nonumber
\xi_{d,\text{fin}}(\alpha) &= {{\bquartic{-1}{ {\alpha}}}} \; {{\bquartic{-d}{ {\alpha}}}} {\chi}_{(2+2i)}(\alpha) \\ \nonumber
&=   {{\bquartic{-1}{{  \alpha}}}} \; {\bquartic{\alpha}{-d}} \; {\bquartic{\chi_{(2+2i)}(\alpha)}{-d}}  {\chi}_{(2+2i)}(\alpha)  \\ \nonumber
&=   {{\bquartic{-1}{{\alpha}}}} \; {\bquartic{\alpha}{d}}  {\chi}_{(2+2i)} (\alpha)^{1-\sum_{\p|d}\frac{(\N(\p)-1)}{4}}.
\end{align}
Since
\begin{align*}
 \bquartic{-1}{\alpha} = (-1)^{(\N(\alpha)-1)/4} = \begin{cases} 1 & \alpha \equiv \pm 1, \pm i \bmod 4 \\
-1 & \alpha \equiv  \pm (3+2i), \pm i(3+2i)\bmod 4, \end{cases}
\end{align*}
it follows from \eqref{def-2+2i} and \eqref{def-4} that
\[\bquartic{-1}{\alpha}{\chi}_{(2+2i)} (\alpha) = \overline{\chi}_{(4)}(\alpha).\]
Combining the above with~\eqref{power of chi2+2i}, the lemma follows.
 \end{proof}

Finally, for any $k \geq 1$, we consider
\begin{equation}\label{sympowerD}
\xi_{d}^k : (\alpha) \mapsto \xi_{d,\text{fin}}^k(\alpha) \left( \frac{\alpha}{|\alpha|} \right)^k.
\end{equation}
Let $\xi_{d,k}$ denote the primitive character inducing $\xi_{d}^k$. Noting that $\chi^{2}_{(2+2i)} = \chi_{(2)}= \overline{\chi}^{2}_{(4)}$ has conductor~$(2)$, it follows from \eqref{expression xidfin} that 
$\xi_{d,k}$ is a character with frequency $k$ and conductor
\begin{equation}\label{conductor}
\mathfrak{f}_{d,k} = \begin{cases}
\mathfrak{f}_d & \text{$k$ odd} \\
(2d) & k \equiv 2 \bmod 4 \\
(1) & k \equiv 0 \bmod 4. \end{cases} 
\end{equation}

Since the same set of primes divide the conductors of both $\xi_{d,k}$ and $\xi_{d}$ whenever $k \not \equiv 0 \bmod 4$, it follows that
\begin{equation}\label{conductor_2}
	\xi_{d,k}(\p) = 
	\begin{cases}
		\xi_{d}^{k}(\p) = 
		\displaystyle{\left(\frac{\boldsymbol{\pi}_{\fp}}{|\boldsymbol{\pi}_{\fp}|}\right)^{k}}  \cdot \overline{\chi}_{\fp}(d)^k
		& \textnormal{ if }\p \nmid \mathfrak{f}_{d}\\
		0 & \textnormal{ if }  \p \mid \mathfrak{f}_{d} \textnormal{ and } k \not \equiv 0 \bmod 4\\
		\left(\displaystyle \frac{\pi_{\fp}}{|\pi_{\fp}|}\right)^{k} & \textnormal{ if } \p = (\pi_{\fp}) \mid \mathfrak{f}_{d} \textnormal{ and } k \equiv 0 \bmod 4,
	\end{cases}
\end{equation}
where, in the last case, we further note that $\xi_{d,k}(\p)$ does not depend on the chosen generator~$\pi_{\fp}$ of~$\p$.

We conclude this section with the following definitions of the angle of the Gaussian primes with respect to the character $\xi_d$.
\begin{defn} Let $d$ be an odd square-free integer, and $p \equiv 1 \bmod 4$. We write $(p) = \fp \overline{\fp}$ and recall that $\xi_d(\overline{\fp}) = \overline{\xi_d(\fp)}$. We define $\theta_{d,p}$ and~$z_{d,p}$ by
\begin{align} \label{def-angle}
\xi_d(\fp) + \xi_d(\overline{\fp}) &:= 2 \cos{\theta_{d,p}}, \;\; \theta_{d,p} \in (0, \pi) \\ \label{def-z} 
z_{d, p} &:= \sqrt{p} e^{i \theta_{d,p}}  \in \Z[i].
\end{align}\end{defn}
For example, if $d=1$, then $z_{1,p} = a+2bi$ is primary with $b \geq 0$, and $p = a^2 + 4b^2.$
We note that, in general, for $p \equiv 1 \bmod 4$, $p\nmid d$, with $(p) = \p \overline{\p}$, we have that
\begin{equation}\label{z_in_terms_xi}
\xi_{d,k}(\p)+\xi_{d,k}(\overline{\p}) = 2 \cos k \theta_{d,p},
\end{equation}
where $\theta_{d,p} \in (0, \pi)$ is the angle defined by \eqref{def-angle}.

\subsection{Computing Root Numbers}\label{background-RN}
From the work of Hecke \cite{Hecke1920},   $L(s, \xi_{d,k})$ has  analytic continuation to $\C$. Noting that the discriminant of $\Q(i)$ is equal to $-4$, we define the \textit{completed $L$-function}
\begin{align}\label{completed_L-function}
\Lambda(s, \xi_{d,k}) &:= \left( {4 \N(\ff_{d,k})} \right)^{s/2} (2 \pi)^{-s}
 \Gamma \left( {s + \frac{k}{2}} \right) L(s, \xi_{d,k}) 
 \end{align}
where $\ff_{d,k}$ is as in \eqref{conductor}.  As in \cite[Thm. 3.8]{IwanKow}, we note that $\Lambda(s, \xi_{d,k})$  satisfies the functional equation
\begin{align}\label{FE}
\Lambda(s, \xi_{d,k}) = W(\xi_{d,k}) \Lambda(1-s, {\overline{\xi}_{d,k}}) =W(\xi_{d,k}) \Lambda(1-s,\xi_{d,k}),
\end{align}
where the sign of the functional equation is denoted by the \textit{root number} $W(\xi_{d,k}) = \pm 1$, and where the last equality follows 
upon noting that $\overline{\xi}_{d,k}(\mathfrak{a}) = \xi_{d,k}(\overline{\mathfrak{a}})$.

We now proceed to compute $W(\xi_{d,k})$ explicitly, for any given $k \in \mathbb{N}$ and odd square-free $d \in \Z$.
\begin{lem} \label{lemma-sign} Let $d$ be an odd square-free integer. If $k$ is even, the root number of $\xi_{d,k}$ is
\begin{align} \label{k-even} W(\xi_{d,k}) =  1.\end{align} 
If $d \equiv 1 \bmod 4$, then it satisfies
\begin{align} \label{d=1mod4}
W(\xi_{d,k}) = \begin{cases} 
-\mathrm{sgn}(d) & \textnormal{ if } d \equiv 5,9 \bmod {16} \textnormal{ and } k \equiv 1,3 \bmod 8\\
\mathrm{sgn}(d) & \textnormal{ if } d \equiv 5,9 \bmod {16} \textnormal{ and } k \equiv 5,7 \bmod 8\\
\mathrm{sgn}(d) & \textnormal{ if } d \equiv 1,13 \bmod {16} \textnormal{ and } k \equiv 1,3 \bmod 8\\
- \mathrm{sgn}(d) & \textnormal{ if } d \equiv 1,13 \bmod {16} \textnormal{ and } k \equiv 5,7 \bmod 8,
\end{cases}
\end{align}
and if $d \equiv 3 \bmod 4$, we have
\begin{align}  \label{d=3mod4} 
	W(\xi_{d,k}) =\begin{cases}
\mathrm{sgn}(d) & \textnormal{ if } d \equiv 3 \bmod 8 \textnormal{ and } k \equiv 1 \bmod 4\\
-\mathrm{sgn}(d) & \textnormal{ if } d \equiv 3 \bmod 8 \textnormal{ and } k \equiv 3 \bmod 4\\
-\mathrm{sgn}(d) & \textnormal{ if } d \equiv 7 \bmod 8 \textnormal{ and } k \equiv 1 \bmod 4\\
\mathrm{sgn}(d) & \textnormal{ if } d \equiv 7 \bmod 8 \textnormal{ and } k \equiv 3 \bmod 4.
\end{cases}
\end{align}
For any fixed odd square-free  $d \in \mathbb Z$, we therefore conclude that
\begin{align} \label{average-RN}
\lim_{K \rightarrow \infty} \frac{\#\{1 \leq k \leq K: W(\xi_{d,k}) = -1\}}{K} = \frac{1}{4}.
\end{align}
\end{lem}
 In particular, recalling definition~\eqref{def S+-}, we have
 \begin{equation}\label{set s- computed}
 	S_{-}(d) = \begin{cases} 	
 		\lbrace 5,7\rbrace &\text{ if } d \equiv 1, 13 \bmod{16}  \text{ and } d>0. \\
 		\lbrace 1,3\rbrace &\text{ if } d \equiv 1, 13 \bmod{16}  \text{ and } d<0. \\
 		\lbrace 1,3\rbrace &\text{ if } d \equiv 5, 9 \bmod{16}  \text{ and } d>0. \\
 		\lbrace 5,7\rbrace &\text{ if } d \equiv 5, 9 \bmod{16}  \text{ and } d<0. \\
 		\lbrace 3, 7\rbrace &\text{ if } d \equiv 3 \bmod{8}  \text{ and } d>0. \\
 		\lbrace 1,5\rbrace &\text{ if } d \equiv 3 \bmod{8}  \text{ and } d<0. \\
 		\lbrace 1,5\rbrace &\text{ if } d \equiv 7 \bmod{8}  \text{ and } d>0. \\
 		\lbrace 3,7\rbrace &\text{ if } d \equiv 7 \bmod{8}  \text{ and } d<0. \\
 	\end{cases}
 \end{equation}

\begin{proof}[Proof of Lemma \ref{lemma-sign}]

As in \cite[$(3.85)$, $(3.86)$]{IwanKow} (see also \cite[\S~4]{De2023+}), we will use the formula
    	\begin{align*}
	W(\xi_{d,k}) = i^{-k} \N(\mathfrak{f}_{d,k})^{-\frac12} \xi_{d,k,\infty}(\gamma_{d,k}) \sum_{x \in \mathbb{Z}[i]/\mathfrak{f}_{d,k}}\xi_{d,k,\mathrm{fin}}(x) e^{2\pi i \tr\left(\tfrac{x}{\gamma_{d,k}}\right)}
\end{align*}
where we take $\gamma_{d,k} \in \Z[i]$ to be any generator of the ideal $(2) \mathfrak{f}_{d,k}$, and $\mathfrak{c} = \Z[i]$, so that $(\mathfrak{c},\mathfrak{f}_{d,k}) = 1$ and $(2)\mathfrak{c}\mathfrak{f}_{d,k} = (\gamma_{d,k})$.

First, we treat the case $k \equiv 0 \bmod 4$.  By~\eqref{conductor} we may choose $\gamma_{d,k} = 2$, from which it follows that
    	\begin{align*}
	W(\xi_{d,k}) = e^{2\pi i \tr\left(\tfrac{1}{2}\right)} = 1.
\end{align*}

For $k \not\equiv 0 \bmod 4$, we note by Lemma~\ref{Lem expression xidfin} that
\[\xi_{d,k,\mathrm{fin}}(x) =
\overline{\chi}_{(d)}^{k}(x)\cdot \eta_{d,k}(x),
\]
where
 \begin{align*}
\eta_{d,k} &:=
\begin{cases}
 \chi_{(2)} & \textnormal{ when } k \equiv 2 \bmod 4 \\
\chi_{(2+2i)}^k & \textnormal{ when } d\equiv 1 \bmod 8 \text{ and } k \text{ is odd}\\
\chi_{(2+2i)}^k\chi_{(2)} & \textnormal{ when } d\equiv 5 \bmod 8 \text{ and } k \text{ is odd}\\
\overline{\chi}^{k}_{(4)}\chi_{(2)} & \textnormal{ when } d\equiv 3 \bmod 8 \text{ and } k \text{ is odd}\\
\overline{\chi}^{k}_{(4)} & \textnormal{ when }d\equiv 7 \bmod 8 \text{ and } k \text{ is odd},
\end{cases}
\end{align*}
 is a primitive character modulo $(g)$ with  $$g:=\begin{cases}
 	2 & \textnormal{ when }k \equiv 2 \bmod4\\
 	2+2i & \textnormal{ when }k\equiv 1 \bmod 2 \textnormal{ and } d\equiv 1\bmod 4\\
 	4 & \textnormal{ when }k\equiv 1 \bmod 2\textnormal{ and } d\equiv 3 \bmod 4.\end{cases}$$
Write $\lvert d\rvert = \prod_j p_j$, where $p_j$ run through the distinct {\it rational} primes dividing the odd square-free  integer $d$. 
Since $(g,d)=1$, by the Chinese remainder theorem there exists a ring isomorphism
\[\Z[i]/(g) \times \prod_{p_{j}|d}\Z[i]/(p_{j}) \rightarrow \Z[i]/(gd) \quad (x_{0},(x_{j})_{j}) \mapsto u x_{0}+\sum_{j}v_{j}x_{j},\]
where $u \equiv 1\bmod g$ and $u \equiv 0 \bmod d$, while $v_{j} \equiv 1 \bmod {p_{j}}$  and  $v_{j} \equiv 0 \bmod {\frac{gd}{p_{j}}}$ for all $j$. Upon choosing $\gamma_{d,k} = 2g|d|$, 
we find that
	\begin{align}
\begin{split}
&\sum_{x \in \mathbb{Z}[i]/\mathfrak{f}_{d,k}}\xi_{d,k,\mathrm{fin}}(x) e^{2\pi i \tr\left(\tfrac{x}{\gamma_{d,k}}\right)} = \sum_{x \in \mathbb{Z}[i]/(gd)}\overline{\chi}_{(d)}^{k}(x)\cdot \eta_{d,k}(x) e^{2\pi i \tr\left(\tfrac{x}{2g|d|}\right)}\\
&= \sum_{x_{0} \in \mathbb{Z}[i]/(g)}\prod_{p_{j}|d}\sum_{x_{j} \in \mathbb{Z}[i]/(p_{j})} \overline{\chi}^{k}_{(p_{j})}\Big(ux_{0}+\sum_{j}v_{j} x_{j}\Big)\eta_{d,k}\Big(ux_{0}+\sum_{j}v_j x_{j}\Big)  e^{2\pi i \tr\left(\tfrac{u x_{0} + \sum_{j}v_{j}x_{j}}{2g|d|}\right)}\\
&= \sum_{x_{0} \in \mathbb{Z}[i]/(g)}\eta_{d,k}\left(ux_{0}\right)e^{2\pi i \tr\left(\tfrac{u x_{0}}{2g|d|}\right)}\prod_{p_{j}|d}\sum_{x_{j} \in \mathbb{Z}[i]/(p_{j})} \overline{\chi}^{k}_{(p_{j})}\left(v_{j} x_{j}\right)  e^{2\pi i \tr\left(\tfrac{v_{j}x_{j}}{2g|d|}\right)}.
\end{split}
\end{align}

Applying the change of variables $\alpha = u x_{0}/|d|$ and $\beta =
v_j x_{j}p_{j}/g|d|$, we then find that
	\begin{align}\label{root_number_split}
\begin{split}
W(\xi_{d,k}) &=  \frac{i^{-k}}{|gd|} 
\left(\frac{g}{|g|}\right)^{k}
\eta_{d,k}(|d|)\sum_{\alpha \in \mathbb{Z}[i]/(g)} \eta_{d,k}(\alpha) e^{2\pi i \tr\left(\tfrac{\alpha}{2g}\right)} \\
	&\phantom{=}\times  \prod_{\substack{p_{j}|d\\ p_{j} > 0}}
	\overline{\chi}^{k}_{(p_{j})}\left(\frac{g|d|}{p_{j}}\right)\sum_{\beta \in \mathbb{Z}[i]/(p_{j})}\overline{\chi}^{k}_{(p_{j})}(\beta) e^{2\pi i \tr\left(\tfrac{\beta}{2p_{j}}\right)} \\
	&= W(\xi_{d,k}, 2) \times \eta_{d,k}(|d|) \times \overline{\chi}^k_{(d)}(g) 
	\times \prod_{\substack{p_{j}|d\\ p_{j} > 0}} W(\xi_{d,k}, p_j) 
\end{split}
\end{align}
where
\begin{align*}
W(\xi_{d,k}, 2) &:= \frac{i^{-k}}{|g|} 
\Big(\frac{g}{|g|}\Big)^{k}
\sum_{\alpha \in \mathbb{Z}[i]/(g)} \eta_{d,k}(\alpha) e^{2\pi i \tr\left(\tfrac{\alpha}{2g}\right)} \\
W(\xi_{d,k}, p) &:= \frac{1}{p}\sum_{x \in \mathbb{Z}[i]/(p)}\overline{\chi}^{k}_{(p)}(x) e^{2\pi i \tr\left(\tfrac{x}{2p}\right)},
\end{align*}
and where we have used  \eqref{chi_at_integers}.

We first study the contribution of $2$.
For $k\equiv 2 \bmod 4$, we compute 
\begin{align}\label{gauss_sum_even}
W(\xi_{d,k}, 2) &= - \frac{1}{2} 
	\sum_{x \in \mathbb{Z}[i]/(2)} \eta_{d,k}(x)e^{2\pi i \tr\left(\tfrac{x}{4}\right)} = - \frac{1}{2} 
	\left(1\cdot 
e^{i \pi}-1 \cdot 1	\right) = 1.
\end{align}
If $k$ is odd and $d \equiv 1 \bmod 4$, then we compute
\begin{align} \nonumber 
i^k e^{\frac{-i \pi k}{4}} W(\xi_{d,k}, 2) &= \frac{1}{2\sqrt{2}} 
	\sum_{x \in \mathbb{Z}[i]/(2+2i)} \eta_{d,k}(x)e^{2\pi i \tr\left(\tfrac{x}{4+4i}\right)} \\
	&= \frac{1}{2\sqrt{2}} \bigg(1\cdot e^{\frac{i \pi}{2}}-e^{-\frac{i \pi}{2}}	+ \eta_{d,k}(i)e^{\frac{i \pi}{2}} +\overline{\eta_{d,k}( i)}e^{\frac{-i \pi}{2}}\bigg)\nonumber \\
	&\label{gauss_sum_1mod4}= \frac{i}{\sqrt{2}} \left(1+\eta_{d,k}(i)\right) 
	 =  \begin{cases}
e^{i \left(\frac{\pi}{2}\mp \frac{\pi}{4}\right)} & \textnormal{ if } k \equiv \pm 1 \bmod 4 \textnormal{ and } d\equiv 1 \bmod 8 \\
e^{i\left(\frac{\pi}{2}\pm\frac{\pi}{4}\right)} & \textnormal{ if } k \equiv \pm1 \bmod 4 \textnormal{ and } d\equiv -3 \bmod 8.
	\end{cases}
\end{align}
Similarly when $k$ is odd and $d\equiv 3 \bmod 4$, we compute
\begin{multline} \label{gauss_sum_3mod4}
i^k  W(\xi_{d,k}, 2) = \frac{1}{4} 
\sum_{x \in \{\pm 1, \pm i, \pm (3+2i),\pm (-2+3i)\}} \eta_{d,k}(x)e^{2\pi i \tr\left(\tfrac{x}{8}\right)}\\ 
= \frac{1}{4} \bigg(1\cdot e^{\frac{i \pi}{2}}-e^{-\frac{i \pi}{2}} + \eta_{d,k}(i) + \eta_{d,k}(-i) 
	 -e^{\frac{3 \pi i}{2}} +e^{-\frac{3 \pi i}{2}} -\eta_{d,k}(2-3i) - \eta_{d,k}(-2 +3i)\bigg) = i. 
\end{multline}

Let us now study the contribution at a rational odd prime $p$. 
We have
\begin{align*}
\sum_{x \in \mathbb{Z}[i]/(p)}\overline{\chi}^{k}_{(p)}(x) e^{2\pi i \tr\left(\tfrac{x}{2p}\right)}&=\sum_{a \in \mathbb{Z}/(p)}\sum_{b \in \mathbb{Z}/(p)}\overline{\chi}^{k}_{(p)}(a+bi) e^{2\pi i \tr\left(\tfrac{a+bi}{2p}\right)}\\
&=\sum_{b \in \mathbb{Z}/(p)}\overline{\chi}^{k}_{(p)}(bi)+\sum_{a \in (\mathbb{Z}/(p))^{\times}}\sum_{b \in \mathbb{Z}/(p)}\overline{\chi}^{k}_{(p)}(a+bi) e^{2\pi i \tr\left(\tfrac{a+bi}{2p}\right)}\\
&=(p-1)\overline{\chi}^{k}_{(p)}(i)+\sum_{a \in (\mathbb{Z}/(p))^{\times}}e^{ \tfrac{2\pi i a}{p}}\sum_{b \in \mathbb{Z}/(p)}\overline{\chi}^{k}_{(p)}(a+bi)
\end{align*}
by \eqref{chi_at_integers}.  Note that the sum
\[\sum_{b \in \mathbb{Z}/(p)}\overline{\chi}^{k}_{(p)}(a+bi) = \overline{\chi}^{k}_{(p)}(a)\sum_{b \in \mathbb{Z}/(p)}\overline{\chi}^{k}_{(p)}(1+bi) = \sum_{b \in \mathbb{Z}/(p)}\overline{\chi}^{k}_{(p)}(1+bi)\]
is independent of $a \in (\mathbb{Z}/(p))^{\times}$.  Moreover, by orthogonality, we find that
\[0 = \sum_{a \in \mathbb{Z}/(p)}\sum_{b \in \mathbb{Z}/(p)}\overline{\chi}^{k}_{(p)}(a+bi) = \sum_{a \in (\mathbb{Z}/(p))^{\times}}\sum_{b \in \mathbb{Z}/(p)}\overline{\chi}^{k}_{(p)}(a+bi)+(p-1)\overline{\chi}^{k}_{(p)}(i).\]
It follows that for all $a \in (\mathbb{Z}/(p))^{\times}$,
\[\sum_{b \in \mathbb{Z}/(p)}\overline{\chi}^{k}_{(p)}(a+bi)= -\overline{\chi}^{k}_{(p)}(i).\]
Hence
\begin{align}
W(\xi_{d,k}, p) = \frac{1}{p}\sum_{x \in \mathbb{Z}[i]/(p)}\overline{\chi}^{k}_{(p)}(x) e^{2\pi i \tr\left(\tfrac{x}{2p}\right)}&=\frac{\overline{\chi}^{k}_{(p)}(i)}{p}\bigg((p-1)-\sum_{a \in (\mathbb{Z}/(p))^{\times}}e^{\tfrac{2\pi i a}{p}}\bigg) \nonumber\\
\label{gauss_sum_prime}
&=  \overline{\chi}^{k}_{(p)}(i)=\begin{cases}
(-1)^{k}& \textnormal{ when }p \equiv \pm 3 \bmod 8\\
1 & \textnormal{ when }p \equiv \pm 1 \bmod 8,
\end{cases}
\end{align}
and therefore for odd square-free $d \in \mathbb{Z}$,
\begin{equation}\label{root_number_product}
\prod_{\substack{p_{j}|d\\ p_{j} > 0}} W(\xi_{d,k}, p_j)= \begin{cases}
(-1)^k& \textnormal{ when }d \equiv \pm 3 \bmod 8\\
1 & \textnormal{ when }d \equiv \pm 1 \bmod 8.
\end{cases}
\end{equation}
By \eqref{root_number_split}, \eqref{gauss_sum_even} and 
\eqref{root_number_product}, it follows that when $k \equiv 2 \bmod 4$,
	\begin{align*}
	W(\xi_{d,k})	&= \eta_{d,k}(|d|)= 1.
\end{align*}
This proves \eqref{k-even}.  When $k \equiv  \pm 1 \bmod 4$ and $d \equiv 1 \bmod 8$, it follows from
\eqref{eisenstein_reciprocity}, \eqref{root_number_split}, \eqref{gauss_sum_1mod4} and 
\eqref{root_number_product}, that
\begin{align*}
	W(\xi_{d,k})	&=W(\xi_{d,k}, 2) \times \eta_{d,k}(|d|) \times \overline{\chi}^k_{(d)}(2+2i) 
	\times \prod_{\substack{p_{j}|d\\ p_{j} > 0}} W(\xi_{d,k}, p_j)\\
&=\mp ie^{\frac{i \pi k}{4}} \; e^{i \left(\frac{\pi}{2}\mp \frac{\pi}{4}\right)}
	\; \chi_{(2+2i)}^{k}(|d|) \; \overline{\chi}_{(d)}^{k}(1+i)\overline{\chi}_{(d)}^{k}(2)\\
	&= \pm e^{ \frac{i\pi (k \mp 1)}{4}} \mathrm{sgn}(d)
(-i)^{\frac{k(d-1)}{4}},
\end{align*}
where in the last line we note that 
$\chi_{(2+2i)}^{k}(|d|) =\mathrm{sgn}(d)$.
Similarly, when $k \equiv \pm 1 \bmod 4$ and $d \equiv 5 \bmod 8$, we find that
\begin{align*}
	W(\xi_{d,k})	&= \pm i e^{\frac{i \pi k}{4}} \; e^{i \left(\frac{\pi}{2}\pm \frac{\pi}{4}\right)} \;
	\eta_{d,k}(|d|)\;  \overline{\chi}_{(d)}^{k}(2+2i)\\
	&= \mp e^{ \frac{i \pi(k \pm 1)}{4}} \mathrm{sgn}(d)(-i)^{\frac{k(d-1)}{4}}.
\end{align*}
This proves \eqref{d=1mod4}.
Finally, when $k$ is odd and $d \equiv 3 \bmod 8$, we note by \eqref{root_number_split}, \eqref{gauss_sum_3mod4}, \eqref{root_number_product}, that
\begin{align*}
	W(\xi_{d,k})	&= W(\xi_{d,k}, 2) \times \eta_{d,k}(|d|) \times \overline{\chi}^k_{(d)}(4) 
	\times \prod_{\substack{p_{j}|d\\ p_{j} > 0}} W(\xi_{d,k}, p_j)\\
&= - i^{1-k}  \overline{\chi}^{k}_{(4)}(|d|) \;
	 =(-1)^{\frac{k-1}{2}}\mathrm{sgn}(d),
\end{align*}
 and if $d \equiv 7 \bmod 8$, we have
\begin{align*}
	W(\xi_{d,k})	= i^{1-k}  \overline{\chi}^{k}_{(4)}(|d|)
	=-(-1)^{\frac{k-1}{2}}\mathrm{sgn}(d),
\end{align*}
which proves \eqref{d=3mod4}.
\end{proof}

\section{Computing $\mathcal{D}(\phi, \xi_{d,k})$}
\label{section compute D_k}

To begin our analysis of $\mathcal{D}(\phi, \xi_{d,k})$, let us first explain the scaling parameter in  \eqref{def-local_OLD}.
	Let
	\[ \mathcal{N}_{d,k}(T):=\#\{s \in \C: L(s,\xi_{d,k})=0, 0 \leq \textrm{Re}(s) \leq 1, -T \leq \textrm{Im}(s) \leq T\}\]
	denote the number of zeros of $L(s,\xi_{d,k})$ on the critical strip up to height $T$.  Using the functional equation \eqref{FE} as  in \cite[Prop. 2.2]{Holm2023}, we find that in the limit as $k \rightarrow \infty$,
	\begin{align} \label{norm}
	\frac{1}{2T} \cN_{d,k}(T) \sim \frac{1}{2T} \frac{T \log (T k^{2} \N(\mathfrak{f}_{d,k}))}{\pi} 
	\sim \frac{ \log{(  k^2 N_{d,k})}}{2\pi} =  \frac{ \log{(  k M_{d,k})}}{\pi}
	\end{align}
	under the assumption of the \textit{Generalized Riemann Hypothesis} (GRH) (see also \cite[Thm. 5.8]{IwanKow} for the $T$-limit), and 
	where we use the notation
	\begin{align}\label{def_M}
	N_{d,k} := \N(\mathfrak{f}_{d,k}),  \quad \n := \N(\mathfrak{f}_{d,k})^{\frac12}.
	\end{align}
	
	To see that the analytic conductor of $L(s, \xi_{d,k})$ (as defined in \cite[p. 95]{IwanKow}) is asymptotic to~$k^2 \N(\ff_{d,k})$, we apply the duplication formula to \eqref{completed_L-function} and write
$$
\Lambda(s, \xi_{d,k}) =  \frac{2^\frac{k}{2}}{2 \sqrt{\pi}} \left( {4 \N(\ff_{d,k})} \right)^{s/2} \pi^{-s} \Gamma \Big( \frac{s + \frac{k}{2}}{2} \Big)  \Gamma \Big( \frac{s + \frac{k}{2} + 1}{2} \Big) L(s, \xi_{d,k}).
$$
The scaling parameter $\tfrac{\log(k^2\N(\ff_{d,k}))}{2\pi}$ in~\eqref{def-local_OLD} is thus chosen so that the average spacing between the scaled zeros is asymptotically equal to $1$, as $k$ grows large. Note that in much of the literature, the scaling parameter is often chosen to be independent of $k$.  We could, therefore, have instead chosen the relevant normalization factor to be, e.g. $\tfrac{\log(K^2d^2)}{2\pi}$.  While such an alternative might facilitate certain computations concerning the $k$-averaging, we do consider our choice to be technically more precise, particularly given our focus on lower order contributions.  Our choice moreover guarantees that the contribution of the $\Gamma$-factors to the main term, prior to taking the $k$-average (see  Lemma~\ref{Gamma fixed k}), is precisely $\phat(0)$.

For each character $\xi_{d, k}$, we rewrite the functional equation \eqref{FE} as 
$$L(s, \xi_{d,k}) = X_{d,k}(s) L(1-s, \xi_{d,k}),$$
where
\begin{align*}
X_{d,k}(s) := W(\xi_{d,k}) \left( {\N(\mathfrak{f}_{d,k})} \right)^{\frac12-s}\pi^{2s-1} \frac{\Gamma \left( 1-s + \frac{k}{2} \right)}{\Gamma \left( s + \frac{k}{2} \right)}.
\end{align*}
Upon taking logarithmic derivatives, we find that 
\begin{align} \label{log1}
\frac{L'}{L}(s, \xi_{d,k}) &= \frac{X_{d,k}'}{X_{d,k}} (s) - \frac{L'}{L}(1-s, \xi_{d,k}) \\ \label{log2}
\frac{X_{d,k}'}{X_{d,k}} (s) &=  -\log\N(\mathfrak{f}_{d,k}) + 2 \log \pi - \frac{\Gamma'}{\Gamma}\Big(1-s + \frac{k}{2}\Big) -  \frac{\Gamma'}{\Gamma}\Big(s + \frac{k}{2}\Big).
\end{align}

Since $\phi$ is an even Schwartz function, and $L(s, \xi_{d, k})$ has no trivial zeros for $\textrm{Re}(s) > -\frac12$, we see from \eqref{log1} that
\begin{align} \nonumber
\mathcal{D}(\phi, \xi_{d,k}) &=\frac{1}{2 \pi i} \left( \int_{(\frac54)} - \int_{(-\frac14)} \right)  \frac{L'}{L}(s,\xi_{d,k}) 
\phi \left(\frac{\log{( k^2 N_{d,k})}}{2\pi}\; \frac{s-\tfrac12}{i} \right) 
 \dd s\\  \nonumber
 &=\frac{1}{2 \pi i} \int_{(\frac54)} \left(  \frac{L'}{L}(s,\xi_{d,k})  - \frac{L'}{L}(1-s,\xi_{d,k}) \right)
 \phi \left(\frac{\log{( k^2 N_{d,k})}}{2\pi}\; \frac{s-\tfrac12}{i} \right) 
  \dd s\\  \nonumber
&=\frac{1}{2 \pi i} \int_{(\frac54)} \left( 2 \frac{L'}{L}(s,\xi_{d,k}) - \frac{X'_{d,k}}{X_{d,k}}(s) \right) 
\phi \left(\frac{\log{ (k^2 N_{d,k})}}{2\pi} \;
\frac{s-\tfrac12}{i} \right) 
 \dd s\\  \label{D-phi-k}
&=U_{L}(\phi,d,k) + U_{\Gamma}(\phi,d,k),
\end{align}
where
\begin{align} \label{def-U-L} U_{L}(\phi,d,k) := \frac{1}{\pi i} \int_{(\frac54)}\frac{L'}{L}(s,\xi_{d,k})
\phi \left(  \frac{\log{ (k^2 N_{d,k})}}{2\pi}
\; \frac{s-\tfrac12}{i} \right) 
  \dd s\end{align}
and
\begin{align} \label{def-U-G}
U_{\Gamma}(\phi,d,k) :=-\frac{1}{2 \pi i} \int_{(\frac54)} \frac{X'_{d,k}}{X_{d,k}}(s) 
\phi \left(\frac{\log{ (k^2 N_{d,k})}}{2\pi}
\; \frac{s-\tfrac12}{i} \right) 
 \dd s.\end{align}

\begin{lem}\label{Gamma fixed k}
	Let $k$ be a positive integer, $d$ be an odd square-free integer, and $\phi$ be an even Schwartz function such that $\widehat{\phi}$ is compactly supported. Then as $k \rightarrow \infty$,
\[U_{\Gamma}(\phi,d,k) =\widehat{\phi}(0)-\frac{\log{2\pi} }{\log{(k M_{d,k})}}\widehat{\phi}(0) +O\left(\frac{1}{k \log(k\n)}\right).\]
\end{lem}

\begin{proof}
Upon applying the change of variables $r = s-\tfrac12$ to \eqref{def-U-G} and using \eqref{log2}, we have
\begin{multline*}
U_{\Gamma}(\phi,d,k) 
=~ \frac{1}{2 \pi i} \int_{(\frac34)} \bigg(\log N_{d,k} - 2 \log \pi + \frac{\Gamma'}{\Gamma}\Big(\frac{k+1}{2}-r\Big) +  \frac{\Gamma'}{\Gamma}\Big(\frac{k+1}{2}+r\Big)\bigg) \\ 
\times
\phi \left( \frac{\log{ (k^2 N_{d,k})}}{2\pi i} \; r \right) \, \dd r.
\end{multline*}
Noting that $\Gamma'(s)/\Gamma(s)$ is holomorphic on the half-plane Re$(s)> 0$, we shift the contour to the imaginary line  Re$(r) = 0$, and upon applying the change of variables 
$$\tau := \frac{\log{ (k^2 N_{d,k})}}{2\pi i}\; r,$$
we find that
\begin{multline*} U_{\Gamma}(\phi,d,k) = \frac{1}{\log{(k^2 N_{d,k})}} \int_{\R} \Big(\log N_{d,k} - 2 \log \pi + \frac{\Gamma'}{\Gamma}\Big(\frac{k+1}{2}-\frac{ 2\pi i \tau}{\log{(k^2N_{d,k})}} 
\Big)  \\ +   \frac{\Gamma'}{\Gamma} \Big( \frac{k+1}{2}+\frac{ 2\pi i \tau}{\log{(k^2N_{d,k})}} \Big) \Big) \phi (\tau ) \, \dd\tau.\end{multline*}
As in the proof of \cite[Lem.~5.1]{Wa2021}, it follows from Stirling's approximation that
\begin{align*}
U_{\Gamma}(\phi,d,k)&= \frac{1}{\log{(k^2 N_{d,k})}}
\Big(( \log N_{d,k} - 2 \log \pi)\widehat{\phi}(0)+2
\int_{\R}\phi(\tau)\Big(\log \Big|\tfrac{k+1}{2}+\tfrac{ 2\pi i \tau}{\log{(k^2N_{d,k})}}\Big| +O\left(\tfrac{1}{k}
\right)\Big)\dd\tau\Big)\\
&= \frac{\widehat{\phi}(0)}{\log{(k^2 N_{d,k})}} \Big( \log N_{d,k} - \log \pi^2 + 2 \log \big(\tfrac{k+1}{2}\big)\Big)\\
&\qquad \qquad \qquad +O\bigg(\frac{1}{\log{(k^2 N_{d,k})}}\int_{\R}|\phi(\tau)|\left(\log\left|1+\tfrac{2}{k+1}\tfrac{2\pi i \tau}{\log{(k^2N_{d,k})}}\right|+\tfrac {1}{k}\right)\dd\tau \bigg).
\end{align*}
Applying the inequality $\log |1+z| \leq |z|$, we note that by the rapid decay of $\phi$,
\begin{align*}
\int_{\R}|\phi(\tau)|\left(\log\left|1+\tfrac{2}{k+1}\tfrac{2\pi i \tau}{\log{(k^2N_{d,k})}}\right|\right)\dd\tau \ll \tfrac{1}{k\log{(k^2N_{d,k})}}\int_{\R}|\phi(\tau)|\left|\tau \right|\dd\tau \ll  \tfrac{1}{k\log{(k^2N_{d,k})}}.
\end{align*}
It follows that
\begin{align*}
U_{\Gamma}(\phi,d,k)&= \phat(0) -  \frac{ \log{ ( 4 \pi^2) }}{\log{(k^2 N_{d,k})}}\phat(0)
 +O\bigg(\frac {1}{k \log{(k^2 N_{d,k})} }\bigg),
\end{align*}
as desired.
\end{proof}

We now  compute the contribution of $U_{L}(\phi,d,k)$. Separating according to the splitting properties of $(p) \subseteq \Z[i]$, we write 
\begin{align} \label{log-derivative}
-  \frac{L'(s, \xi_{d,k})}{L(s, \xi_{d,k})} = \sum_{\substack{p \equiv 1 \bmod 4\ \\ (p) = \p\overline{\p}\\n \geq 1}} \frac{ \big(\xi_{d,k}^{n}(\p) + \overline{\xi}_{d,k}^{n}(\p) \big) \log{p}}{p^{ns}}
+ \sum_{\substack{p \equiv 3 \bmod 4 \\n \geq 1}} \frac{2\xi_{d,k}^{n}((p))\log{p}}{p^{2 n s}}  
+ \sum_{n \geq 1} \frac{\xi_{d,k}^{n}((1+i)) \log{2}}{2^{ns}},
\end{align}
and 
\begin{equation}\label{UL_sum}
U_{L}(\phi,d,k) = U_{\text{split}}(\phi, d,k)+U_{\text{inert}}(\phi,d, k)+U_{\text{ram}}(\phi, k)
\end{equation}
where
 \begin{align} \label{Usplit}
U_{\text{split}}(\phi,d,k) &:= - \frac{1}{ \pi i} \int_{(\frac54)} \sum_{\substack{p \equiv 1 \bmod 4\\ (p) = \p \overline{\p} \\ 
 n \geq 1}} \frac{ \big(\xi_{d,k}^{n}(\p) + \overline{\xi}_{d,k}^{n}(\p) \big)\log{p}}{p^{ns}}
 \phi \left( \frac{(s-\frac12) \log(k^2 N_{d,k})}{2\pi i} \right) \, \dd s,
\end{align}
\begin{align}\label{Uinert}
 U_{\text{inert}}(\phi,d, k) &:= - \frac{1}{ \pi i} \int_{(\frac54)} \sum_{\substack{p \equiv 3 \bmod4\\ (p)\nmid \mathfrak{f}_{d,k}\\ n \geq 1}} \frac{2(-1)^{kn}\log{p} }{p^{2 n s}}  \phi \left( \frac{(s-\frac12) \log(k^2 N_{d,k})}{2\pi i} \right) \dd s,
\end{align}
\begin{equation}\label{Uram}
U_{\text{ram}}(\phi,k):=- \frac{1}{ \pi i} \int_{(\frac54)}\sum_{n \geq 1} \frac{\xi_{d,k}^{n}((1+i)) \log{2}}{2^{ns}}\phi \left(\frac{(s-\frac12) \log(k^2 N_{d,k})}{2\pi i}\right) \dd s.
\end{equation}
Note that by \eqref{conductor_2}, one has
\[\xi_{d,k}((1+i)) =
\begin{cases}
	0 & \textnormal{ if } k \not \equiv 0 \bmod 4\\
	(-1)^{\frac{k}{4}} & \textnormal{ if } k \equiv 0 \bmod 4,
\end{cases}\]
so that $U_{\text{ram}}(\phi,k)$ is indeed independent of $d$. 

\begin{lem}  \label{U-split-and-inert} Let $k$ be a positive integer, $d$ be an odd square-free integer, and $\phi$ be an even Schwartz function such that $\widehat{\phi}$ is compactly supported.
Let $U_{\textnormal{split}}(\phi,d,k)$, $U_{\textnormal{inert}}(\phi,d, k)$, and $U_{\textnormal{ram}}(\phi,k)$ be given by \eqref{Usplit}, \eqref{Uinert} and \eqref{Uram}, respectively.   Then
\begin{align} \label{U-split}
U_{\textnormal{split}}(\phi,d,k) &= - \frac{1}{\log ( k \n)} \sum_{\substack{p \equiv 1 \bmod4\\
(p) = \p \overline{\p}\\n \geq 1}} \frac{ \big(\xi_{d,k}^{n}(\p) + \overline{\xi}_{d,k}^{n}(\p) \big)\log{p}}{p^{\frac{n}{2}}}
\widehat\phi \left( \frac{n \log{p}}{ 2\log (k \n)} \right)\\ \label{U-inert}
U_{\textnormal{inert}}(\phi,d,k) &= - \frac{1}{\log (k \n)} \sum_{\substack{p \equiv 3\bmod4\\  (p) \nmid \mathfrak{f}_{d,k}\\ n \geq 1}} \frac{2(-1)^{kn}\log{p}}{p^n} \; \widehat{\phi}  \left( \frac{n  \log p}{\log (k \n)} \right)\\
\label{U-ram} U_{\textnormal{ram}}(\phi,k)&=
-\frac{1}{\log (k \n)}\sum_{n \geq 1} \frac{  (-1)^{\frac{kn}{4}} \log 2}{2^{\frac{n}{2}}}\widehat{\phi}\Bigg(\frac{n \log 2}{ 2\log (k \n)}\Bigg) \quad \text{ if } k \equiv 0 \bmod 4,
\end{align}
and $U_{\textnormal{ram}}(\phi,k) = 0 $ if $k \not\equiv 0 \bmod 4$.
\end{lem}
\begin{proof}
For the inert primes, we compute 
\begin{align*}  - U_{\text{inert}}(\phi, d,k) 
&= \frac{1}{ \pi i} \int_{(\frac54)}  \sum_{\substack{p \equiv 3 \bmod4\\(p) \nmid \mathfrak{f}_{d,k}\\n \geq 1}} \frac{2(-1)^{kn}\log{p} }{p^{2 n s}}  
 \phi \left(\frac{(s-\frac12) \log(k^2 N_{d,k})}{2\pi i}\right) \dd s\\ &=
 \sum_{\substack{p \equiv 3 \bmod 4\\(p) \nmid \mathfrak{f}_{d,k}\\n \geq 1}} 
 \frac{2(-1)^{kn}\log{p}}{ \pi i} \int_{(\frac12)}  {p^{-2 n s}}   \phi \left( \frac{(s-\frac12) \log(k^2 N_{d,k})}{2\pi i} \right) \dd s.
\end{align*}
Note that switching the order of summation and integration is justified upon noting that at Re$(s)=\frac54,$
\begin{align*}
\int_{(\frac54)}\sum_{\substack{p \equiv 3 \bmod4\\(p) \nmid \mathfrak{f}_{d,k}\\n \geq 1}}& \bigg \vert \frac{2(-1)^{kn}\log{p}}{p^{2 n s}}  
 \phi \left( \frac{(s-\frac12) \log(k^2 N_{d,k})}{2\pi i} \right)\bigg \vert \dd s\\
&  \ll \sum_{m} \frac{\log{m}}{m^{\frac{5}{2}}}  
\int_{(\frac54)} \bigg \vert  \phi \left( \frac{(s-\frac12) \log(k^2 N_{d,k})}{2\pi i} \right)\bigg \vert \dd s\ll \infty,\end{align*}
since $\phi$ is a Schwartz function (e.g. \cite[Lem.~3.7]{Wa2021}). The shift from $\textnormal{Re}(s) = \frac54$ to $\textnormal{Re}(s) = \frac12$ is then further justified upon noting that $s \mapsto p^{-2ns}\phi \left(\frac{(s-\frac12) \log(k^2 N_{d,k})}{2\pi i} \right)$ is holomorphic. 
Applying the change of variables $t = (s-\frac12)  \frac{\log(k^2 N_{d,k})}{2\pi i}$, and noting that $p^{-2 n s} = p^{-2n (s-\frac12)} p^{-n}$, we obtain
\begin{multline*} \frac{1}{ \pi i} \int_{(\frac12)}  {p^{-2 n s}}   \phi \left(  \frac{(s-\frac12) \log(k^2 N_{d,k})}{2\pi i} \right) \, \dd s
	\\
= \frac{2}{p^{n} \log (k^2 N_{d,k})} \int_{-\infty}^\infty \phi(t)  e^{- 2\pi i \frac{2n  \log p}{\log (k^2N_{d,k})}   t} \dd t =\frac{1}{p^{n} \log (k \n)} \widehat{\phi}\left(\frac{n  \log p}{\log (k \n)}  \right).
\end{multline*}
The proofs for $U_{\textnormal{split}}(\phi,d,k)$ and $U_{\textnormal{ram}}(\phi,k)$ follow similarly.
\end{proof}

We now compute $U_{\textnormal{ram}}(\phi,k)$, including lower-order terms in descending powers of $\log(k\n)$.

\begin{lem}\label{ram_computation} Let $k$ be a positive integer, $d$ be an odd square-free integer, and $\phi$ be an even Schwartz function such that $\widehat{\phi}$ is compactly supported.
Then for any $J \in \mathbb{N}$ and $k \rightarrow \infty$, we have
\[U_{\textnormal{ram}}(\phi,k) = \sum_{\substack{j=0 \\ j \textnormal{ even}}}^{J-1}\frac{c_{j,\textnormal{ram}}(k)\; \widehat{\phi}^{(j)}(0)}{(\log( k \n))^{j+1}}+ O_{J}\left((\log (k\n))^{-J-1}\right),\]
where for $j\geq 0$, 
\begin{equation}\label{ram_lower_order}
c_{j,\textnormal{ram}}(k) :=
\begin{cases}
-\frac{2}{j!}\left(\frac{\log 2}{2}\right)^{j+1}\textnormal{Li}_{-j}\left(\frac{(-1)^{\frac{k}{4}}}{\sqrt{2}}\right) & \textnormal{ if } k \equiv 0 \bmod 4\\
0 &  \textnormal{ if } k \not \equiv 0  \bmod 4,
\end{cases}
\end{equation}
and
\begin{align} 
\label{def-Li}  \textnormal{Li}_{-j}(z):=\sum_{n=1}^{\infty}n^{j}z^{n}.\end{align}
\end{lem}

\begin{proof} 
	Suppose $k \equiv 0 \bmod 4$.
	Then by \eqref{U-ram}, we have
\begin{multline*}
U_{\textnormal{ram}}(\phi,k)=-\frac{1}{\log(k \n)}\sum_{n = 1}^{N} \frac{(-1)^{\frac{kn}{4}}\log{2} }{2^{\frac{n}{2}}}\widehat{\phi}\left(\frac{n \log 2}{2 \log(k \n)}\right)
\\-\frac{1}{\log(k \n)}\sum_{n \geq N + 1} \frac{(-1)^{\frac{kn}{4}}\log{2} }{2^{\frac{n}{2}}}\widehat{\phi}\left(\frac{n \log 2}{2 \log(k \n)}\right),
\end{multline*}
for any $N \in \mathbb{N}$. Choosing $N = 2J \log \log(k \n)/\log 2$, we then bound 
\begin{align*}
\frac{1}{\log(k \n)}\sum_{n \geq N + 1} \frac{(-1)^{\frac{kn}{4}}\log{2} }{2^{\frac{n}{2}}}\widehat{\phi}\left(\frac{n \log 2}{2 \log(k \n)}\right)\ll \frac{2^{-\frac{N}{2}}}{\log(k \n)} = \frac{2^{-\frac{J \log \log(k \n)}{ \log 2}}}{\log(k \n)} = (\log(k \n))^{-J-1}.
\end{align*}
Moreover, by Taylor expansion one has
\begin{equation}\label{taylor_expansion}
\widehat{\phi}\left(\frac{n \log 2}{2 \log(k \n)}\right) = \sum_{j=0}^{J-1}\frac{\widehat{\phi}^{(j)}(0)}{j!}\left(\frac{n \log 2}{2 \log(k \n)}\right)^{j}+O_{J}\left(\left(\frac{n \log 2}{2 \log(k \n)}\right)^{J}\right),
\end{equation}
so that
\begin{align*}
\sum_{n = 1}^{N} \frac{(-1)^{\frac{kn}{4}}\log{2} }{2^{\frac{n}{2}}}&\widehat{\phi}\left(\frac{n \log 2}{2 \log(k \n)}\right)\\
& = \sum_{n = 1}^{N} \frac{(-1)^{\frac{kn}{4}}\log{2} }{2^{\frac{n}{2}}}\left(\sum_{j=0}^{J-1}\frac{\widehat{\phi}^{(j)}(0)}{j!}\left(\frac{n \log 2}{2 \log(k \n)}\right)^{j}+O_{J}\left(\left(\frac{n \log 2}{2 \log(k \n)}\right)^{J}\right)\right).
\end{align*}
Since
\[\sum_{n = 1}^{N} \frac{\log{2} }{2^{\frac{n}{2}}}\left(n \log 2\right)^{J} \ll_{J} \sum_{n = 1}^{\infty} \frac{n^{J} }{2^{\frac{n}{2}}}< \infty, \]
we find that
\begin{align*}
U_{\textnormal{ram}}(\phi,k) 
&= -2\sum_{j=0}^{J-1}\frac{\widehat{\phi}^{(j)}(0)}{j!}\left(\frac{\log 2}{2 \log(k \n)}\right)^{j+1}\sum_{n = 1}^{N} \frac{(-1)^{\frac{kn}{4}}n^{j}}{2^{\frac{n}{2}}} + O_{J}\left((\log(k \n))^{-J-1}\right).
\end{align*}
Since $n \mapsto \frac{n^{j}}{2^{\frac{n}2}}$ is decreasing in the range $n >\frac{2j}{\log 2}$, and moreover $N \geq \tfrac{2J}{\log 2}$ for sufficiently large $k$, one has that
\begin{align*}
\Big\lvert \sum_{n = N+1}^{\infty} \frac{(-1)^{\frac{kn}{4}}n^{j}}{2^{\frac{n}{2}}} \Big\rvert \leq \sum_{n = N+1}^{\infty} \frac{n^{j}}{2^{\frac{n}{2}}} \leq \int_{N}^{\infty}\frac{x^{j}}{2^{\frac{x}{2}}}\dd x = \int_{N}^{\infty}x^{j}e^{-{\frac{x}{2}\log 2}}\dd x.
\end{align*}
Upon repeated application of integration by parts, we find that
\begin{align*}
\int_{N}^{\infty}x^{j}e^{-{\frac{x}{2}\log 2}}\dd x &=  \frac{N^{j}}{2^{N/2}}\frac{2}{\log 2}+\frac{2j}{\log 2}\int_{N}^{\infty}x^{j-1}e^{-x\frac{\log 2}{2}}\dd x = O_j\left(\frac{N^{j}}{2^{N/2}}\right) = O_{J}\left(\frac{(\log \log(k \n))^{j}}{(\log(k \n))^{J}}\right).
\end{align*}
We thus may write
\begin{align*}
	U_{\textnormal{ram}}(\phi,k) 
	=& -2\sum_{j=0}^{J-1}\frac{\widehat{\phi}^{(j)}(0)}{j!}\left(\frac{\log 2}{2 \log(k \n)}\right)^{j+1}\left(\sum_{n = 1}^{\infty} \frac{(-1)^{\frac{kn}{4}}n^{j}}{2^{\frac{n}{2}}}  + O_{J}\left(\frac{(\log \log(k \n))^{j}}{(\log(k \n))^{J}}\right) \right)
	\\
	& + O_{J}\left((\log(k \n))^{-J-1}\right) \\
	=& -2\sum_{j=0}^{J-1}\frac{\widehat{\phi}^{(j)}(0)}{j!}\left(\frac{\log 2}{2 \log(k \n)}\right)^{j+1}\sum_{n = 1}^{\infty} \frac{(-1)^{\frac{kn}{4}}n^{j}}{2^{\frac{n}{2}}} + O_{J}\left((\log(k \n))^{-J-1}\right),
\end{align*}
as desired, where we note that since $\phi$ is even, we have~$\phat^{(j)}(0) = 0$ for all odd $j$.
\end{proof}

Next, we write
\[U_{\textnormal{inert}}(\phi,d,k)=U_{\textnormal{inert}}(\phi,k)+U_{\textnormal{inert},d}(\phi,k),\]
where
\begin{equation}\label{Def Uinert indep d}
	U_{\textnormal{inert}}(\phi,k):= - \frac{1}{\log{(k\n)}} \sum_{\substack{p \equiv 3 \bmod4\\ n \geq 1}} \frac{2(-1)^{kn}\log{p}}{p^n} \; \widehat{\phi}  \left( \frac{n \log{p}}{\log{(k\n)}} \right)
\end{equation}
and
\[U_{\textnormal{inert},d}(\phi,k):= \frac{1}{\log{(k\n)}} \sum_{\substack{p \equiv 3 \mod4\\  (p) \mid \mathfrak{f}_{d,k}\\ n \geq 1}} \frac{2(-1)^{kn}\log{p}}{p^n} \; \widehat{\phi}  \left( \frac{n \log{p}}{\log{(k\n)}} \right).
\]
We proceed to compute $U_{\textnormal{inert},d}(\phi,k)$ including the lower-order terms in descending powers of $\log( k\n)$.

\begin{lem}\label{inert_D_computation} Let $k$ be a positive integer, $d$ be an odd square-free integer, and $\phi$ be an even Schwartz function such that $\widehat{\phi}$ is compactly supported.
 Then for any $J \in \mathbb{N}$ and $k \rightarrow \infty$, we have
\[U_{\textnormal{inert},d}(\phi,k) = \sum_{\substack{j=0 \\ j \textnormal{ even}}}^{J-1}\frac{c_{j,\textnormal{inert},d}(k)\; \widehat{\phi}^{(j)}(0)}{(\log(k\n))^{j+1}}+ O_{J}\left(\frac{\log |d|}{(\log(k\n))^{J+1}}\right),\]
where for $j\geq 0$, the constants $c_{j,\textnormal{inert},d}(k)$ are 
\begin{equation}\label{c_inert_D}
c_{j,\textnormal{inert},d}(k) :=
\begin{cases}
\frac{2}{j!}\sum_{\substack{p \equiv 3 \bmod 4 \\ p \mid d}}(\log p)^{j+1}\textnormal{Li}_{-j}\left(\frac{(-1)^{k}}{p}\right) & \textnormal{ if } k \not \equiv 0 \bmod 4\\
0 &  \textnormal{ if } k \equiv 0 \bmod 4,
\end{cases}
\end{equation}
and $\textnormal{Li}_{-j}(z)$ is defined by \eqref{def-Li}.
\end{lem}

\begin{proof} We get for $k \not \equiv 0 \bmod 4$
\begin{align*}
U_{\textnormal{inert},d}(\phi,k)&=
\frac{1}{\log{(k\n)}} \sum_{\substack{p \equiv 3 \bmod4\\ p \mid d}}\sum_{n \geq 1} \frac{2(-1)^{kn}\log{p}}{p^n} \; \widehat{\phi}  \left( \frac{n \log{p}}{\log(k\n)} \right).
\end{align*}
For each such $p|d$, we cut the inner sum at $N_{p} = J \log \log (k\n)/\log p$, to obtain
\begin{align*}
\frac{1}{\log (k\n)}\sum_{n \geq N_{p} + 1} \frac{2(-1)^{kn}\log{p} }{p^{n}}\widehat{\phi}\left(\frac{n \log p}{\log( k\n)}\right)\ll \frac{p^{-(N_{p}+1)}\log p}{\log (k\n)} = \frac{\log p}{p}(\log(k\n))^{-J-1}.
\end{align*}
Taylor expanding  as in \eqref{taylor_expansion}, we find that
\begin{multline*}
\frac{1}{\log (k\n)}\sum_{n = 1}^{\infty} \frac{2(-1)^{kn}\log{p} }{p^{n}}\widehat{\phi}\left(\frac{n \log p}{\log (k\n)}\right) \\= \frac{1}{\log (k\n)}\sum_{1 \leq n \leq N_{p}} \frac{2(-1)^{kn}\log{p} }{p^{n}}\sum_{j=0}^{J-1}\frac{\widehat{\phi}^{(j)}(0)}{j!}\left(\frac{n \log p}{\log(k\n)}\right)^{j}
+O_{J}\left(\frac{1}{p}\left(\frac{\log p}{\log(k\n)}\right)^{J+1}\right),
\end{multline*}
where the bound on the error term is computed upon noting that
\[\frac{1}{\log (k\n)}\sum_{1\leq n \leq N_{p}} \frac{\log{p} }{p^{n}}\left(\frac{n \log p}{\log(k\n)}\right)^{J} \ll \sum_{n=1}^{\infty}\frac{n^{J}}{p^{n}}\left(\frac{\log{p}}{\log(k\n)}\right)^{J+1} \ll_{J} \frac{1}{p}\left(\frac{\log p}{\log(k\n)}\right)^{J+1}.\]
Upon applying the trivial bound
\[\sum_{p\mid d}\frac{(\log p)^{J+1}}{p}\ll_{J} \sum_{p\mid d}\log p = \log |d|,\]
we thus obtain
\begin{align*}
U_{\textnormal{inert},d}(\phi,k)&=\sum_{j=0}^{J-1}\frac{\widehat{\phi}^{(j)}(0)}{j!}\sum_{\substack{p \equiv 3 \bmod 4 \\ p \mid d}} \sum_{1 \leq n \leq N_{p}} \frac{2(-1)^{kn} n^j}{p^{n}}\left(\frac{ \log p}{\log(k\n)}\right)^{j+1}+O_{J}\left(\frac{\log |d|}{(\log(k\n))^{J+1}}\right).
\end{align*}
Upon repeated application of integration by parts, we find similarly to as above that
\begin{align*}
\sum_{n\geq N_{p}+1} \frac{n^{j}(\log{p})^{j+1}}{p^{n}}\ll_{j} \frac{N_{p}^{j}(\log{p})^{j+1}}{p^{N_{p}}} =  O_{J}\left(\frac{(\log \log(k\n))^{j}\log p}{(\log(k\n))^{J}}\right),
\end{align*}
and therefore we have
\begin{align*}
U_{\textnormal{inert},d}(\phi,k)	&=\sum_{j=0}^{J-1}\frac{\widehat{\phi}^{(j)}(0)}{j!(\log(k\n))^{j+1}}\sum_{\substack{p \equiv 3 \bmod 4 \\ p \mid d}}\sum_{n\geq 1} \frac{2(-1)^{kn}n^{j}(\log{p})^{j+1}}{p^{n}}+O_{J}\left(\frac{\log |d|}{(\log(k\n))^{J+1}}\right),
\end{align*}
as desired, upon recalling that~$\phat^{(j)}(0) = 0$ when $j$ is odd.
\end{proof}

Next we compute $U_{\textnormal{inert}}(\phi,k)$. To this end, we define
\begin{equation}\label{psi}
\psi_{k}(t; 3,4) := (-1)^{k}\sum_{\substack{p^n \leq t\\ p \equiv 3 \bmod4 \\ n \geq 1}} (-1)^{kn}\log {p} = \frac{t}{2} + O_A \left( \frac{t}{(\log{t})^A} \right),
\end{equation}
for any $A> 0$, by the prime number theorem in arithmetic progressions.

\begin{lem}\label{inert_computation}
Let $k$ be a positive integer, $d$ be an odd square-free integer, and $\phi$ be an even Schwartz function such that $\widehat{\phi}$ is compactly supported.
 Then for any $J \in \mathbb{N}$ and $k \rightarrow \infty$, we have
\[U_{\textnormal{inert}}(\phi,k) = \frac{(-1)^{k+1}}{2} \int_{\R}\widehat{\phi}(u) du+
\sum_{\substack{j=0 \\ j \textnormal{ even}}}^{J-1}\frac{c_{j,\textnormal{inert}}(k) \; \widehat{\phi}^{(j)}(0)}{(\log(k\n))^{j+1}}+O_{J}\left((\log(k\n))^{-J-1}\right),\]
where
\begin{equation}\label{c_0_inert}
c_{0,\textnormal{inert}}(k):=(-1)^{k+1}\left(1+2\int_1^\infty
 \frac{(\psi_{k}(t; 3,4) -\frac{t}2)}{t^{2}}\dd t\right)
\end{equation}
and for $j \geq 1$,
\begin{equation}\label{c_j_inert}
c_{j,\textnormal{inert}}(k):=2(-1)^{k} \int_1^\infty
 \frac{(\psi_{k}(t; 3,4) -\frac{t}2)}{t^{2}} \;\frac{(\log t)^{j-1}}{(j-1)!}\left(1-\frac{\log t}{j}\right)\dd t.
 \end{equation}
 \end{lem}

\begin{proof}
	Recalling~\eqref{Def Uinert indep d}, we write
\begin{align*}
U_{\textnormal{inert}}(\phi,k) = - \frac{1}{\log{(k\n)}} \sum_{\substack{p \equiv 3 \bmod4\\n \geq 1}} \frac{2(-1)^{kn} \log{p}}{p^n} \; \widehat{\phi}  \left( \frac{n \log{p}}{\log{(k\n)}} \right)
\end{align*}
and compute
\begin{align} \label{before-PNT-ET} \nonumber
&\sum_{\substack{p \equiv 3 \bmod 4\\n \geq 1}} \frac{2(-1)^{kn} \log{p}}{p^n} \; \widehat{\phi}  \left( \frac{n \log{p}}{\log{(k\n)}} \right)\\  \nonumber
&=2 (-1)^{k}
\int_1^\infty \widehat\phi \left( \frac{\log t}{\log (k\n)} \right) \frac{\dd \psi_{k}(t; 3,4)}{t}= 2(-1)^{k+1} \int_1^\infty \psi_{k}(t; 3,4) \dd\Bigg( \frac{ \widehat\phi \left( \frac{\log t}{\log (k\n)} \right)} {t} \Bigg) \\
 \nonumber
&= (-1)^{k}\left(\widehat\phi(0) + \int_1^\infty \widehat\phi \left( \frac{\log{t}}{\log{(k\n)}} \right) \frac{\dd t}{t}\right)
- 2(-1)^{k} \int_1^\infty \left(\psi_{k}(t; 3,4) -\frac{t}2 \right) \dd\Bigg( \frac{ \widehat\phi \left( \frac{\log t}{\log (k\n)} \right)} {t} \Bigg)\\
&= (-1)^{k}\left(\widehat\phi(0)+\frac{\log{(k\n)}}{2} \int_{\R}\widehat{\phi}(u) \dd u
- 2 \int_1^\infty \left(\psi_{k}(t; 3,4) -\frac{t}2 \right) \dd\Bigg( \frac{ \widehat\phi \left( \frac{\log t}{\log(k\n)} \right)} {t} \Bigg)\right),
\end{align}
since $\widehat{\phi}(u)$ is even. Noting that
\begin{align*}
\frac{\dd}{\dd t}\Bigg( \frac{ \widehat\phi \left( \frac{\log t}{\log(k\n)} \right)} {t} \Bigg) &= \frac{1}{t^{2}}\left(\frac{1}{\log(k\n)}\cdot \widehat{\phi}^{(1)}\left(\frac{\log t}{\log(k\n)}\right)-\widehat{\phi}\left(\frac{\log t}{\log(k\n)}\right)\right),
\end{align*}
and that
\[\widehat{\phi}\left(\frac{\log t}{\log(k\n)}\right) = \sum_{j=0}^{J-1}\frac{\widehat{\phi}^{(j)}(0)}{j!}\left(\frac{\log t}{\log(k\n)}\right)^{j}+O_{J}\left(\left(\frac{\log t}{\log(k\n)}\right)^{J}\right),\]
we find that
\begin{align*}
&\frac{\dd}{\dd t}\Bigg( \frac{ \widehat\phi \left( \frac{\log t}{\log(k\n)} \right)} {t} \Bigg) \\
&= \frac{1}{t^{2}}\left(\left(\sum_{j=0}^{J-1}\frac{\widehat{\phi}^{(j+1)}(0)}{j!}\left(\frac{(\log t)^{j}}{(\log(k\n))^{j+1}}\right)
-\sum_{j=0}^{J-1}\frac{\widehat{\phi}^{(j)}(0)}{j!}\left(\frac{\log t}{\log(k\n)}\right)^{j}\right)+O_{J}\left(\left(\frac{\log t}{\log(k\n)}\right)^{J}\right)\right)\\
&= \frac{1}{t^{2}}\left(-\widehat{\phi}(0)+\left(\sum_{j=1}^{J-1}\frac{\widehat{\phi}^{(j)}(0)(\log t)^{j-1}}{(\log(k\n))^{j}(j-1)!}\left(1-\frac{\log t}{j}\right)\right)+O_{J}\left(\frac{(\log t)^{J}+(\log t)^{J-1}}{\left(\log(k\n)\right)^{J}}\right)\right).
\end{align*}
Thus
\begin{align*}
&- 2 \int_1^\infty \Big(\psi_{k}(t; 3,4) -\frac{t}2 \Big) \dd\Bigg( \frac{ \widehat\phi \left( \frac{\log t}{\log(k\n)} \right)} {t} \Bigg)\\
&= 2 \int_1^\infty
 \frac{(\psi_{k}(t; 3,4) -\frac{t}2)}{t^{2}} \Bigg(\widehat{\phi}(0)-\left(\sum_{j=1}^{J-1}\frac{\widehat{\phi}^{(j)}(0)(\log t)^{j-1}}{(\log (k\n))^{j}(j-1)!}\left(1-\frac{\log t}{j}\right)\right)\\
 &\qquad \qquad \qquad \qquad \qquad \qquad \qquad \qquad \qquad +O_{J}\left(\frac{(\log t)^{J}+(\log t)^{J-1}}{(\log(k\n))^{J}}\right)\Bigg)\dd t \\
 &= 2 \widehat{\phi}(0)\int_1^\infty
 \frac{(\psi_{k}(t; 3,4) -\frac{t}2)}{t^{2}} \dd t
 - (-1)^k  \sum_{j=1}^{J-1} \frac{\widehat{\phi}^{(j)}(0)c_{j,\textnormal{inert}}(k)}{(\log(k\n))^{j}} +O_{J}\left((\log(k\n))^{-J}\right),
\end{align*}
and where the error term in the last step is bounded upon noting that for $A > J+1$,
\[\int_{1}^{\infty}\frac{(\psi_{k}(t; 3,4) -\frac{t}2)}{t^{2}}(\log t)^{J}  \dd t \ll \int_{1}^{\infty}\frac{ (\log t)^{J-A}}{t}  \dd t\ll 1,\]
by~\eqref{psi}.
Inserting this into~\eqref{before-PNT-ET}, we find that
\[U_{\textnormal{inert}}(\phi,k) =\frac{(-1)^{k+1}}{2} \int_{\R}\widehat{\phi}(u) \dd u+
\sum_{j=0}^{J-1}\frac{c_{j,\textnormal{inert}}(k) \widehat{\phi}^{(j)}(0)}{(\log(k\n))^{j+1}}+O_{J}\left((\log(k\n))^{-J-1}\right),\]
as desired, upon again recalling that~$\phat^{(j)}(0) = 0$ when $j$ is odd.
\end{proof}

By \eqref{D-phi-k} and \eqref{UL_sum}, we see that upon combining Lemma \ref{Gamma fixed k}, Lemma \ref{ram_computation}, Lemma \ref{inert_D_computation}, and Lemma \ref{inert_computation}, we obtain the following.

\begin{prop}\label{local_OLD}
Let $k$ be a positive integer, $d$ be an odd square-free integer, and $\phi$ be an even Schwartz function such that $\widehat{\phi}$ is compactly supported.
 Then for any $J \in \mathbb{N}$ and $k \rightarrow \infty$, we have
\begin{align*}
\mathcal{D}(\phi, \xi_{d,k}) &= \widehat{\phi}(0) + \frac{(-1)^{k+1}}{2} \int_{\R}\widehat{\phi}(u) \dd u\\
&\;\; +
\sum_{\substack{j=0 \\ j \textnormal{ even}}}^{J-1}\frac{c_{j}(d,k) \widehat{\phi}^{(j)}(0)}{(\log(k\n))^{j+1}}+ U_{\textnormal{split}}(\phi,d,k)+O_{J}\left(\frac{1+\log |d|}{(\log(k\n))^{J+1}}\right),
\end{align*}
where
\begin{equation}\label{c_0_combined}
c_{0}(d,k):= - \log{2\pi}+c_{0,\textnormal{ram}}(k)+c_{0,\textnormal{inert}}(k)+c_{0,\textnormal{inert},d}(k)
\end{equation}
and for $j \geq 1$,
\begin{equation}\label{c_j_combined}
	c_{j}(d,k):=c_{j,\textnormal{ram}}(k)+c_{j,\textnormal{inert}}(k)+c_{j,\textnormal{inert},d}(k).
\end{equation}
We moreover observe that $c_{j}(d,k)$ depends  only on the congruence class of $k$ modulo~$8$.
\end{prop}
\section{An Alternative Expression for $c_{j,\textnormal{inert}}(k)$}\label{alternative_lower_constants}

Next, we provide an alternative, more explicit, expression for the lower order constants, $c_{j,\textnormal{inert}}(k)$, which is moreover ultimately more amenable to numerical analysis.  

\begin{lem} Let $k$ be a positive integer. Then, we have
 \begin{equation}\label{c_0_inert explicit}
	c_{0,\textnormal{inert}}(k)= (-1)^{k}\left(\gamma_{0}-\frac{L'}{L}(1,\chi_{4}) + \log 2\right) +2 \frac{\zeta'_{4,3}}{\zeta_{4,3}}(2),
\end{equation}
where $\zeta_{4,3}(s) := \prod\limits_{p\equiv 3 \bmod 4}(1 - p^{-s})^{-1}$ on $ \re(s) > 1$.
\end{lem}
\begin{proof}
To begin, we write
\begin{align} \label{ID1}
2 \int_1^\infty
 \frac{\psi_{k}(t; 3,4) -\frac{t}2}{t^{2}}\dd t&=\int_1^\infty
 \frac{2\psi_{k}(t; 3,4) -\psi(t)+\psi(t)-t}{t^{2}}\dd t.
\end{align}
Then, denoting the \textit{Chebyshev} $\psi$-functions\[\psi(t):= \sum_{p^{n} \leq t}\log p \quad \textnormal{ and } \quad \psi(t,\chi_{4}):=\sum_{p^{n} \leq t}\chi_{4}(p^{n})\log p,\]
where $\chi_{4}$ is the quadratic Dirichlet character modulo $4$ over $\Z$,
we note the identity
\begin{align} \label{ID2}
2\psi_k(t;3,4) - (-1)^k 2\psi_0(\sqrt{t};3,4) =
2 \sum_{\substack{p \equiv 3 \bmod4 \\ n \textnormal{ odd} \\ p^n \leq t }} \log p
= \psi(t) - \psi(t,\chi_4) - \sum_{2^n\leq t}\log 2.
\end{align}
Using \eqref{ID1} and \eqref{ID2}, we get
\begin{align}
 \label{3-integrals}
 2 \int_1^\infty
 \frac{\psi_{k}(t; 3,4) -\frac{t}2}{t^{2}}\dd t
 &=\int_1^\infty
 \frac{-\psi(t,\chi_{4})+\psi(t)-t+2(-1)^{k}\psi_{0}(\sqrt{t};3,4) - \log 2 \lfloor \tfrac{\log t}{\log 2}\rfloor}{t^{2}} \;\dd t.
 \end{align}
We separate this integral into four components.
We first compute the contribution of the prime $2$ by noting that
\begin{align}\label{c_0 inert at 2}
	- \log 2 \int_1^{\infty} \lfloor \tfrac{\log t}{\log 2}\rfloor \frac{\dd t}{t^2} = 	-  \log 2 \sum_{n =0}^{\infty} n \int_{2^n}^{2^{n+1}}  \frac{\dd t}{t^2}
	= 	- \log 2 \sum_{n =0}^{\infty} \frac{n}{2^{n+1}} = - \log 2.
\end{align}
Next, by the prime number theorem in arithmetic progressions, together with Perron's formula, we write
\begin{align*}
-\int_1^\infty
 \frac{\psi(t,\chi_{4})}{t^{2}}\dd t &=
 -\int_1^T
 \frac{\psi(t,\chi_{4})}{t^{2}}\dd t + O\left(\frac{1}{(\log T)^{A}}\right)\\
 &=\frac{1}{2\pi i}\int_1^T\left(\int_{C_{U}(t)}
 \frac{L'}{L}(s,\chi_{4})\frac{t^{s-2}}{s}\dd s\right)\dd t + O\left(\frac{1}{(\log T)^{A}}+ \frac{(\log T)^{3}}{U}\right),
\end{align*}
for any $A > 1$, where for any $t \geq 1$, $C_{U}(t)$ denotes a finite path taken along the vertical line-segment $\re(s)=1+\tfrac{1}{\log (1+|t|)}$, where $|\im(s)|\leq U$.  Next, we shift the inner integral to a finite path $C(U)$, taken within a zero-free region of $L(s,\chi_{4})$ located to the left of the line $\re(s) = 1 - \tfrac{c}{\log U}$ for some $c > 0$.  Upon applying appropriate bounds on the growth of $\tfrac{L'}{L}(s,\chi_{4})$ within the critical strip (e.g. \cite[Thm. 11.4]{MV}), and picking up an additional error term from the horizontal segments, we find that
\begin{align*}
-\int_1^\infty
 \frac{\psi(t,\chi_{4})}{t^{2}}\dd t
&= \frac{1}{2\pi i}\int_{C(U)}
 \frac{L'}{L}(s,\chi_{4}) \int_1^T t^{s-2} \dd t\frac{\dd s }{s}+  O\left(\frac{1}{(\log T)^{A}}+ \frac{(\log T)^{3}}{U}+ \frac{\log T \log U}{U}\right).
\end{align*}
Next, we note that for any $s$ on $C(U)$,
$$
\int_{1}^{T} t^{s-2} \; \dd t =  \frac{T^{s-1}-1}{s-1} = \frac{1}{1-s}+O\left(\frac{\exp\left(-\frac{c\log T}{\log U}\right)}{|s-1|}\right).
$$
Upon setting $U = (\log T)^{4}$ and taking the limit as $T \rightarrow \infty$, we find that
\begin{equation*}
-\int_1^\infty
 \frac{\psi(t,\chi_{4})}{t^{2}}\dd t = \frac{1}{2\pi i}\int_{C}
 \frac{L'}{L}(s,\chi_{4})\frac{\dd s}{s(1-s)},
 \end{equation*}
where $C$ is now an infinite path taken within a zero-free region of $L(s,\chi_{4})$ located to the left of the line $\re(s) = 1$. Proceeding, we move the integral to the line Re$(s) = R > 1$ and pick up the pole at $s=1$.  Noting that $\frac{L'}{L}(s,\chi_4)$ is bounded on the vertical line Re$(s) = R$, it then follows that
\[\int_{(R)}
 \frac{L'}{L}(s,\chi_{4})\frac{\dd s}{s(1-s)} = o \left(\frac{1}{R}\right),\]
which goes to zero as $R \rightarrow \infty$.  This yields
\begin{equation}\label{L_at_one}
-\int_1^\infty
 \frac{\psi(t,\chi_{4})}{t^{2}}\dd t  =  \frac{L'}{L}(1,\chi_{4}).
 \end{equation}
Similarly, by applying the bound in \cite[Thm. 6.7]{MV}, we find that
\begin{align*}
\int_1^\infty
 \frac{\psi(t)-t}{t^{2}}\dd t &=\frac{1}{2\pi i}\int_1^\infty\left(\int_{(2)}
 - \frac{\zeta'}{\zeta}(s)\frac{t^{s-2}}{s}\dd s-\frac{1}{t}\right)\dd t\\
&=\frac{1}{2\pi i}\int_1^\infty\left(\int_{C}
 -\frac{\zeta'}{\zeta}(s)\frac{t^{s-2}}{s}\dd s+\frac{1}{t}-\frac{1}{t}\right)\dd t =\frac{1}{2\pi i}\int_{C}
 \frac{\zeta'}{\zeta}(s)\frac{\dd s}{s(s-1)},
 \end{align*}
where $C$ is defined analogously to as above.  Recall that
\[\zeta(1+s) = \frac{1}{s}+\sum_{n=0}^{\infty}\frac{(-1)^{n}}{n!}\gamma_{n}s^{n},\]
where $\gamma_{n}$ are known as the \textit{Stieltjes constants}  (and $\gamma_{0}$ is the \textit{Euler–Mascheroni constant}), so that
\begin{align*}
\frac{1}{s(s-1)}\frac{\zeta'}{\zeta}(s) &= \left(\frac{1}{s-1}-1+O(s-1)\right)\left(-\frac{1}{s-1}+\gamma_{0}+O(s-1)\right)\\
&=-\frac{1}{(s-1)^2}+\frac{\gamma_{0}+1}{s-1}+O(1).
\end{align*}
Upon shifting the contour $C$ far to the right, we find by Cauchy's residue theorem that
\begin{equation}\label{gamma_constant_revealed}
\int_1^\infty
 \frac{\psi(t)-t}{t^{2}}\dd t=\frac{1}{2\pi i}\int_{C}
 \frac{\zeta'}{\zeta}(s)\frac{\dd s}{s(s-1)}= -\gamma_{0}-1.
 \end{equation}
Finally, we note that
\begin{equation}\label{arithmetic_LOT}
\int_{1}^{\infty} \frac{2\psi_{0}(\sqrt{t};3,4)}{t^{2}}\dd t = - \frac{1}{2\pi i}\int_1^\infty 2\left(\int_{C}\frac{\zeta'_{4,3}}{\zeta_{4,3}}(2s)\frac{t^{s-2}}{s}\dd s\right)\dd t = -2\frac{\zeta'_{4,3}}{\zeta_{4,3}}(2),\end{equation}
 where
 \begin{equation}\label{zeta43}
 	\zeta_{4,3}(s) := \prod_{p\equiv 3 \bmod 4}(1 - p^{-s})^{-1}, \quad \quad \re(s) > 1,
 \end{equation}
so that
 \begin{equation}\label{Def_A_sum}
 	\frac{\zeta'_{4,3}}{\zeta_{4,3}}(s) = - \sum_{\substack{p\equiv 3 \bmod 4}}\sum_{n=1}^{\infty}\frac{\log p}{p^{ns}}.
 \end{equation}
Inserting \eqref{c_0 inert at 2}, \eqref{L_at_one}, \eqref{gamma_constant_revealed}, and \eqref{arithmetic_LOT}, into \eqref{3-integrals},
we have
\begin{align}
	2 \int_1^\infty
	\frac{\psi_{k}(t; 3,4) -\frac{t}2}{t^{2}}\dd t
	&= \frac{L'}{L}(1,\chi_{4}) -\gamma_{0}-1 +2(-1)^{k+1} \frac{\zeta'_{4,3}}{\zeta_{4,3}}(2) - \log 2,
\end{align}
from which the lemma follows by \eqref{c_0_inert}.
\end{proof}

\begin{lem}\label{lem c_j inert express middle} Let $k, j$ be positive integers. Then, we have
 \begin{align}\label{c_j expressed middle}
c_{j,\textnormal{inert}}(k)=&\frac{(-1)^{j+k}}{j!}\left(\left(\frac{f'}{f}\right)^{(j)}(1)-\left(\frac{L'}{L}\right)^{(j)}(1,\chi_{4})\right) \\ &
+ (-1)^{j} \frac{2^{j+1}}{j!} \left(\frac{\zeta'_{4,3}}{\zeta_{4,3}}\right)^{(j)}(2)  + (-1)^k \frac{(\log 2)^{j+1}}{j!}\mathrm{Li}_{-j}(\tfrac12),\nonumber
\end{align}
where $f(s):=(s-1)\zeta(s)$, $\zeta_{4,3}(s)$ is as in~\eqref{zeta43} and where $\mathrm{Li}_{-j}(z)$ is as in~\eqref{def-Li}.
\end{lem}
\begin{proof}
To begin, recall that for $j \geq 1$,
\begin{equation*}
c_{j,\textnormal{inert}}(k)=2(-1)^{k} \int_1^\infty
 \frac{(\psi_{k}(t; 3,4) -\frac{t}2)}{t^{2}} \frac{(\log t)^{j-1}}{(j-1)!}\left(1-\frac{\log t}{j}\right)\dd t.
 \end{equation*}
As above, we note that
\begin{align} \nonumber
&2 \int_1^\infty
 \frac{(\psi_{k}(t; 3,4) -\frac{t}2)}{t^{2}}\frac{(\log t)^{j-1}}{(j-1)!}\left(1-\frac{\log t}{j}\right)\dd t\nonumber \\
 &=\int_1^\infty
 \frac{(-\psi(t,\chi_{4})+\psi(t)-t+2(-1)^{k}\psi_{0}(\sqrt{t};3,4) - \log 2 \lfloor \tfrac{\log t}{\log 2}\rfloor)}{t^{2}}\frac{(\log t)^{j-1}}{(j-1)!}\left(1-\frac{\log t}{j}\right)\dd t.\label{new3-integrals}
 \end{align}
We again treat the four terms separately.
First, by Perron's formula, we write
\begin{align*}
-\int_1^\infty
 \frac{\psi(t,\chi_{4})}{t^{2}}\frac{(\log t)^{j-1}}{(j-1)!}\left(1-\frac{\log t}{j}\right)\dd t &=\frac{1}{2\pi i}\int_1^\infty\left(\int_{C(t)}
 \frac{L'}{L}(s,\chi_{4})\frac{t^{s-2}}{s}\dd s\right)\frac{(\log t)^{j-1}}{(j-1)!}\left(1-\frac{\log t}{j}\right)\dd t,
\end{align*}
where $C(t)$ is again a vertical path to the right of the line $\re(s)=1$. We move the contour to the left to a path $C$, inside a zero-free region of $L(s,\chi_{4})$ and  flip the order of integration as above.
Let us first compute, for $j=1$ and $\re(s)<1$,
\begin{align*}
I_{1}(s) :=& \int_{1}^{\infty} t^{s-2}\left(1-\log t\right) \; \dd t=\int_{0}^{\infty}e^{(s-1)u}(1-u)\dd u\\
=& \left[(1-u)\frac{e^{u(s-1)}}{s-1}\right]_{0}^{\infty} + \int_{0}^{\infty}\frac{e^{u(s-1)}}{s-1} \dd u= -\frac{1}{s-1}+\frac{1}{s-1}\left[\frac{e^{u(s-1)}}{s-1}\right]_{0}^{\infty}\\
=& -\frac{1}{s-1}-\frac{1}{(s-1)^2} = -\frac{s}{(s-1)^2}.
\end{align*}
For $j \geq 2$, we have
\begin{align*}
I_{j}(s) :=& \int_{1}^{\infty} t^{s-2}\frac{(\log t)^{j-1}}{(j-1)!}\left(1-\frac{\log t}{j}\right) \; \dd t = \int_{0}^{\infty} e^{(s-1)u}\frac{u^{j-1}}{(j-1)!}\left(1-\frac{u}{j}\right) \; du\\
=& \left[
\frac{e^{(s-1)u}}{s-1}\left(\frac{u^{j-1}}{(j-1)!}-\frac{u^j}{j!}\right)
\right]_{0}^{\infty} - \frac{1}{s-1}\int_{0}^{\infty}e^{(s-1)u}\left(\frac{u^{j-2}}{(j-2)!}-\frac{u^{j-1}}{(j-1)!}\right)du\\
=& -\frac{1}{s-1}I_{j-1}(s) = \left(-\frac{1}{s-1}\right)^{j-1}I_{1}(s) = \frac{(-1)^{j}s}{(s-1)^{j+1}}
\end{align*}
 by induction.
It follows that
\begin{equation*}
-\int_1^\infty
 \frac{\psi(t,\chi_{4})}{t^{2}}\frac{(\log t)^{j-1}}{(j-1)!}\left(1-\frac{\log t}{j}\right)\dd t =\frac{1}{2\pi i}\int_{C}
 \frac{L'}{L}(s,\chi_{4}) \frac{I_j(s)}{s} \dd s= \frac{1}{2\pi i}\int_{C}
 \frac{L'}{L}(s,\chi_{4})\frac{(-1)^{j}}{(s-1)^{j+1}} \dd s,
 \end{equation*}
so that upon shifting the integral and picking up the pole at $s=1$, we find that
\begin{equation}\label{c_j L at one}
-\int_1^\infty
 \frac{\psi(t,\chi_{4})}{t^{2}}\frac{(\log t)^{j-1}}{(j-1)!}\left(1-\frac{\log t}{j}\right)\dd t  =  \frac{(-1)^{j+1}}{j!}\left(\frac{L'}{L}\right)^{(j)}(1,\chi_{4}).
 \end{equation}

 Similarly, with $C$ as defined above, we find that
\begin{multline*}
\int_1^\infty
 \frac{(\psi(t)-t)}{t^{2}}\frac{(\log t)^{j-1}}{(j-1)!}\left(1-\frac{\log t}{j}\right)\dd t=\frac{1}{2\pi i}\int_1^\infty\left(\int_{(2)}
 - \frac{\zeta'}{\zeta}(s)\frac{t^{s-2}}{s} \dd s-\frac{1}{t}\right)\frac{(\log t)^{j-1}}{(j-1)!}\left(1-\frac{\log t}{j}\right)\dd t\\
=\frac{1}{2\pi i}\int_1^\infty\left(\int_{C}
 -\frac{\zeta'}{\zeta}(s)\frac{t^{s-2}}{s} \dd s+\frac{1}{t}-\frac{1}{t}\right)\frac{(\log t)^{j-1}}{(j-1)!}\left(1-\frac{\log t}{j}\right)\dd t\\
=-\frac{1}{2\pi i}\int_{C}
 \frac{\zeta'}{\zeta}(s)I_{j}(s)\frac{\dd s}{s} =-\frac{1}{2\pi i}\int_{C}
 \frac{\zeta'}{\zeta}(s)\frac{(-1)^{j}}{(s-1)^{j+1}} \dd s.
 \end{multline*}
Upon writing
\[\zeta(s) = \frac{f(s)}{s-1},\]
where
\begin{equation}\label{zeta_holomorphic}
f(s) := 1+\sum_{n=0}^{\infty}\frac{(-1)^n}{n!}\gamma_{n}(s-1)^{n+1} = \sum_{k=0}^{\infty}\frac{f^{(k)}(1)}{k!}(s-1)^{k}
\end{equation}
is a holomorphic function, we find that
\begin{align*}
\frac{\zeta'}{\zeta}(s) &= \frac{f'}{f}(s)-\frac{1}{s-1}=-\frac{1}{s-1}+\sum_{k=0}^{\infty}\left(\frac{f'}{f}\right)^{(k)}(1)\frac{(s-1)^{k}}{k!}
\end{align*}
and therefore that 
\begin{align*}
-\frac{1}{2\pi i}\int_{C}
 \frac{\zeta'}{\zeta}(s)\frac{(-1)^{j}}{(s-1)^{j+1}}\dd s&=-\frac{1}{2\pi i}\int_{C}\left(-\frac{1}{(s-1)^{j+2}}+\sum_{k=0}^{\infty}\left(\frac{f'}{f}
 \right)^{(k)}(1)\frac{(s-1)^{k-j-1}}{k!}\right)(-1)^{j}\dd s\\
 &=\frac{(-1)^{j}}{j!}\left(\frac{f'}{f}\right)^{(j)}(1).
\end{align*}

We next note that
\begin{align*}
2(-1)^{k}\int_1^\infty
 &\frac{\psi_{0}(\sqrt{t};3,4)}{t^{2}}\frac{(\log t)^{j-1}}{(j-1)!}\left(1-\frac{\log t}{j}\right)\dd t \\
 &= \frac{(-1)^{k+1}}{2\pi i}\int_1^\infty2 \left(\int_{C}
\frac{\zeta'_{4,3}(2s)}{\zeta_{4,3}(2s)} \frac{t^{s-2}}{s} \dd s\right)\frac{(\log t)^{j-1}}{(j-1)!}\left(1-\frac{\log t}{j}\right)\dd t\\
&= \frac{(-1)^{k+1}}{2\pi i} 
\int_{C} 2\frac{\zeta'_{4,3}(2s)}{\zeta_{4,3}(2s)}
\frac{ I_j(s)}{s} \dd s
= (-1)^{k+j} \frac{2^{j+1}}{j!} \left(\frac{\zeta'_{4,3}}{\zeta_{4,3}}\right)^{(j)}(2).
 \end{align*}
Finally, upon writing $B(s) := \sum_{n\geq 0} \tfrac{\log2}{2^{ns}} $, we similarly find that
\begin{align*}
	- \log 2 \int_1^{\infty} \lfloor \tfrac{\log t}{\log 2}\rfloor \frac{(\log t)^{j-1}}{(j-1)!}\left(1-\frac{\log t}{j}\right) \frac{\dd t}{t^2}
 &= -\frac{1}{2\pi i}\int_1^\infty\left(\int_{C}
B(s)\frac{t^{s-2}}{s} \dd s\right)\frac{(\log t)^{j-1}}{(j-1)!}\left(1-\frac{\log t}{j}\right)\dd t\\
&= -\frac{1}{2\pi i} 
\int_{C}
\frac{ B(s)I_j(s)}{s} \dd s
= (-1)^{j}\frac{B^{(j)}(1)}{j!} = \frac{(\log 2)^{j+1}}{j!}\mathrm{Li}_{-j}(\tfrac12),
\end{align*}
from which the lemma follows.
\end{proof}

\section{Computing $\cD(K; \phi, \cF^{\alpha}_d)$}\label{section_small_support}

We now compute $\cD(K; \phi, \cF^{\alpha}_d)$ by averaging $\cD(\phi, \xi_{d,k})$ over appropriate $1 \leq k \leq K$.

\begin{prop}\label{inert-primes-even} Let $d$ be an odd square-free integer, $\alpha \in \Z/8\Z$, and
let $\phi$ be an even Schwartz function such that $\widehat{\phi}$ is compactly supported. Then for any $J \geq 1$, one has as $K \rightarrow \infty$,
\begin{align*}
\mathcal{D}(K; \phi, \cF_{d}^{\alpha}) &= 
\widehat{\phi}(0)+\frac{(-1)^{\alpha+1}}{2} \int_{\R}\widehat{\phi}(u) du
+\sum_{m=1}^{J}\frac{C_{m}(d,\alpha, \phi)}{(\log K M_{d,\alpha})^{m}}+\frac{8}{K}  \sum_{\substack{1\leq k \leq K\\k \equiv \alpha \bmod 8}}  U_{\textnormal{split}}(\phi,d, k)\\
&\phantom{=}+ O_{J}\left(\frac{1+\log|d|}{(\log K M_{d,\alpha})^{J+1}}\right),
\end{align*}
with
\begin{align} \label{define-Cm}
C_{m}(d,\alpha, \phi)&:= (m-1)! \sum_{\substack{j = 0 \\ j \textnormal{ even}}}^{m-1} c_{j}(d,\alpha) \frac{\widehat{\phi}^{(j)}(0)}{j!},
\end{align}
where for $\alpha \equiv 0 \mod 4$ and $j \geq 0$,
\begin{align*}
c_j(d, \alpha) = c_j(\alpha) :=& -\delta_{0}(j)\log{2\pi}+B_{j}+ \frac{(\log 2)^{j+1}}{j!}\mathrm{Li}_{-j}(\tfrac12)
  +(-1)^{j} \frac{2^{j+1}}{j!} \left(\frac{\zeta'_{4,3}}{\zeta_{4,3}}\right)^{(j)}(2)\\
  &-\frac{2}{j!}\left(\frac{\log 2}{2}\right)^{j+1}\textnormal{Li}_{-j}\left(\frac{(-1)^{\frac{\alpha}{4}}}{\sqrt{2}}\right), \nonumber
\end{align*}
where  for $\alpha \not\equiv 0 \mod 4$ and $j \geq 0$,
\begin{align*}
c_j(d, \alpha) &= -\delta_{0}(j)\log{2\pi} + (-1)^{\alpha}\left(B_{j} + \frac{(\log 2)^{j+1}}{j!}\mathrm{Li}_{-j}(\tfrac12)\right) +(-1)^{j} \frac{2^{j+1}}{j!} \left(\frac{\zeta'_{4,3}}{\zeta_{4,3}}\right)^{(j)}(2)
\\ & \quad + \frac{2}{j!}\sum_{\substack{p \equiv 3 \bmod 4 \\ p \mid d}}(\log p)^{j+1}\textnormal{Li}_{-j}\left(\frac{(-1)^{\alpha}}{p}\right),\nonumber
\end{align*}
and where
\begin{align*}
B_{j}:=&\frac{(-1)^{j}}{j!}\left(\left(\frac{f'}{f}\right)^{(j)}(1)-\left(\frac{L'}{L}\right)^{(j)}(1,\chi_{4})\right), \qquad \qquad f(s):=(s-1)\zeta(s).
\end{align*}
%
\end{prop}

\begin{proof} Recall by~\eqref{One_level_K} that
$$\cD(K; \phi, \cF_d^{\alpha}) = \frac{8}{K} \sum_{\substack{1 \leq k \leq K\\k \equiv \alpha \bmod 8}} \cD(\phi, \xi_{d,k}),$$
where the estimation for $\cD(\phi, \xi_{d,k})$ is given in Proposition~\ref{local_OLD}.
Let $j\in \mathbb N$.
Noting that $c_j(d,k)$ and $\n$ depend only on the congruence class of $k$ modulo $8$, we fix an appropriate representative $\alpha \in \{1,\dots 8\}$.  It follows that for fixed $c \geq 1$,
\[\sum_{\substack{1 \leq k \leq K\\ k \equiv \alpha \bmod 8}}\frac1{(\log(ck))^{j+1}} = \sum_{0 \leq \ell \leq \left \lfloor \frac{K-\alpha}{8}\right \rfloor}\frac{1}{(\log (c(8 \ell + \alpha)))^{j+1} }
=  \sum_{1 \leq \ell \leq \left \lfloor \frac{K-\alpha}{8}\right \rfloor}\frac{1}{(\log(8c \ell))^{j+1} } + O(1),\]
where, by \eqref{conductor}, we need only consider the case $c =1$ when $\alpha \in \{4,8\}$. The last equality above follows from
\begin{align*}
 \frac1{(\log(8 c \ell))^{j+1}} - \frac{1}{(\log (8 c\ell + c \alpha))^{j+1} } &= \frac{(\log (8c \ell + c \alpha))^{j+1} - (\log (8 c\ell))^{j+1}}{(\log (8 c \ell + c \alpha))^{j+1} (\log (8c \ell ))^{j+1}  }\\
&\ll \frac{ \left( \log(8c\ell+c\alpha) - \log{(8c\ell)} \right) (\log (8c \ell))^{j}}{(\log (8c \ell ))^{2j+2}  } \ll \frac{1}{\ell \,  (\log (8c \ell))^{j+2}}
\end{align*}
and the sum over $\ell$ converges.
By the Euler--Maclaurin formula,
\begin{align*}
\sum_{1 \leq \ell \leq \left \lfloor \frac{K-\alpha}{8}\right \rfloor}\frac{1}{(\log(8c \ell))^{j+1} } & = \int_1^{  \frac{K}{8} } \frac{\dd t}{(\log(8ct))^{j+1}} + O(1) = \frac{1}{8c}\int_{8c}^{ cK } \frac{\dd t}{(\log t)^{j+1}} + O(1) \\
&= \frac1{8c} \frac{cK}{(\log(cK))^{j+1}}\sum_{\ell = 0}^{J-j} \frac{(j+\ell)!}{j! (\log(cK))^{\ell}} + O_J\Big(\frac{K}{(\log(cK))^{J+2}}\Big).	
\end{align*}

By  Proposition \ref{local_OLD}, and upon noting in particular that $c_{j}(d,k) = c_{j}(d,\alpha)$ for any even $j$ and $k \equiv \alpha \bmod 8$, we find that
\begin{align*}
\mathcal{D}(K; \phi, \cF_{d}^{\alpha}) 
=&~ \widehat{\phi}(0) + \frac{(-1)^{\alpha+1}}{2} \int_{\R}\widehat{\phi}(u) du
 +
\sum_{\substack{j=0\\j \textnormal{ even}}}^{J-1} \frac{c_{j}(d,\alpha) \widehat{\phi}^{(j)}(0)}{(\log(K M_{d,\alpha}))^{j+1}}\sum_{\ell = 0}^{J-j} \frac{(j+\ell)!}{j! (\log(K M_{d,\alpha}))^{\ell}} \\& +  \frac{8}{K}\sum_{\substack{1 \leq k \leq K\\k \equiv \alpha \bmod 8}}U_{\textnormal{split}}(\phi,d,k)+O_{J}\left(\frac{1+\log |d|}{(\log (K M_{d,\alpha}))^{J+1}}\right).
\end{align*}
Setting $m=j+\ell+1$, we write the $j,\ell$-sum as
\begin{align*}
\sum_{\substack{j=0\\j \textnormal{ even}}}^{J-1} \frac{c_{j}(d,\alpha) \widehat{\phi}^{(j)}(0)}{j!} \sum_{m = j+1}^{J+1} \frac{(m-1)!}{(\log(K M_{d,\alpha}))^{m}} = \sum_{m=1}^{J+1}  \frac{(m-1)!}{(\log(K M_{d,\alpha}))^{m}}
\sum_{\substack{j=0\\j \textnormal{ even }}}^{\min{(m-1, J-1)}} \frac{c_{j}(d,\alpha) \widehat{\phi}^{(j)}(0)}{j!}
\end{align*}
and the formula for $\mathcal{D}(K; \phi, \cF_{d}^{\alpha})$ follows by defining $C_{m}(d,\alpha, \phi)$ as in \eqref{define-Cm}. Finally, we recall that for $j \geq 1$,
\begin{align*}
c_{0}(d,k) &= - \log 2\pi+c_{0,\textnormal{ram}}(k)+c_{0,\textnormal{inert}}(k)+c_{0,\textnormal{inert},d}(k)\\
	c_{j}(d,k) &=c_{j,\textnormal{ram}}(k)+c_{j,\textnormal{inert}}(k)+c_{j,\textnormal{inert},d}(k),
\end{align*}
so that by the expressions \eqref{ram_lower_order}, \eqref{c_inert_D}, \eqref{c_0_inert explicit} and \eqref{c_j expressed middle} above, we obtain the desired formulas in the statement of the proposition. 
\end{proof}

Theorem \ref{thm-main} now follows from Proposition \ref{inert-primes-even} provided that we can bound the contribution of the split primes for test-functions $\phi$ such that $\text{supp}(\widehat{\phi})\subset (-1,1).$  
We begin by first proving the following weaker result, which, combined with  Proposition \ref{inert-primes-even}, provides an asymptotic for $\mathcal{D}(K;\phi,\mathcal{F}_{d}^{\alpha})$ when $\text{supp}(\widehat\phi) \subset (-\frac12, \frac12)$.

\begin{lem} \label{lemma-split-primes} Let $d$ be an odd square-free integer, $\alpha \in \Z/8\Z$, and
let $\phi$ be an even Schwartz function such that $\mathrm{supp}(\widehat{\phi}) \subset (-\nu, \nu)$. Then as $K \rightarrow \infty$, we have
$$\frac{8}{K}  \sum_{\substack{1 \leq k \leq K\\k \equiv \alpha \bmod 8}}  U_{\textnormal{split}}(\phi,d,k) = O(d^{2\nu}K^{2\nu -1}).$$
\end{lem}
\begin{proof}
Summing~\eqref{U-split} over $k$, we get
\begin{align} \nonumber
&\frac{8}{K}  \sum_{\substack{1 \leq k \leq K\\k \equiv \alpha \bmod 8}} U_{\text{split}}(\phi,d, k) \\ \label{eq:split_prime_sum} &= -\frac{8}{K}  \sum_{\substack{p \equiv 1 \bmod4\\
		(p) = \p \overline{\p}\\n \geq 1}} \frac{ \log{p}}{p^{\frac{n}{2}}}  \sum_{\substack{1 \leq k \leq K\\k \equiv \alpha \bmod 8}} \left( \xi_{d,k}^n(\fp) +  \overline{\xi}_{d,k}^n(\fp) \right) 
\widehat\phi \left( \frac{n \log{p}}{2 \log( k M_{d,\alpha})} \right) \frac{1}{\log{(k  M_{d, \alpha})}}.
\end{align}

We have
$$
\sum_{\substack{1 \leq k \leq K\\k \equiv \alpha \bmod 8}} \left( \xi_{d,k}^n(\fp) +  \overline{\xi}_{d,k}^n(\fp) \right) = \sum_{\substack{1\leq k \leq K\\k \equiv \alpha \bmod 8}} (e^{i k n \theta_{d,p}}+e^{-i k n \theta_{d,p}}) 
$$
with $\theta_{d,p}$ as defined by \eqref{def-angle}. We define
\begin{align*}
D_{K,\alpha}(x) :=  \sum_{\substack{1\leq k \leq K\\k \equiv \alpha \bmod 8}} (e^{i k x}+e^{-i kx}) = 2 \sum_{\substack{1\leq k \leq K\\k \equiv \alpha \bmod 8}} \cos(k x).
\end{align*}
We note that
$$ 1 + \sum_{\alpha \in \Z/8\Z}D_{K,\alpha}(x)= D_K(x)$$
where
$$D_K(x)  :=\sum_{k=-K}^K e^{i k x} = \frac{\textnormal{sin}\left( (K+\frac12) x \right)}{\sin(\frac{x}{2})}$$ is the \textit{Dirichlet kernel} at $x$.
Fixing an appropriate representative for $\alpha \in \{1,\dots 8\}$, we observe that 
\begin{align}\label{Dirichlet_kernel_bound}
\begin{split}
D_{K,\alpha}(x) &=\sum_{0 \leq \ell \leq \lfloor\frac{K-\alpha}{8}\rfloor}(e^{ix(8\ell + \alpha)}+e^{-ix(8\ell + \alpha)})\\
&=\frac{e^{i\alpha x}-e^{ix(8\lfloor\frac{K-\alpha}{8}\rfloor + \alpha+8)}}{1-e^{8ix}}+\frac{e^{-i\alpha x}-e^{-ix(8\lfloor\frac{K-\alpha}{8}\rfloor + \alpha+8)}}{1-e^{-8ix}}\\
&=\frac{-e^{i(\alpha -4)x}+e^{ix(8\lfloor\frac{K-\alpha}{8}\rfloor + \alpha+4)}+e^{-i(\alpha-4) x}-e^{-ix(8\lfloor\frac{K-\alpha}{8}\rfloor + \alpha+4)}}{2i \sin (4x) }\\
&=\frac{\sin((8\lfloor\frac{K-\alpha}{8}\rfloor + \alpha+4)x)-\sin((\alpha -4)x)}{\sin (4x) } \ll \frac{1}{\lvert\sin(4x)\rvert}.
\end{split}
\end{align}                                                                                                                                              
Using the bound \eqref{Dirichlet_kernel_bound} and partial summation, we get 
\begin{align*}
&\sum_{\substack{1 \leq k \leq K\\k \equiv \alpha \bmod 8}} \left( \xi_{d,k}^n(\fp) +  \overline{\xi}_{d,k}^n(\fp) \right) 
\widehat\phi \left( \frac{n \log{p}}{2 \log( k M_{d,\alpha})} \right) \frac{1}{\log{(k  M_{d, \alpha})}} \\&= 
D_{K, \alpha}(n \theta_{d,p}) \widehat\phi \left( \frac{n \log{p}}{2 \log( K M_{d,\alpha})} \right)  \frac{1}{\log{(K  M_{d, \alpha})}}
- \int_{1}^{K} D_{t, \alpha}(n \theta_{d,p}) \frac{\dd}{\dd t} \Bigg( \frac{\widehat\phi \left( \frac{n \log{p}}{2 \log( t M_{d,\alpha})} \right)}{\log{(t  M_{d, \alpha})}} \Bigg)  \dd t + O(1) \\
&\ll  \frac{1}{\lvert\sin(4n\theta_{d,p})\rvert}  +   \frac{1}{\lvert\sin(4n\theta_{d,p})\rvert} \int_{1}^{K}  \Bigg\lvert\frac{\dd}{\dd t} \Bigg( \frac{\widehat\phi \left( \frac{n \log{p}}{2 \log( t M_{d,\alpha})} \right)}{\log{(t  M_{d, \alpha})}} \Bigg) \Bigg\rvert\dd t \ll \frac{1}{\lvert\sin(4n\theta_{d,p})\rvert},
\end{align*}
since the integral converges. Since $\mathrm{supp}(\widehat{\phi}) \subset (-\nu, \nu)$, it then follows from \eqref{eq:split_prime_sum} that
\begin{align} \label{before-bounding-sin}
\frac{8}{K}  \sum_{\substack{1 \leq k \leq K\\k \equiv \alpha \bmod 8}} U_{\text{split}}(\phi,d, k) \ll \frac{1}{K}\sum_{\substack{p^{n} \leq (KM_{d,\alpha})^{2\nu}\\ p \equiv 1 \bmod 4}} \frac{\log{p}}{{p^{\frac{n}2}}\lvert\sin (4n \theta_{d,p})\rvert}.
\end{align}

In order to give a lower bound on $\lvert\sin (4n \theta_{d,p})\rvert$, we first demonstrate that $\theta_{d,p} \not \in \pi \Q$ for any $p \equiv 1 \bmod{4}$.  Observe that since $z_{d,p} \in \Z[i]$, we have that $\tan(\theta_{d,p}) \in \mathbb Q$. Moreover, Niven's theorem states that $\{a \in \Q: \tan(a \pi) \in \Q\} \subset \tfrac{1}{4}\Z$.  Since we know that
$\theta_{d,p} \not \in \frac{\pi}{4}\Z$ for any $p \equiv 1 \bmod{4}$, it follows that $\theta_{d,p} \not \in \pi \Q$, as desired. In particular, when $p \equiv 1 \bmod{4}$, the point $z_{d,p}^{n} = p^{n/2} e^{i n \theta_{d,p}} \in \Z[i]$ is a lattice point off the lines $x=0$, $y=0$ and $x=\pm y$ of the complex plane, i.e.
\[|\textrm{Im}(z^{n}_{d,p})|,\;|\textrm{Re}(z^{n}_{d,p})|,\; |\textrm{Im}(z^{n}_{d,p})\pm \textrm{Re}(z^{n}_{d,p})| \geq 1.\]  
Using
\[\sin(4x)  = 4 \sin x \cos x(\cos x-\sin x)(\cos x+\sin x),\]
we find that for, say, $0 \leq n \theta_{d,p} \leq \tfrac{\pi}{6}$ and $\tfrac{5\pi}{6} \leq n \theta_{d,p} \leq \pi$,
\[
\lvert\sin (4n \theta_{d,p})\rvert \geq  4 \tfrac{|\textnormal{Im}(z^{n}_{d,p})|}{|z_{d,p}^{n}|}(\cos \tfrac{\pi}{6})^{2}\left(\cos \tfrac{\pi}{6} - \sin \tfrac{\pi}{6}\right) \geq p^{-\frac{n}{2}}.
\]
For $\tfrac{\pi}{6} \leq n \theta_{d,p} \leq \tfrac{\pi}{3}$ and $\tfrac{2\pi}{3} \leq n \theta_{d,p} \leq \tfrac{5\pi}{6}$ we similarly find that
\[
\lvert\sin (4n \theta_{d,p})\rvert \geq  4 \sin \tfrac{\pi}{6} \cos \tfrac{\pi}{3}\tfrac{|\textnormal{Re}(z^{n}_{d,p})\pm \textnormal{Im}(z^{n}_{d,p})|}{|z_{d,p}^{n}|}\left(\cos \tfrac{\pi}{3}+\sin \tfrac{\pi}{6}\right) \geq p^{-\frac{n}{2}},\]
while for $\tfrac{\pi}{3} \leq n \theta_{d,p}\leq \tfrac{2\pi}{3}$,
\[
\lvert\sin (4n \theta_{d,p})\rvert \geq  4 \sin \tfrac{\pi}{3} \tfrac{|\textnormal{Re}(z^{n}_{d,p})|}{|z_{d,p}^{n}|}(\sin^{2} \tfrac {\pi}{3}-\cos^{2} \tfrac{\pi}{3}) \geq p^{-\frac{n}{2}}.\]
Since $x\mapsto \sin(4x)$ is $\frac{\pi}{2}$-periodic, we have proven that for any $p$,
\begin{equation}\label{lowerbound}
 \left|\sin (4n \theta_{d,p})\right|  \geq {p^{-\frac{n}{2}}}.
\end{equation}
Inserting into \eqref{before-bounding-sin} and using the prime number theorem in arithmetic progressions,  we have
$$
\frac{8}{K} \sum_{\substack{1 \leq k \leq K\\ k \equiv \alpha \bmod8}} U_{\text{split}}(\phi,d,k) \ll d^{2\nu} K^{2\nu -1},
$$
as desired.
\end{proof}

\begin{rem} In the following, we will use the function $\lVert \cdot \rVert_{2\pi}$ to denote the distance to the nearest integer multiple of $2\pi$.
Observe that	since $\frac12 \lVert 2x \rVert_{2 \pi}  \geq \lvert \sin x\rvert$, we have by \eqref{lowerbound} that
\begin{align} \label{useful-bound}
\lVert 8 n \theta_{d,p} \rVert_{2 \pi} \geq 2p^{-\frac{n}{2}}.
\end{align}
We will use it often in the next section.
\end{rem}

\section{Proof of Theorem~\ref{thm-main}} \label{extended-range}

In order to extend the admissible support in Lemma \ref{lemma-split-primes} to $(-1, 1)$ (and prove Theorem~\ref{thm-main}), we improve upon our bound of the contribution of the split primes
by averaging the possible values of $\theta_{d,p}$, and not simply bounding by the worst value.   
Let
\begin{eqnarray}\label{def_Spn}
S_{p^{n}} := \Big\{ z = r(z) e^{i \theta(z)} \in \C^{\times} &:  n \theta_{d,p} - \frac{1}{8} p^{-\frac n 2} \leq \theta(z) \leq n \theta_{d,p} + \frac{1}{8} p^{-\frac n 2} \\ \nonumber
& \text{and }\;  p^{\frac{n}{2}} - \frac{1}{4}  \leq r(z) \leq p^{\frac{n}{2}} + \frac{1}{4} \Big \}.
\end{eqnarray}
Note that the area of $S_{p^n}$ is equal to
\begin{align*}
A(S_{p^n}):=\iint_{S_{p^n}}r \dd r \dd\theta:&= \int_{\theta_{\textnormal{min}}}^{\theta_{\textnormal{max}}}\int_{r_{\textnormal{min}}}^{r_{\textnormal{max}}}r \dd r \dd\theta=
\frac{1}{2} \left(r_{\textnormal{max}}^{2}-r_{\textnormal{min}}^{2}\right)  \left( \theta_{\text {max}} - \theta_{\text {min}} \right)\\
& = \frac{1}{2}\left(\left( p^{\frac n 2}+\tfrac{1}{4}\right)^{2}-\left( p^{\frac n 2}-\tfrac{1}{4}\right)^{2}\right)\left(\tfrac{1}{4} p^{-\frac n 2}\right)=\frac{1}{8}.
\end{align*}

\begin{lem}\label{eq:area_approx} Let $d$ be an odd square-free integer,  $p \equiv 1 \bmod 4$ and $n$ a positive integer. Then if $ p^{-\frac{n}{2}}=o(\lVert  8n\theta_{d,p}\rVert_{2\pi}),$ we have 
\[\frac{\log p}{p^{\frac{n}{2}}\lVert  8n\theta_{d,p} \rVert_{2\pi}}=\frac{2}{n} \frac{1}{A(S_{p^{n}})}
\iint_{S_{p^n}} \frac{\log{r}}{\lVert  8\theta \rVert_{2\pi}} \dd r  \dd\theta\left(1+O\left(\frac{1}{p^{\frac{n}{2}}\lVert 8n\theta_{d,p}\rVert_{2\pi}}\right)\right).
\]
Unconditionally, we moreover find that
\[\frac{\log p}{p^{\frac{n}{2}}\lVert  8n\theta_{d,p} \rVert_{2\pi}} \ll \frac{1}{n} \iint_{S_{p^n}} \frac{\log{r}}{\lVert  8\theta \rVert_{2\pi}}  \dd r \dd\theta.\]
\end{lem}

\begin{rem} 
Note that by \eqref{useful-bound}, for any $re^{i \theta} \in S_{p^{n}}$, and for all $m \in \Z$,

\begin{align*}
\begin{split}
|8 \theta - 2m \pi| &\geq |8n \theta_{d,p} - 2m \pi|-|8n \theta_{d,p} -8 \theta|\\
&\geq \min_{m \in \Z}|8n \theta_{d,p} - 2m \pi|-|8n \theta_{d,p} -8 \theta| =\lVert 8n \theta_{d,p} \rVert_{2\pi} - |8n \theta_{d,p} -8 \theta|\\
& \geq 2p^{-\frac{n}{2}} - p^{-\frac{n}{2}} = p^{-\frac{n}{2}}.
\end{split}
\end{align*}
Thus
\begin{equation}\label{theta_away_2pi}
\lVert 8\theta \rVert_{2\pi} \geq p^{-\frac{n}{2}} > 0,
\end{equation}
and therefore the integral in Lemma \ref{eq:area_approx} is indeed well-defined.
\end{rem}

\begin{proof}[Proof of Lemma \ref{eq:area_approx}]
Let
	\[f(x) := \frac{\log x}{x}\hspace{5mm} \textnormal{and} \hspace{5mm} f'(x) = \frac{1-\log x}{x^{2}},\]
so that for  $re^{i \theta} \in S_{p^{n}}$, we have
\begin{align*}
\left |\frac{\log{r}}{r}-\frac{n}{2}\frac{\log{p}}{p^{\frac{n}{2}}}\right| \ll |r-p^{\frac{n}{2}}|\cdot \sup_{t\in [\min(r,p^{\frac{n}2}),\max(r,p^{\frac{n}2})]}\lvert f'(t)\rvert \ll \left|\frac{\log (p^{\frac{n}{2}}-\frac{1}{4})}{(p^{\frac{n}{2}}-\frac{1}{4})^{2}}\right| \ll \frac{n \log p}{p^n},
\end{align*}
and
\begin{align*}
\iint_{S_{p^n}} \frac{1}{\lVert  8\theta \rVert_{2\pi}} \frac{\log{r}}{r}r \dd r \dd\theta&= 
\int_{\theta_{\textnormal{min}}}^{\theta_{\textnormal{max}}} \left(\frac{\dd\theta}{\lVert 8 \theta \rVert_{2\pi}}
\right)\int_{r_{\textnormal{min}}}^{r_{\textnormal{max}}}\left(\frac{n}{2}\frac{\log p}{p^{\frac n 2}}+O\left(\frac{n \log p}{p^{n}}\right)
\right)
r \dd r.
\end{align*}
We begin by computing the $r$-integral.  We have
\begin{align*}
\int_{r_{\textnormal{min}}}^{r_{\textnormal{max}}}\frac{\log r}{r} r \dd r 
&= \left(\frac{n}{2}\frac{\log p}{p^{\frac n 2}}+O\left(\frac{n\log p}{p^{n}}\right)
\right)\int_{p^{\frac n 2}-\frac{1}{4}}^{p^{\frac n 2}+\frac{1}{4}}r \dd r
=  \frac{n \log p}{4}+O\left(\frac{n \log p}{p^{\frac{n}{2}}}\right),
\end{align*}
and, in particular,
\begin{equation}\label{r_bound_careful_constant}
\frac{n \log p}{4} \ll \int_{r_{\textnormal{min}}}^{r_{\textnormal{max}}}\frac{\log r}{r} r \dd r.
\end{equation}

We now compute the $\theta$-integral.  Up to translation by $\Z \tfrac{\pi}{4}$, we may assume that $[\theta_{\textnormal{min}},\theta_{\textnormal{max}}] \subset \left(0,\tfrac{\pi}{4}\right)$.  We consider several cases.  First, suppose $\theta_{\textnormal{max}} \leq \tfrac{\pi}{8}$.  Then
\begin{align*}
\int_{\theta_{\textnormal{min}}}^{\theta_{\textnormal{max}}} \frac{\dd\theta}{\lVert  8\theta \rVert_{2\pi}} &= \int_{\theta_{\textnormal{min}}}^{\theta_{\textnormal{max}}} \frac{\dd\theta}{8\theta}
= \frac{1}{8}\left(\log \left(n\theta_{d,p}+\frac{1}{8p^{\frac{n}{2}}}\right)-\log \left(n\theta_{d,p}-\frac{1}{8p^{\frac{n}{2}}}\right) \right)\\
  &= \frac{1}{8}\left(\log \left(1+\frac{1}{8n\theta_{d,p}p^{\frac{n}{2}}}\right)-\log \left(1-\frac{1}{8n\theta_{d,p}p^{\frac{n}{2}}}\right) \right)\\
    &= \sum_{\ell=1}^{\infty}\frac{(-1)^{\ell+1
}+1}{8\ell}(8n\theta_{d,p}p^{\frac{n}{2}})^{-\ell}= \frac{1}{32n\theta_{d,p}p^{\frac{n}{2}}}\left(1+O\left((n\theta_{d,p}p^{\frac{n}{2}})^{-2}\right)\right)\\
     &= \frac{1}{4p^{\frac{n}{2}}\lVert 8n\theta_{d,p}\rVert_{2\pi}}\left(1+O\left((p^{\frac{n}{2}}\lVert 8n\theta_{d,p}\rVert_{2\pi})^{-2}\right)\right).
\end{align*}
Since the coefficients of the Taylor series are all non-negative, the implied constant is positive, enabling us to conclude that
\begin{equation}\label{theta_bound_careful_constant}
\frac{1}{4p^{\frac{n}{2}}\lVert 8n\theta_{d,p}\rVert_{2\pi}} \ll  \int_{\theta_{\textnormal{min}}}^{\theta_{\textnormal{max}}} \frac{\dd\theta}{\lVert  8\theta \rVert_{2\pi}}.
\end{equation}
Next, suppose $\theta_{\textnormal{min}} \geq \tfrac{\pi}{8}$.  Then
\begin{align*}
\int_{\theta_{\textnormal{min}}}^{\theta_{\textnormal{max}}} \frac{\dd\theta}{\lVert  8\theta \rVert_{2\pi}} &= \int_{\theta_{\textnormal{min}}}^{\theta_{\textnormal{max}}} \frac{\dd\theta}{2\pi - 8\theta}
= \frac{1}{8}\left(\log \left(\pi - 4n\theta_{d,p}+\frac{1}{2p^{\frac{n}{2}}}\right)-\log \left(\pi - 4n\theta_{d,p}-\frac{1}{2p^{\frac{n}{2}}}\right) \right)\\
    &= \sum_{\ell=1}^{\infty}\frac{(-1)^{\ell+1
}+1}{8\ell}(2(\pi - 4n\theta_{d,p})p^{\frac{n}{2}})^{-\ell}= \frac{1}{4p^{\frac{n}{2}}\lVert 8n\theta_{d,p}\rVert_{2\pi}}\left(1+O\left((p^{\frac{n}{2}}\lVert 8n\theta_{d,p}\rVert_{2\pi})^{-2}\right)\right),
\end{align*}
and where again we may conclude \eqref{theta_bound_careful_constant}, upon noting that the implicit constant is positive.  Finally, we consider the case $\theta_{\textnormal{min}} < \tfrac{\pi}{8} < \theta_{\textnormal{max}}$, in which case $|n \theta_{d,p} - \tfrac{\pi}{8}| < \frac{1}{8}p^{-\frac{n}{2}}$.  Then
\begin{align*}
\int_{\theta_{\min}}^{\theta_{\max}} \frac{\dd\theta}{\lVert  8\theta \rVert_{2\pi}} &= \int_{\theta_{\min}}^{\frac{\pi}{8}} \frac{\dd\theta}{8\theta}+\int_{\frac{\pi}{8}}^{\theta_{\max}} \frac{\dd\theta}{2\pi - 8\theta}\\
&=\frac{1}{8}\left(\log \left(\frac{\pi}{8}\right)-\log \left(n\theta_{d,p}-\frac{1}{8p^{\frac{n}{2}}}\right)+\log \left(\frac{\pi}{8}\right)-\log \left(\frac{\pi}{4} - n\theta_{d,p}-\frac{1}{8p^{\frac{n}{2}}}\right)  \right)\\
&=\frac{1}{8}\Bigg(2\log \left(\frac{\pi}{2}\right)-\log \left(4 n\theta_{d,p}\right)-\log \left(\pi - 4n\theta_{d,p}\right)\\
&\phantom{==}-\log \left(1-\frac{1}{8 n\theta_{d,p}p^{\frac{n}{2}}}\right)-\log \left(1-\frac{1}{2(\pi - 4n\theta_{d,p})p^{\frac{n}{2}}}\right)\Bigg).
\end{align*}

Note that
\begin{align*}
2\log \left(\frac{\pi}{2}\right)-\log \left(4n\theta_{d,p}\right)-\log \left(\pi - 4n\theta_{d,p}\right) &=- \log \left(1 - \frac{\pi/2-4n\theta_{d,p}}{\pi/2}\right)- \log \left(1 + \frac{\pi/2-4n\theta_{d,p}}{\pi/2}\right)\\
&=O\left(\left(\frac{\pi}{2}-4n\theta_{d,p}\right)^{2}\right) =O\left(p^{-n}\right),
\end{align*}
as well as that
\begin{align*}
-&\log \left(1-\frac{1}{8n\theta_{d,p}p^{\frac{n}{2}}}\right)-\log \left(1-\frac{1}{2(\pi - 4n\theta_{d,p})p^{\frac{n}{2}}}\right)\\
&=\sum_{\ell=1}^{\infty}\frac{1}{\ell}(8n\theta_{d,p}p^{\frac{n}{2}})^{-\ell} +\sum_{\ell=1}^{\infty}\frac{1}{\ell}(2(\pi - 4n\theta_{d,p})p^{\frac{n}{2}})^{-\ell}\\
&=\frac{1}{p^{\frac{n}{2}}}\left(
\frac{1}{8n\theta_{d,p}}+\frac{1}{2\pi - 8n\theta_{d,p}}\right)+O(p^{-n}).
\end{align*}
This last term can be approximated by noting that
\begin{align*}
&\frac{1}{8n\theta_{d,p}}+\frac{1}{2\pi - 8n\theta_{d,p}}= \frac{1}{\lVert 8n\theta_{d,p}\rVert_{2\pi}+8n\theta_{d,p}-\lVert 8n\theta_{d,p}\rVert_{2\pi}}+\frac{1}{\lVert 8n\theta_{d,p}\rVert_{2\pi}+2\pi - 8n\theta_{d,p}-\lVert 8n\theta_{d,p}\rVert_{2\pi}}\\
&= \frac{2}{\lVert 8n\theta_{d,p}\rVert_{2\pi}}\Big(1+O\big(\lVert 8n\theta_{d,p}\rVert_{2\pi}^{-1}\big(|2\pi - 8n\theta_{d,p}-\lVert 8n\theta_{d,p}\rVert_{2\pi}|+|8n\theta_{d,p}-\lVert 8n\theta_{d,p}\rVert_{2\pi}|\big)\big)\Big)\\
&= \frac{2}{\lVert 8n\theta_{d,p}\rVert_{2\pi}}\Big(1+O\big((p^{\frac{n}{2}}\lVert  8n\theta_{d,p} \rVert_{2\pi} )^{-1}\big)\Big),
\end{align*}
where for the last equality, we note that one of the terms in the innermost parentheses vanishes while the other is $< 2p^{-\frac{n}2}$, by hypothesis. 
Note also that $\lVert  8n\theta_{d,p} \rVert_{2\pi}^{-1} \asymp 1$ in this range.  To summarize,
\[\int_{\theta_{\textnormal{min}}}^{\theta_{\textnormal{max}}} \frac{\dd\theta}{\lVert  8\theta \rVert_{2\pi}} = \frac{1}{4p^{\frac n 2}\lVert 8n\theta_{d,p}\rVert_{2\pi}}\left(1+O\left(p^{-\frac{n}{2}}\lVert  8n\theta_{d,p} \rVert_{2\pi}^{-1}\right)\right).\]
This is our worst error term under the assumption that $p^{-\frac{n}{2}}\lVert  8n\theta_{d,p} \rVert_{2\pi}^{-1}\rightarrow 0$.
Noting that $\lVert 8\theta \rVert_{2\pi}^{-1} \geq \tfrac{1}{\pi}$ for all $\theta \in [\theta_{\textnormal{min}},\theta_{\textnormal{max}}]$, we further observe that
\begin{equation}
 \int_{\theta_{\textnormal{min}}}^{\theta_{\textnormal{max}}} \frac{\dd\theta}{\lVert  8\theta \rVert_{2\pi}} \geq \frac{1}{\pi}\int_{\theta_{\textnormal{min}}}^{\theta_{\textnormal{max}}} \dd\theta = \frac{1}{4\pi p^{\frac{n}{2}}},
\end{equation}
from which we may conclude \eqref{theta_bound_careful_constant}, upon noting again that $\lVert  8n\theta_{d,p} \rVert_{2\pi}^{-1} = O(1)$.

 It follows that, in each case,
\begin{align*}
\iint_{S_{p^n}} \frac{1}{\lVert  8\theta \rVert_{2\pi}} \frac{\log{r}}{r}r \dd r \dd\theta&=\left(\frac{n \log p}{4}+O\left(\frac{n \log p}{p^{\frac{n}{2}}}\right)\right)\frac{1}{4p^{\frac{n}{2}}\lVert 8n\theta_{d,p}\rVert_{2\pi}}\left(1+O\left(p^{-\frac{n}{2}}\lVert  8n\theta_{d,p} \rVert_{2\pi}^{-1}\right)\right)\\ 
&= A(S_{p^{n}})\frac{1}{2}\frac{n\log p}{p^{\frac{n}{2}}\lVert  8n\theta_{d,p} \rVert_{2\pi}}\left(1+O\left(\frac{1}{p^{\frac{n}{2}}\lVert 8n\theta_{d,p}\rVert_{2\pi}}\right)\right).
\end{align*}
In particular, by \eqref{r_bound_careful_constant} and 
\eqref{theta_bound_careful_constant}, we conclude that
\[\frac{n\log p}{p^{\frac{n}{2}}\lVert  8n\theta_{d,p} \rVert_{2\pi}} \ll \frac{1}{A(S_{p^{n}})}\iint_{S_{p^n}} \frac{1}{\lVert  8\theta \rVert_{2\pi}} \frac{\log{r}}{r}r \dd r \dd\theta,\]
as desired.
\end{proof}

\begin{prop}\label{lemma-split-primes supp1}
Let $d$ be an odd square-free integer,  $\alpha \in \Z/8\Z$ and $\phi$ be an even Schwartz function such that $\text{supp}(\widehat\phi) \subseteq (-\nu, \nu)$. 
Then as $K \rightarrow \infty$, we have
$$\frac{8}{K} \sum_{\substack{1\leq k \leq K \\ k \equiv \alpha \bmod8}} U_{\textnormal{split}}(\phi,d,k) = O_{\nu} (d^{\nu}K^{\nu -1}(\log (dK))^{3}).$$
\end{prop}
\begin{proof}
	Noting that $\frac1{\pi} \lVert 2x \rVert_{2 \pi}  \leq \lvert\sin x\rvert$, it follows from
 \eqref{before-bounding-sin}  and Lemma \ref{eq:area_approx}, that

\begin{align}
\frac{8}{K}  \sum_{\substack{1 \leq k \leq K\\k \equiv \alpha \bmod 8}} U_{\text{split}}(\phi,d, k) &\ll \frac{1}{K}\sum_{\substack{p^{n} \leq (KM_{d,\alpha})^{2\nu}\\ p \equiv 1 \bmod 4}} \frac{\log{p}}{{p^{n/2}}\lvert\sin (4n \theta_{d,p})\rvert}\nonumber\\
&\label{Usplit support 1 bound} \ll \frac{1}{K}\sum_{n=1}^{\frac{2 \nu  \log(KM_{d,\alpha})}{\log 5}} \frac{1}{n}\sum_{\substack{p \equiv 1 \bmod 4\\p^{n} \leq (KM_{d,\alpha})^{2\nu}
}}\iint_{S_{p^{n}}} \frac{ \log{r} }{ \lVert 8 \theta \rVert_{2\pi}}  \,\dd rd\theta.
\end{align}

By \eqref{def_Spn}, we note that $|z_{d,p}^{n}-z| < 1/2$ for any $z \in S_{p^{n}}$, and thus since $z^{n}_{d,p} \in \Z[i]$, it follows that for $p \neq q$, 
$S_{p^{n}} \cap S_{q^{n}} = \emptyset$. 
By \eqref{theta_away_2pi}, we find that for any $re^{i \theta} \in S_{p^{n}}$, 
\[\frac{1}{2r} \leq  \frac{1}{2p^{\frac n 2}-\tfrac12} \leq  p^{-\frac{n}{2}} \leq \lVert 8\theta \rVert_{2\pi}.\]
It follows that for any $n \geq 1$,
\[\bigcup_{\substack{p \equiv 1 \bmod 4\\p^{n} \leq (KM_{d,\alpha})^{2\nu}
}}S_{p^n} \subseteq R_{\nu} := \left \{  z = re^{i\theta} \in \C^{\times} \;:\;    \lVert 8 \theta\rVert_{2\pi} \geq \frac{1}{2r} \;\text{and}\;  1  \leq r \leq (KM_{d,\alpha})^{\nu} +\frac{1}{4}\right \}.\]
Thus, as $\theta \mapsto\Vert 8 \theta \rVert_{2\pi}$ is $\frac{\pi}{4}$-periodic, we have
\begin{multline*}
\sum_{\substack{p \equiv 1 \bmod 4\\p^{n} \leq (KM_{d,\alpha})^{2\nu}
}}\iint_{S_{p^{n}}} \frac{1}{ \lVert 8 \theta \rVert_{2\pi}}  \, \log{r} \dd r\dd\theta\ll \iint_{R_{\nu}} \frac{1}{\lVert 8 \theta \rVert_{2\pi}}\log{r} \, \dd r \dd\theta\\
\ll \int_{1}^{(KM_{d,\alpha})^{\nu}+\frac 1 4}  \int_{\frac{1}{16r}}^{\frac{\pi}{4} - \frac{1}{16r}}\frac{\dd\theta}{\Vert 8 \theta \rVert_{2\pi}} 
\log{r} \dd r
\ll \int_{1}^{(KM_{d,\alpha})^{\nu}+\frac 1 4} \log r\left( \int_{\frac{1}{16r}}^{\frac{\pi}{8}} \frac{\dd\theta}{8 \theta}+\int_{\frac{\pi}{8}}^{\frac{\pi}{4} - \frac{1}{16r}}\frac{\dd\theta}{2\pi - 8\theta} \right)  \dd r\\
\ll \int_{1}^{(KM_{d,\alpha})^{\nu}+\frac 1 4} \log r\left(\log \frac{\pi}{8} + \log 16r \right)  \dd r\ll_{\nu} (KM_{d,\alpha})^{\nu}(\log (KM_{d,\alpha}))^{2}.
\end{multline*}
Thus by \eqref{Usplit support 1 bound} we conclude that
\begin{align*}
\frac{8}{K}\sum_{\substack{1 \leq k \leq K \\ k \equiv \alpha \bmod 8}} U_{\textnormal{split}}(\phi,d, k)&\ll_{\nu} \sum_{n=1}^{\frac{2 \nu  \log (KM_{d,\alpha})}{\log 5}} \frac{1}{n} K^{-1}(KM_{d,\alpha})^{\nu}(\log (KM_{d,\alpha}))^{2} \\
&\ll_{\nu} d^{\nu}K^{\nu-1}(\log (KM_{d,\alpha}))^2 \log \log (KM_{d,\alpha}).
\end{align*}
Finally, by \eqref{conductor} and \eqref{def_M} we note that $M_{d,\alpha} \ll d$, from which the desired result now follows.
\end{proof}
Combining Proposition \ref{inert-primes-even} and Proposition \ref{lemma-split-primes supp1}, this proves Theorem \ref{thm-main}.

\section{Non-vanishing results} \label{section-non-vanishing}

In this section we prove Corollary~\ref{non-vanishing}.
Fix $0< \nu <1$ and consider the usual smooth test function
$$\phi(x) = \phi_\nu(x) := \left( \frac{\sin{(\pi \nu x)}}{\pi \nu x} \right)^2.$$
We note that $\phi_\nu(0) = 1$, $\phi_\nu(x) \geq 0$ for all $x \in \mathbb R$,  
$$
\widehat\phi_\nu(t) =  \begin{cases}  \frac{\nu-|t|}{\nu^2} & \text{if $|t| < \nu$} \\    0 & \text{otherwise,} \end{cases} 
$$
and that $\text{supp}(\widehat \phi) = [-\nu, \nu]$. 
As $\phi_\nu(x)$ is not a Schwartz function, we will approximate it by a family of even Schwartz functions $\{\phi_{m}\}_{m \geq 1}$.  In particular, we will assume that  $\{\phi_{m}\}_{m \geq 1}$ well-approximates $\phi_{\nu}$ in the $L^{1}$ norm, and that $\mathrm{supp}(\widehat{\phi}_{m}) \subseteq (-1,-1)$.  For simplicity, we will moreover assume that $\phi_{m}\geq \phi_{\nu}$ for all $m \geq 1$.\\
Assume the Riemann Hypothesis holds for $L(s, \xi_{d,k})$ (i.e the non-trivial zeros $\tfrac12 + i\gamma$ satisfy $\gamma \in \mathbb R$) for all $k\geq 1$.  We then have, for $\alpha$ even, and every $m \geq 1$, that
\begin{align} \label{for-proportion-11}
 \frac 8 {K} \sum_{\substack{1 \leq k \leq K\\ k \equiv \alpha \bmod 8 \\ L( \tfrac12, \xi_{d,k}) = 0}}  1  &\leq
 \frac 8 {K} \sum_{\substack{1 \leq k \leq K\\ k \equiv \alpha \bmod 8}}  \frac{\mbox{ord}_{s=\tfrac12} L(s, \xi_{d,k})}{2} \leq   \frac{4}{K} \sum_{\substack{1 \leq k \leq K\\\text{$k \equiv \alpha \bmod 8$}}} \cD_{K}(\phi_m, \xi_{d, k})  
\end{align}
where we have used the fact that for $\alpha$ even, the order of vanishing of $L(s, \xi_{d,k})$ at $s=\tfrac12$ is even, since $W(\xi_{d,k})=1$ by Lemma \ref{lemma-sign}.  Thus, for all $m \geq 1$ and large $K$, we find that
\begin{align*}
\frac{8}{K} \; \# \left\{ 1 \leq k \leq K, \, k \equiv \alpha \bmod 8 \;:\; L( \tfrac12, \xi_{d,k}) \neq 0 \right\}\geq 1 - \frac{4}{K} \sum_{\substack{1 \leq k \leq K\\\text{$k \equiv \alpha \bmod 8$}}} \cD_{K}(\phi_m, \xi_{d,k}) + O\left(\frac{1}{K}\right).
\end{align*}
Applying Theorem \ref{thm-main}, for any $\nu < 1$, we compute
\begin{align*}
\lim_{m \rightarrow \infty}\lim_{K \rightarrow \infty}  \frac{8}{K} \sum_{\substack{1 \leq k \leq K\\\text{$k \equiv \alpha \bmod 8$}}} \cD_{K}(\phi_m, \xi_{d,k})&= \widehat{\phi}_\nu(0) - \frac{1}{2} \int_{-\nu}^\nu \left( \frac{1}{\nu} - \frac{|t|}{\nu^2} \right) \;\dd t =  \frac{1}{\nu}  - \frac12,
\end{align*}
so that 
\[\liminf_{K \rightarrow \infty} \frac{8}{K} \; \# \left\{ 1 \leq k \leq K, \, k \equiv \alpha \bmod 8 \;:\; L( \tfrac12, \xi_{d,k}) \neq 0 \right\} \geq 75\%.\]

Suppose now that $\alpha$ is odd and $\alpha \in S_{+}(d)$, i.e. for $k \equiv \alpha \bmod 8$, we have $W(\xi_{d, k})=1$, which implies that $L(s, \xi_{d, k})$ vanishes with even order at $\tfrac12$. Working as above under the GRH for $L(s, \xi_{d,k})$, we have 
\begin{align} \label{for-proportion-1}
 \frac 8 {K} \sum_{\substack{1 \leq k \leq K\\ k \equiv \alpha \bmod 8 \\ L( \tfrac12, \xi_{d,k}) = 0}}  1  &\leq
 \frac 8 {K} \sum_{\substack{1 \leq k \leq K\\ k \equiv \alpha \bmod 8}}  \frac{\mbox{ord}_{s=\tfrac12} L(s, \xi_{d,k})}{2} \leq   \frac{4}{K} \sum_{\substack{1 \leq k \leq K\\\text{$k \equiv \alpha \bmod 8$}}} \cD_{K}(\phi_m, \xi_{d,k})  
\end{align}
but in this case, we have orthogonal symmetry from Theorem \ref{thm-main}. We compute
\begin{align} \label{orthogonal}
\lim_{m \rightarrow \infty}\lim_{K \rightarrow \infty}  \frac{8}{K} \sum_{\substack{1 \leq k \leq K\\\text{$k \equiv \alpha \bmod 8$}}} \cD_{K}(\phi_m, \xi_{d,k}) =   \widehat{\phi}_\nu(0) + \frac{1}{2} \int_{-\nu}^\nu \left( \frac{1}{\nu} - \frac{|t|}{\nu^2} \right) \;\dd t
= \frac{1}{\nu} + \frac12
\end{align}
and replacing in \eqref{for-proportion-1} with $\nu<1$, we get 
\begin{align*}
\liminf_{K \rightarrow \infty} \frac{8}{K} \; \# \left\{ 1 \leq k \leq K, \,k \equiv \alpha \bmod 8 \;:\; L( \tfrac12, \xi_{d,k}) \neq 0 \right\}  \geq 25 \%.
\end{align*}

Finally, suppose that $\alpha$ is odd and $\alpha \in S_{-}(d)$, i.e. that for $k \equiv \alpha \bmod 8$, we have $W(\xi_{d, k})=-1$.  This implies that $L(s, \xi_{d, k})$ vanishes with odd order $\geq 1$, and working as above under the GRH for $L(s, \xi_{d,k})$,
\begin{align} \label{for-proportion}
 \frac 8 {K} \sum_{\substack{1 \leq k \leq K\\ k \equiv \alpha \bmod 8 \\ \mbox{ord}_{s=\frac12} L( s, \xi_{d,k}) > 1}}  1  &\leq
 \frac 8 {K} \sum_{\substack{1 \leq k \leq K\\ k \equiv \alpha \bmod 8}}  \frac{\mbox{ord}_{s=\tfrac12} L(s, \xi_{d,k})-1}{2} \leq   \frac{4}{K} \sum_{\substack{1 \leq k \leq K\\\text{$k \equiv \alpha \bmod 8$}}} \cD_{K}(\phi_m, \xi_{d,k})  - \frac12 + O(K^{-1}),
\end{align}
and we again have orthogonal symmetry.  As in \eqref{orthogonal}, we compute
$$
\lim_{m \rightarrow \infty}\lim_{K \rightarrow \infty}  \frac{4}{K} \sum_{\substack{1 \leq k \leq K\\\text{$k \equiv \alpha \bmod 8$}}} \cD_{K}(\phi_m, \xi_{d,k})  - \frac{1}{2} = \frac{1}{2\nu} - \frac14$$
and replacing in \eqref{for-proportion} with $\nu<1$, we get 
\begin{align*}
\liminf_{K \rightarrow \infty} \frac{8}{K} \left\{ 1 \leq k \leq K, \,k \equiv \alpha \bmod 8 \;:\; \mbox{ord}_{s=\tfrac12} L( s, \xi_{d,k}) = 1  \right\} \geq 75 \%.
\end{align*}
This completes the proof. \qed

\begin{rem}
As remarked in \cite{ILS}, the function $\phi_1(x) = \left( \frac{\sin{(\pi  x)}}{\pi  x} \right)^2$ is optimal among all functions $\phi \in S_1$ where
$$S_1 := \left\{ \phi \in L^{1}(\R), \phi \geq 0, \phi(0)=1, \;\text{and supp}(\widehat{\phi}) \subset [-1, 1] \right\}$$
 for both symplectic and orthogonal symmetry.  In other words,
\begin{align*}
\inf_{\phi \in S_1} \int_{-\infty}^{\infty} \phi(x) W_G(x) \dd x
\end{align*}
is attained for $\phi_1$ when $W_G(x) = 1 \pm \frac{1}{2} \frac{\sin(2 \pi x)}{2 \pi x}.$ Indeed, noting that $\widehat\phi_{1} = g\star g$ where $$g : x \mapsto \begin{cases}
	1 &\text { if } x \in [-\tfrac12,\tfrac{1}{2}], \\
	0 & \text{ otherwise,} 
\end{cases}$$
 the proof of \cite[Cor. A.2]{ILS} demonstrates that $\phi_1$ is optimal if and only if $g$
is an eigenvector for the operator $f  \mapsto \int_{\mathbb R} f(y) (\widehat{W}_G - \delta_{0})(x-y) \dd y $ on $L^1([-\tfrac12,\tfrac12])$.
By \eqref{W_hat_options}, and taking $\eta$ as in~\eqref{def eta}, this is equivalent to the condition that the maps
$$x \mapsto\pm \frac{1}{2} \int_{-\frac12}^\frac12 \eta(x-y) \dd y$$ 
are constant on the interval $[-\tfrac{1}{2},\tfrac12]$, which is indeed the case.
\end{rem}

\appendix
\section{Evaluating the $c_j$ constants numerically}\label{appendix computation}
In this appendix, we numerically approximate the values of $c_j(1,0)$ using the formulas of Proposition~\ref{inert-primes-even} for the cases $j=0$ and $j=2$, and explain the procedure for doing so for arbitrary $j  \geq 0$.  
We  write
\[c_{j}:= c_j(1,0) = A_{j}+B_{j}-T_{j},\]
where
\[A_{j}:= \frac{(\log 2)^{j+1}}{j!}\mathrm{Li}_{-j}(\tfrac12)-\frac{2}{j!}\left(\frac{\log 2}{2}\right)^{j+1}\textnormal{Li}_{-j}\left(\frac{1}{\sqrt{2}}\right)-\delta_{0}(j)\log 2\pi\]
\begin{align}\label{c_j_0(4)}
\begin{split}
B_{j} :=&\frac{(-1)^j}{j!}\left(\left(\frac{f'}{f}\right)^{(j)}(1)-\left(\frac{L'}{L}\right)^{(j)}(1,\chi_{4})\right)
 \end{split}
\end{align}
where $f(s):=(s-1)\zeta(s)$, and 
\[T_{j}:=(-1)^{j+1} \frac{2^{j+1}}{j!} \left(\frac{\zeta'_{4,3}}{\zeta_{4,3}}\right)^{(j)}(2).\]
We seek to re-express the constants $B_{j}$ and $T_{j}$ in a more computationally-friendly form.  To begin, we find that by the Faà di Bruno formula \cite{FaaDiBruno} for higher derivatives,
\begin{align*}
\left(\frac{f'}{f}\right)^{(j)}(1) = (\log  f)^{(j+1)}(s)\bigg \vert_{s=1} = \sum_{k=1}^{j+1}\log^{(k)}(f(1))\cdot B_{j+1,k}\left(f^{(1)}(1),f^{(2)}(1),\dots,f^{(j+2-k)}(1)\right),
\end{align*}
where the \textit{partial exponential Bell polynomials} are defined as 
\begin{equation}\label{Bell_eq}
B_{n, k} (x_1,\dots, x_{n-k+1}) := \sum_{\substack{j_1, \dots, j_{n-k+1} \geq 0\\
j_1 + \dots  + j_{n-k+1}=k\\
j_1 + 2 j_2 + \dots + (n-k+1) j_{n-k+1} = n
}} \frac{n!}{j_1! \dots  j_{n-k+1}!}
\left( \frac{x_1}{1} \right)^{j_1} \dots \left( \frac{x_{n-k+1}}{(n-k+1)!} \right)^{j_{n-k+1}}.
\end{equation}
By comparing terms in \eqref{zeta_holomorphic}, we note that for $\ell \geq 1$,
\[f^{(\ell)}(1) = \ell(-1)^{\ell-1}\gamma_{\ell-1},\]
where $\gamma_\ell$ are the Stieltjes constants.  Moreover, for $k\geq1$, since $f(1)=1$, we have
\[\log^{(k)}(f(1)) = (-1)^{k+1}(k-1)!.\]
We therefore conclude that
\begin{align*}
\left(\frac{f'}{f}\right)^{(j)}(1) = \sum_{k=1}^{j+1}(-1)^{k+1}(k-1)!\cdot B_{j+1,k}\left(\gamma_{0},-2 \gamma_{1},\dots,
(j+2-k)(-1)^{j+1-k}\gamma_{j+1-k}
\right).
\end{align*}
This formula may then be coded into Mathematica as follows:

\begin{lstlisting}
Bellzeta[j_,k_]:=BellY[j+1,k, 
  Table[(-1)^i*(i+1)*StieltjesGamma[i], {i,0,j+1-k}]]
  (*Defining the Bell polynomial $B_{j+1,k}\left(\gamma_{0},\dots,
(j+2-k)(-1)^{j+1-k}\gamma_{j+1-k}
\right)$*);

Bzeta[j_]:=Sum[(-1)^{k+1}(k-1)! Bellzeta[j,k], {k,1,j+1}] 
(*Summing over the relevant Bell polynomials*);
\end{lstlisting}

Next, let $\chi$ be a Dirichlet character modulo $q$.  For $n \geq 1$, we note that
\begin{equation}\label{L_gamma_constants}
L^{(n)}(1,\chi)= (-1)^{n}\gamma_{n}(\chi),
\end{equation}
where $\gamma_{n}(\chi)$ is known as the $n^{th}$ Laurent--Stieltjes constant for $L(s,\chi)$.  As in~\cite{EddinThesis}, we find that $\gamma_{n}(\chi)$ may be expressed as
\[\gamma_{n}(\chi) = \sum_{a=1}^{q}\chi(a)\gamma_{n}(a,q), \quad \quad \gamma_{n}(a,q):=\lim_{x \rightarrow \infty}\left(\sum_{\substack{0 < m \leq x \\m \equiv a \bmod q}}\frac{(\log m)^n}{m}-\frac{(\log x)^{n+1}}{q(n+1)}\right),\]
where $\gamma_{n}(a,q)$ are sometimes referred to as \textit{generalized Euler constants} (for arithmetical progressions).  In particular, we find that for $n \geq 1$,
\begin{equation}\label{gamma_chi}
\gamma_{n}(\chi_{4}) = \gamma_{n}(1,4)-\gamma_{n}(3,4),
\end{equation}
and moreover that $\gamma_{0}(\chi_{4}) = L(1,\chi_{4})=\tfrac{\pi}{4}$.  For small values of $n$ ($n \leq 20$), $\gamma_{n}(1,4)$ and $\gamma_{n}(3,4)$ have been explicitly computed in \cite{Dilcher}:
\begin{center}
	\vspace{3pt}
\renewcommand{\arraystretch}{1.2}
\begin{tabular}{ |c|c| } 
 \hline
$\gamma_{1}(1,4)$ & -0.154621845705\\
$\gamma_{2}(1,4)$ & -0.095836601153\\
$\gamma_{3}(1,4)$ & -0.049281458556\\
 \hline
\end{tabular}\quad
\begin{tabular}{ |c|c| } 
 \hline
$\gamma_{1}(3,4)$ & 0.038279471092\\
$\gamma_{2}(3,4)$ & 0.058305123277\\
$\gamma_{3}(3,4)$ & 0.045601400650\\
 \hline
\end{tabular}\quad
\begin{tabular}{ |c|c| } 
 \hline
$\gamma_{1}(\chi_{4})$ & -0.19290131679\\
$\gamma_{2}(\chi_{4})$ & -0.15414172443\\
$\gamma_{3}(\chi_{4})$ & -0.0948828592\\
 \hline
\end{tabular}
\captionof{table}{Values of $\gamma_{n}(1,4)$,  $\gamma_{n}(3,4)$, and  $\gamma_{n}(\chi_{4})$.}
\end{center}
Again by the Faà di Bruno formula for higher derivatives, it follows similarly to as above that
\begin{align*}
\left(\frac{L'}{L}\right)^{(j)}(1,\chi_{4})  &= \sum_{k=1}^{j+1}\log^{(k)}(L(1,\chi_{4}))\cdot B_{j+1,k}\left(L^{(1)}(1,\chi_{4}),L^{(2)}(1,\chi_{4}),\dots,L^{(j+2-k)}(1,\chi_{4})\right)\\
&= \sum_{k=1}^{j+1} (-1)^{k+1}(k-1)!\left(\frac{\pi}{4}\right)^{-k} B_{j+1,k}\left(-\gamma_{1}(\chi_{4}),\dots,
(-1)^{j+2-k}\gamma_{j+2-k}(\chi_{4})\right).
\end{align*}
Each such contribution may then be computed in Mathematica as follows:

%
\begin{lstlisting}
Lconstants:={-0.19290131679,-0.15414172443,-0.0948828592}
 (*Values of $\gamma_{j}(\chi_{4})$ extracted from previously published computations*)
 
BellL[j_,k_]:=BellY[j+1,k,Table[(-1)^{i}*Lconstants[[i]],{i,1,j+2-k}]]
(*Defining the Bell polynomial $B_{j+1,k}\left(-\gamma_{1}(\chi_{4}),\dots,
(-1)^{j+2-k}\gamma_{j+2-k}(\chi_{4})\right)$*)

BL[j_]:=Sum[(-1)^{k+1}(k-1)!(Pi/4)^{-k} BellL[j,k],{k,1,j+1}]
(*Summing over the relevant Bell polynomials for the L-function factor*)

B[j_]:=(-1)^{j}/(j)!(Bzeta[j]-BL[j])
\end{lstlisting}

\begin{figure}[h]
\begin{center}
\renewcommand{\arraystretch}{1.2}
\begin{tabular}{|c|c|c|c|} 
 \hline
 $j$ & $(f'/f)^{(j)}(1)$ & $(L'/L)^{(j)}(1, \chi_{4})$ & $B_{j}$\\
 \hline
$j=0$ & 0.57721566 & 0.24560958 & 0.33160608\\
$j=2$ & 0.10337726&  0.29505047 & $-$0.0958366\\
 \hline
\end{tabular}
\captionof{table}{Values of $(f'/f)^{(j)}(1)$, $(L'/L)^{(j)}(1,\chi_{4})$, and $B_{j}$, for $j=0$ and $j=2$.}
\end{center}
\end{figure}

Next, we note that since
\[\left(\frac{\zeta'_{4,3}}{\zeta_{4,3}}\right)^{(j)}(s) = \sum_{p \equiv 3 \bmod4}\sum_{n=1}^{\infty}\frac{(-\log p)^{j+1}n^j}{p^{ns}} = \sum_{p \equiv 3 \bmod4}(-\log p)^{j+1}\mathrm{Li}_{-j}\left(\frac1{p^{s}}\right),\]
we have that
\[T_{j} = (-1)^{j+1} \frac{2^{j+1}}{j!} \left(\frac{\zeta'_{4,3}}{\zeta_{4,3}}\right)^{(j)}(2) = \frac{2^{j+1}}{j!}\sum_{p \equiv 3 \bmod4} (\log p)^{j+1} \mathrm{Li}_{-j}\left(\frac1{p^2}\right).\]
To compute the inert prime contribution, we write
\[T_{j} = T_{j}[x]+R_{j}[x],\]
where
\[T_{j}[x]:=\frac{2^{j+1}}{j!}\sum_{\substack{p\equiv 3 \bmod 4\\ p < x}}(\log p)^{j+1}\textnormal{Li}_{-j}(p^{-2})\]
and
\begin{align*}
R_{j}[x]:=&\frac{2^{j+1}}{j!}\sum_{\substack{p \equiv 3 \bmod 4\\p \geq x}}\sum_{n=1}^{\infty}\frac{(\log p)^{j+1}n^{j}}{p^{2n}} \leq \frac{2^{j+1}}{j!}\left(\frac{(\log x)^{j+1}}{x^2}+\sum_{k=0}^{j+1}\frac{(\log x)^{k}}{x}\frac{(j+1)!}{k!}\right),
\end{align*}
so long as $x$ is sufficiently large such that the map $t \mapsto (\log t)^{j+1}/t^2$ is decreasing in the range $[x,\infty)$.
The resulting contribution may then be computed by implementing the following code:

\begin{lstlisting}
X=1000000;  primes3 := Select[Range[3, X, 4],PrimeQ]
 (*List of primes 3 mod 4 up to X*);

T[j_]:=(2^{j+1}/j!)Sum[N[Log[primes3[[i]]]^{j+1}] 
N[PolyLog[-j,(1/primes3[[i]]^2)]],{i,1,Length[primes3]}]
(*computes $T_{j}[X]$*)
\end{lstlisting}

Plugging in the value $x = 10^{6}$ yields the following values:

\begin{figure}[h]
\begin{center}
\renewcommand{\arraystretch}{1.2}
\begin{tabular}{ |c|c|c|c|c|} 
 \hline
$j$ & $T_{j}[10^6]$ & Bound for $R_{j}[10^{6}]$& $A_{j}$ &$c_{j}$ \\
\hline
$j=0$ & 0.45747 & $0.00003$ & $-2.81814$ & $-2.9440$\\
$j=2$ & 3.93 & $0.014$ & $-1.00081$ & $-5.0$\\
 \hline
\end{tabular}
\end{center}
\captionof{table}{Numerical approximation of $c_{0}$ and $c_{2}$.}
\end{figure}
\begin{rem}\label{Rem Appendix}
	Numerically, we observe  that as~$j$ grows, $B_j$ approaches~$0$, $A_j$ approaches~$-1$, and $T_j$ tends towards infinity.
	This suggests that $c_j <0$  for all~$j\geq 0$, and moreover that the main contribution  to $c_j$ is given by the inert primes, in particular the $(j+1)^{\text{th}}$ logarithmic derivative of $\zeta_{4,3}$ at $2$.  Finally, we further note that the higher order logarithmic derivative of $\zeta_{4,3}(s)$ at $s=2$ may likely be computed more efficiently using methods adapted from \cite{LM} (see, in particular, Remark 1 therein).
\end{rem}


\begin{thebibliography}{1}

\bibitem{WaxRatio}
R.\ Chen, Y.\ Kim, J.\ Lichtman, S.\ J.\ Miller, A.\ Shubina, S.\ Sweitzer, E.\ Waxman, E.\ Winsor, J.\ Yang, A Refined Conjecture for the Variance of Gaussian Primes across Sectors, Exp. Math.  (2023), no.\ 32(1), 33--53. 

\bibitem{D-G} C.\ David, A.\ G\"{u}lo\u{g}lu,  
One-level density and non-vanishing for cubic $L$-functions over the Eisenstein field,
Int. Math. Res. Not. (2022), no.\ 23, 18833–18873.


\bibitem{DPR} S.\ Drappeau, K.\ Pratt, M.\ Radziwi\l\l,
One-level density estimates for {D}irichlet {$L$}-functions with extended support,
Algebra Number Theory 17 (2023), no.\ 4, 805–830.

\bibitem{De2023+} L.\ Devin, Discrepancies in the distribution of Gaussian primes. Math. Z. \textbf{312} (2026), no. 31.

\bibitem{DFS1} L.\ Devin, D.\ Fiorilli, A.\ S\"odergren, Low-lying zeros in families of holomorphic cusp forms: the weight aspect, Quart.\ J.\ Math.\ {73} (2022), no.\ 4, 1403--1426. 

\bibitem{Dilcher}
K.\ Dilcher, Supplement to Generalized Euler Constants for Arithmetical Progressions, Math.\ Comp.\ 59 (1992), no.\ 199, S21--S24. 

\bibitem{EddinThesis}
  S.\ S.\ Eddin, On two problems concerning the Laurent-Stieltjes coefficients of Dirichlet $L$-series, Ph.D.\ thesis, University of Lille 1, France, 2013.
  
\bibitem{FaaDiBruno}
F.\ Fa\`a di Bruno, Note sur une nouvelle formule de calcul diff\'erentiel, Q.\ J.\ Pure Appl.\ Math.\ {1} (1857), 359--360.
  
  \bibitem{FPS1} D.\ Fiorilli, J.\ Parks, A.\ S\"odergren, Low-lying zeros of elliptic curve $L$-functions: Beyond the Ratios Conjecture, Math.\ Proc.\ Cambridge Philos.\ Soc.\ {160} (2016), no.\ 2, 315--351.
  
  \bibitem{FPS2} D.\ Fiorilli, J.\ Parks, A.\ S\"odergren, Low-lying zeros of quadratic Dirichlet $L$-functions: Lower order terms for extended support, Compos.\ Math.\ {153} (2017), no.\ 6, 1196--1216.

\bibitem{G-Z} 
P.\ Gao, L.\ Zhao, 
One-level density of low-lying zeros of quadratic Hecke $L$-functions of imaginary quadratic number fields,
J.\ Aust.\ Math.\ Soc.\ 112 (2022), no.\ 2, 170--192.

\bibitem{HB} D.\ R.\ Heath-Brown,
The average analytic rank of elliptic curves,
Duke Math.\ J.\ 122 (2004), no.\ 3, 591--623.


\bibitem{Hecke1920}
E.\ Hecke, Eine neue Art von Zetafunktionen und ihre Beziehungen zur Verteilung
der Primzahlen, I, Math.\ Z.\  1 (1918), 357--376; II, Math.\ Z.\ 6 (1920), 11--51.

\bibitem{HarmanLewis}
G.\ Harman, P.\ Lewis, Gaussian primes in narrow sectors, Mathematika {48} (2001), no.\ 1-2, 119--135.

\bibitem{Holm2023}
K.\ Holm, The 1-Level Density For Zeros of Hecke L-Functions of Imaginary Quadratic Number Fields of Class Number 1, 
J.\ Théor.\ Nombres Bordeaux {37} (2025), no.\ 1, 237--283.

\bibitem{HR} C.\ P.\ Hughes, Z.\ Rudnick, Linear statistics of low-lying zeros of $L$-functions,  Q.\ J.\ Math.\ {54} (2003), no.\ 3, 309--333.

\bibitem{IreRos}
K.\ Ireland, M.\ Rosen, \textit{A classical introduction to modern number theory},
Second edition,
Graduate Texts in Mathematics, 84,
Springer-Verlag, New York, 1990.

\bibitem{Iwaniec} H.\ Iwaniec, \textit{Topics in classical automorphic forms},
American Mathematical Society
Graduate Texts in Mathematics, 17, American Mathematical Society, Providence, RI, 1997.



\bibitem{IwanKow}
H.\ Iwaniec, E.\ Kowalski, \textit{Analytic number theory}, American Mathematical Society
Colloquium Publications, 53, American Mathematical Society, Providence, RI, 2004.

\bibitem{ILS} H.\ Iwaniec, W.\ Luo, P.\ Sarnak,
Low lying zeros of families of $L$-functions,
Inst.\ Hautes Études Sci.\ Publ.\ Math.\ (2000), no.\ 91, 55--131.

\bibitem{KSbook}
N.\ Katz, P.\ Sarnak, \textit{Random Matrices, Frobenius Eigenvalues and Monodromy}, AMS Colloq.\ Publ.\ 45 (1999). 

\bibitem{KS} N.\  Katz, P.\ Sarnak, Zeroes of zeta functions and symmetry, Bull.\ Amer.\ Math.\ Soc.\ (N.S.) {36} (1999), no.\ 1, 1--26.

\bibitem{LM}
A.\ Languasco, P.\ Moree, Euler constants from primes in arithmetic progression, Math. Comp. 95 (2026), 363–387

\bibitem{Lemmermeyer}
F.\ Lemmermeyer, 
\textit{Reciprocity laws : 
From Euler to Eisenstein},
Springer Monogr. Math.
Springer-Verlag, Berlin, (2000).

\bibitem{M1} S.\ J.\ Miller, One- and two-level densities for rational families of elliptic curves: evidence for the underlying group symmetries, 
Compos.\ Math.\ {140} (2004), no.\ 4, 952--992.

\bibitem{Miller-et-al}  B.\ Mackall, S.\ J.\ Miller, C.\ Rapti, C.\ Turnage-Butterbaugh, K.\ Winsor, Some results in the theory of low-lying zeros of families of L-functions, Proceedings of Simons Symposia, \emph{Families of Automorphic Forms and the Trace Formula}, Springer-Verlag (2016), 435--476.

\bibitem{MV} H.\ L.\ Montgomery, R.\ C.\ Vaughan, \emph{Multiplicative number theory. I. Classical theory}, Cambridge Studies in Advanced Mathematics 97, Cambridge University Press, Cambridge, 2007.

\bibitem{Neukirch} J.\ Neukirch,
\textit{Algebraic number theory},
Translated from the 1992 German original and with a note by N.\ Schappacher. With a foreword by G.\ Harder
Grundlehren Math.\ Wiss.\ 322 
Springer-Verlag, Berlin, (1999). 


\bibitem{O-S} A.\ E.\ \"Ozl\"uk, C.\ Snyder, Small zeros of quadratic $L$-functions, Bull.\ Austral.\ Math.\ Soc.\ {47} (1993), no.\ 2, 307--319.

\bibitem{RicottaRoyer LowerOrder} G.\ Ricotta, E.\ Royer, Lower order terms for the one-level densities of symmetric power $L$-functions in the level aspect,
Acta Arith.\ {141} (2010), no.\ 2, 153--170.

\bibitem{RR} G.\ Ricotta, E.\ Royer, Statistics for low-lying zeros of symmetric power $L$-functions in the level aspect, Forum Math.\ {23} (2011), no.\ 5, 969--1028.

\bibitem{RudWax}
Z.\ Rudnick, E.\ Waxman, Angles of Gaussian Primes, Israel J.\ Math.\ (2019), no.\ 232, 159--199.

\bibitem{SST} P.\ Sarnak, S.\ W.\ Shin, N.\ Templier, Families of $L$-functions and their symmetry, Proceedings of Simons Symposia, \emph{Families of Automorphic Forms and the Trace Formula}, Springer-Verlag (2016), 531--578.




\bibitem{Wa2021}
E.\ Waxman, Lower Order Terms for the One Level Density of a Symplectic Family of Hecke L-Functions, J.\ Number Theory 221 (2021), 447--483.  

\bibitem{Young}  M.\ P.\ Young, Low-lying zeros of families of elliptic curves, J.\ Amer.\ Math.\ Soc.\ 19 (2006), no.\ 1, 205--250.




\end{thebibliography}
\end{document}